\documentclass[a4paper,12pt]{article}
\usepackage{amsthm,amsmath,amscd,amsfonts,mathtools,thmtools}
\usepackage{tikz-cd}
\tikzcdset{arrow style=math font}
\tikzcdset{every label/.append style = {font = \small}}
\usepackage{svg}
\usepackage{caption}
\usepackage{enumitem}
\setlist[enumerate,1]{label=\Roman*.}
\setlist[itemize,2]{label=$\centerdot$}
\setlist[itemize,3]{label=$\triangle$}
\setlist[enumerate]{noitemsep,topsep=1em,parsep=0pt}
\usepackage[charter]{mathdesign}
\usepackage{hyperref}

\usepackage[shortcuts]{extdash}

\usepackage{subcaption}
\numberwithin{equation}{section}
\usepackage[vmargin=2cm,hmargin=2cm,headheight=14.5pt,top=2cm,headsep=.5cm]{geometry}

\newtheoremstyle{ptheorem}{1em}{0em}{\itshape}{}{\bfseries}{.}{.5em}{\thmname{#1}\thmnumber{
		#2}\thmnote{ (\hspace{-1sp}{#3})}}

\theoremstyle{ptheorem}

\newtheorem{thm}{Theorem}[section]
\newtheorem{pro}[thm]{Proposition}
\newtheorem{lem}[thm]{Lemma}
\newtheorem{cor}[thm]{Corollary}

\newtheoremstyle{hdef}{1em}{0em}{}{}{\bfseries}{.}{.5em}{\thmname{#1}\thmnumber{
		#2}\thmnote{ (\hspace{-.01pt}{#3})}}
\theoremstyle{hdef}

\newtheorem{dfn}[thm]{Definition}
\newtheorem{rem}[thm]{Remark}

\newtheorem{exa}[thm]{Example}

\newcommand\restr[2]{{
 \left.\kern-\nulldelimiterspace 
 #1 
 \vphantom{\big|} 
 \right|_{#2} 
 }}

\newcommand{\cL}{{\mathcal L}}

\newcommand{\bC}{{\mathbb C}}

\newcommand{\bF}{{\mathbb F}}

\newcommand{\bN}{{\mathbb N}}

\newcommand{\bR}{{\mathbb R}}

\renewcommand{\a}{\alpha}
\renewcommand{\b}{\beta}

\renewcommand{\l}{\lambda}

\renewcommand{\phi}{\varphi}

\renewcommand{\leq}{\leqslant}
\renewcommand{\ge}{\geqslant}
\renewcommand{\geq}{\geqslant}

\newcommand{\n}{{n\in\bN}}

\newcommand{\Ra}{\Rightarrow}

\renewcommand{\(}{\left(}
\renewcommand{\)}{\right)}

\newcommand{\til}{\widetilde}

\newcommand{\bs}{\backslash}

\newcommand{\olb}[1]{%
	\vbox{\offinterlineskip\ialign{\hfil##\hfil\cr
			$\rotatebox[origin=c]{90}{$]$}$\cr\noalign{\kern-.45ex}{$#1$}\cr}}}

\newcommand{\noop}[1]{}

\renewcommand{\ss}{\subset}
\usepackage{stmaryrd}

\newcommand{\norm}[1]{\left\lVert#1\right\rVert}

\makeatletter
\newcommand*\bigcdot{\mathpalette\bigcdot@{.5}}
\newcommand*\bigcdot@[2]{\mathbin{\vcenter{\hbox{\scalebox{#2}{$\m@th#1\bullet$}}}}}
\makeatother

\begin{document}

\title{Stieltjes polynomial interpolation}
\author{V\'ictor Cora and F. Adri\'an F. Tojo}

\date{}

\author{Víctor Cora$^{*}$\\
	\normalsize e-mail: victor.cora.calvo@usc.es\\
	F. Adri\'an F. Tojo$^{*\dagger}$ \\
	\normalsize e-mail: fernandoadrian.fernandez@usc.es\\ \small \emph{$^{*}$Departamento de Estatística, Análise Matemática e Optimización},\\ \small \emph{Universidade de Santiago de Compostela, 15782, Santiago de Compostela, Spain.}
	\\\small \emph{$^{\dagger}$CITMAga, 15782, Santiago de Compostela, Spain}}

\date{}

\maketitle

\medbreak

\begin{abstract}
We show Lagrange and Hermite interpolation are possible using Stieltjes polynomials, linear combinations of iterated Stieltjes integrals of a constant function. We introduce divided differences via Newton interpolation and provide an explicit error formula of Peano kernel type via a new Taylor formula. Finally, we use the properties the space of Stieltjes polynomials has to recover two fundamental approximating properties: the uniqueness of the best uniform polynomial approximant and the Chebyshev equioscillation theorem.
\end{abstract}


\medbreak


\section{Introduction}

Polynomial interpolation is a fundamental tool not only in the design of numerical schemes but also in several areas of abstract analysis. In this work, we revisit polynomial interpolation in the Stieltjes setting, a framework that generalizes the theory of time scales, see \cite{RodrigoTheory}, and thereby, unifies continuous and discrete analysis.

We show that interpolation of Lagrange and Hermite type can be achieved using Stieltjes polynomials, defined as linear combinations of iterated Stieltjes integrals of a constant function. When the derivator is discontinuous, the space of Stieltjes polynomials is no longer closed under multiplication and therefore fails to form an algebra. Consequently, the usual approach to Lagrange and Hermite interpolation via the fundamental theorem of algebra is not available. In fact, we proceed in the opposite direction. Since Hermite interpolation is possible, we deduce that a Stieltjes polynomial of degree $n$ can have at most $n$ roots, counting multiplicities. Moreover, we show that interpolation of both types is possible at right hand limits of discontinuity points of the derivator, as well as at convex combinations of evaluation functionals.

We also introduce Newton polynomials and divided differences consistently in the Stieltjes setting. Furthermore, we establish a Taylor formula in Stieltjes calculus, which allows us to derive a Peano kernel type result. From this, we obtain integral formulas for the errors of interpolatory projections and divided differences.

Finally, we prove that the space of Stieltjes polynomials of degree at most $n$, denoted by $\operatorname{P}_n$, is a Chebyshev space. More precisely, we show that $\operatorname{P}_n$ is a  $g$\--Haar space (a Haar space in the Stieltjes setting) and use the classical Haar theorem to relate these notions. In fact, we prove that $\operatorname{P}_n$ is both a  $g$\--Haar space and a  $g$\--weak Chebyshev space (the notion of weak Chebyshev being defined following \cite{deutsch1980weak}). These results allow us to recover two fundamental theorems of approximation theory to the Stieltjes framework: the uniqueness of the best uniform polynomial approximant and the Chebyshev equioscillation theorem.

This work builds on \cite{Cora2023}, where Stieltjes polynomials and their main properties were introduced together with the class of Stieltjes analytic functions, and on \cite{weier}, where a version of the Weierstrass approximation theorem for Stieltjes calculus was proved in the case of a derivator with a finite amount of discontinuities.

Since a polynomial defined over a time scale is a particular case of a Stieltjes polynomial, see \cite[Remark~3.2]{Cora2023}, all these results generalize to time scales ---see \cite[Chapter 3.10]{BohnerPeterson2003} or \cite{BohnerGuseinov2007}, allowing for interpolation in dynamic calculus, which has been unexplored so far.

This work is organized as follows. In Section~\ref{sec:Preliminaries}, we introduce the main concepts of Stieltjes calculus including Stieltjes polynomials, and several important operators that will be used in later sections. In Section~\ref{sec:gdomain}, we extend the domain of a uniformly  $g$\--continuous function to $\mathbb{R}_+$, the set of real numbers with a copy of the set of discontinuities of the derivator. We endow $\mathbb{R}_+$ with an appropriate topology that allows continuous functions to be evaluated at right hand limits of points of discontinuity. In Section~\ref{sectionunifcont}, using the domain introduced in the previous section, we redefine the notion of a root of a uniformly  $g$\--continuous function, which allows us to recover Bolzano's and Rolle's theorems.

In Section~\ref{lagrangeinterpolsec}, we address the Lagrange interpolation problem and show that it can be solved using Stieltjes polynomials and the tools developed in the previous sections. Moreover, we show that Lagrange interpolation can be performed at nodes of $\mathbb{R}_+$, which include right hand limits of discontinuity points, or at convex combinations of evaluation functionals. In Section~\ref{sec:chebishev}, we exploit the interpolatory properties of Stieltjes polynomials to obtain both the uniqueness of the best uniform approximant and the Chebyshev equioscillation theorem.

In Section~\ref{sec:divdiff}, we introduce the notions of Newton polynomials and divided differences together with some of their properties. In Section~\ref{sec:taylor}, we establish a Taylor formula for Stieltjes calculus. In Section~\ref{sec:peano}, we use this Taylor formula to derive a Peano kernel theorem for Stieltjes calculus and obtain several error formulas. In Section~\ref{sec:gmoncenteredRmas}, we introduce the notion of a  $g$\--monomial centered at a right hand limit. This concept allows  $g$\--polynomials to be centered at points in $\mathbb{R}_+$ and plays a key role in the solvability of the Hermite interpolation problem. Finally, in Section~\ref{sec:hermite}, we introduce the Hermite* interpolation problem, a generalization of the classical Hermite interpolation problem, and prove that it admits a unique solution under suitable necessary and sufficient conditions.

\section{Preliminaries}\label{sec:Preliminaries}
 Let $X$ and $Y$ be two sets. We denote by $X^Y$ the set of mappings of $Y$ with values on $X$. Also,
 given two subsets $A,B\subset X$ of $X$ we will consider their \emph{symmetric difference} $A \,\triangle\, B: = (A\cup B)\setminus (A\cap B)$. We will write $\#X$ for the cardinal of a set.

\begin{dfn}
	\label{derdef}
	We say that a $g:\mathbb R\to \mathbb{R}$ left\--continuous non\--decreasing function is a \emph{derivator}. The letter $g$ will always denote a derivator.
\end{dfn}

 Derivators are identified, up to a constant, with Lebesgue\--Stieltjes measures, i.e., Borel measures over the reals that are finite on bounded sets. Denote $\mu_g$ the Lebesgue\--Stieltjes measure induced by $g$. We will refer to the integrals of Lebesgue\--Stieltjes measures as Lebesgue\--Stieltjes integrals. For more details, see \cite[Chapter 1]{Athreya}. Let $\mathbb F\in\{\mathbb R,\mathbb C\}$. Denote by $\cL^1_g(A,\mathbb F)$ the vector space of $g$\--integrable functions over a $g$\--measurable set $A\subset \mathbb R$.
 In particular, we have that, for any interval $[a,b)$, $\mu_g([a,b))=g(b)-g(a)$.

 Consider the pseudometric given by $d_g(x,y)=\left\lvert g(x)-g(y)\right\rvert$ for $x, y \in \mathbb{R}$. Let $\tau_g$ be the topology induced by $d_g$. We denote by $\tau_u$ the usual topology of $\bF^n$ and by $d_u$ the associated euclidean metric.

\begin{dfn}
 \label{continuidaddefinicionlol}
	Given any $X\subset \mathbb{R}$, we say a function $f:X \to \mathbb{F}$ is \emph{$g$\--continuous at a point $x\in X$}, if for all $\varepsilon>0$ there exists $\delta>0$ such that
	\[ \forall y\in X,\, |g(y)-g(x)|<\delta \Rightarrow \left\lvert f(x)-f(y)\right\rvert<\varepsilon.\]
	We say $f$ is \emph{$g$\--continuous on $X$} if it is $g$\--continuous at every point $x\in X$. We will say that $f$ is \emph{uniformly $g$\--continuous} if for all $\varepsilon>0$ there exists $\delta>0$ such that
	\[ \forall x,y\in X \text{ with } |g(y)-g(x)|<\delta \Rightarrow \left\lvert f(x)-f(y)\right\rvert<\varepsilon.\]

More generally, given a pseudometric $d:X\times X\to\bR$, we say a function $f:X\to\bF$ is \emph{continuous at} $x\in X$ if for all $\varepsilon>0$ there exists $\delta>0$ such that
\[ \forall y\in X \text{ such that } d(x,y)<\delta \Rightarrow \left\lvert f(x)-f(y)\right\rvert<\varepsilon.\]
We say $f$ is \emph{continuous on} $X$ if it is continuous at every point $x\in X$. A function $f:X\to\bF$ is \emph{uniformly continuous} if for all $\varepsilon>0$ there exists $\delta>0$ such that
\[ \forall x,y\in X \text{ with } d(x,y)<\delta \Rightarrow \left\lvert f(x)-f(y)\right\rvert<\varepsilon.\]
Denote as $\operatorname{C}(X,\bF)$ the vector space of $\bF$--valued continuous functions and as $\operatorname{UC}(X,\bF)$ the subspace of uniformly continuous functions. This will be useful in Section~\ref{sec:gdomain}.
\end{dfn}

Given $X\ss \bR$, and a Banach space $Y$ we will denote by $\operatorname{R}(X,Y)$ the \emph{space of regulated functions} from $X$ to $Y$.

We will use the notation  $g(x^{+})\equiv\lim\limits_{y\to x^{+}}g(y)\in\mathbb{R}$, for any $x\in\mathbb{R}$. Define $\Delta g(x):=g(x^{+})-g(x)$ and let \[ D_g=\{x\in\mathbb{R} \ |\ \Delta g(x)>0 \}\]
be the \emph{set of discontinuities} of $g$. Define also
\[ C_g=\{x\in\mathbb{R} \ |\ g\text{ is constant on } (x-\varepsilon,x+\varepsilon)\text{ for some } \varepsilon>0 \}.\]
Note that $C_g$ is open in the usual topology and can be written as $	C_{g}=\bigcup_{n\in\widetilde{\Lambda}}(a_{n},b_{n})$,
where
$\widetilde{\Lambda}\subset \mathbb{N}$ and the intervals $(a_{n},b_{n})$ are pairwise disjoint.

\begin{dfn}\label{def:gderivative}
Let $\varepsilon > 0$ and let $f : (t - \varepsilon, t + \varepsilon) \to \mathbb{F}$ be a function. Define the \emph{Stieltjes derivative of $f$ at $t$} as
\[f'_{g}(t)=
\begin{dcases}
	\lim\limits_{s\rightarrow t}\frac{f(s)-f(t)}{g(s)-g(t)}, & t\notin D_{g}\cup C_{g}, \\
	\lim\limits_{s\rightarrow t^{+}}\frac{f(s)-f(t)}{g(s)-g(t)}, & t\in D_{g},\\
	\lim\limits_{s\rightarrow b_{n}^{+}}\frac{f(s)-f(b_{n})}{g(s)-g(b_{n})}, & t\in (a_{n},b_{n}),
\end{dcases} \]
if the respective limit can be considered and exists. In this case, we say that $f$ is \textit{ $g$\--differentiable at $t$}, or \textit{differentiable with respect to $g$ at $t$}.
\end{dfn}

Stieltjes integrals and derivatives are coherent and the Fundamental Theorem of Stieltjes\--Calculus holds.
\begin{thm}[{\cite[Theorem 5.4]{PousoRodriguez}}]
 \label{tf}
 Let $F:[a,b] \to \mathbb R$. The following conditions are equivalent:
 \begin{enumerate}
 \item The function $F$ is \emph{$g$\--absolutely continuous}: for every $\varepsilon>0$, there exists some $\delta>0$ such that, for any family $\{(a_n,b_n)\}_{n=1}^{m}$ of pairwise disjoint open subintervals of $[a,b]$,
 \[
 \sum_{n=1}^{m}(g(b_n)-g(a_n))<\delta\Ra
 \sum_{n=1}^m|F(b_n)-F(a_n)|<\varepsilon.
 \]
 \item The function $F$ fulfills the following properties:
 \begin{enumerate}
 \item there exists $F'_g(t)$ for $\mu_g$\--almost all $t\in [a,b)$ (i.e., for all $t$ except on a set of $\mu_g$ measure zero);
 \item $F'_g \in \cL^1_g([a,b))$;
 \item for each $t \in [a,b]$, we have
 \begin{equation*}
 F(t)=F(a)+\int_{[a,t)} F'_g(s) \operatorname*{d}\mu_g.
 \end{equation*}
 \end{enumerate}
 \end{enumerate}
 \end{thm}

Note that $g$\--absolutely continuous functions defined on a bounded interval $[a,b]$ are uniformly $g$\--continuous. Now we present a series of definitions and results from \cite{Cora2023,weier}. We will also use the following notation for integrals:
\[\int_{x}^{y}f\operatorname{d}\mu_{g}\equiv\int_{x}^{y}f(s)\operatorname{d}\mu_{g}(s):=\left\{\begin{aligned}
	&\int_{[x,y)} f(s), \operatorname{d}\mu_g(s),\; y\geq x,\\
	& -\int_{[y,x)} f(s)\, \operatorname{d}\mu_g(s),\; y< x.
\end{aligned}
\right.\]
\begin{dfn}[{\cite[Definition 3.1]{Cora2023}}]\label{def:pol}
	Let $g:\mathbb{R}\to\mathbb{R}$ be a derivator and fix some $x_0\in\mathbb{R}$. We define $g_{x_0,0}(x)=1$ for all $x\in\mathbb{R}$. Given any $n\in\mathbb N$, we define $g_{x_0,n}:\mathbb{R}\to\mathbb{R}$ recursively as
	\[
	g_{x_0,n}(x)=
 n\int_{x_0}^x g_{x_0,n-1}\operatorname*{d}\mu_g,\quad x\in\bR.
	\]
	These functions are called \emph{$g$\--monomials centered at $x_0$}, where $x_0$ is called the \emph{center} of $g_{x_0,n}$. We will call \emph{$g$\--polynomials} linear combinations of $g$\--monomials. We say a $g$\--polynomial $p$ centered at $x_0$ is of \emph{degree $n$} if $n$ is the greatest $m$ such that $g_{x_0,m}$ appears in the linear combination that forms $p$ (with a nonzero coefficient).
\end{dfn}

By Theorem~\ref{tf}, $g$\--polynomials are $g$\--absolutely continuous functions and, for all $n\in\mathbb N$, $(g_{x_0,n})'_g(x) = ng_{x_0,n-1}(x)$ whenever the $g$\--derivative can be considered at $x\in\mathbb R$. Thus, these functions belong to the class of \emph{$g$\--smooth functions} (that is, functions which are infinitely $g$\--differentiable).
\begin{pro}[{\cite[Lemma 3.5, Propositions 3.15-3.16]{Cora2023}}]
 \label{propiedadesmonomios}
 Fix $x_0\in \mathbb R$. Then,
 \begin{enumerate}
 \item[I.] For $x\geq x_0$ and $n\in\mathbb N$, $g_{x_0,n}(x)\geq 0$.
 \item[II.] For $x\leq x_0$ and $n\in\mathbb N$, $g_{x_0,2n}(x)\geq 0$ and $g_{x_0,2n-1}(x)\leq 0$.
 \item[III.]
 If $g$ is continuous on $[a,b]$, for $x_0\in[a,b]$,
 $
 g_{x_0,n}(x)=g_{x_0,1}^n(x)
$
 for all $x\in[a,b]$.
 \item[IV.]If $g$ is such that $g^C=0$ and $\#[x_0,x)\cap D_g=m\in \mathbb N$ for some $x> x_0$, then, $g_{x_0,n}(y)=0$ for $n\geq m+1$ and $y\in[x_0,x]$.
\end{enumerate}
\end{pro}
\begin{pro}[{\cite[Proposition 3.14]{Cora2023}}]
\label{center}
For all $r,s\in \mathbb R$, $n\in\mathbb N$ and $x\in \mathbb R$,
\[
g_{r,n}(x)=\sum_{k=0}^n {n\choose k} g_{r,k}(s) g_{s,n-k}(x).
\]
\end{pro}
The formula above lets us rewrite any $g$\--monomial centered at a certain point as a $g$\--polynomial (of the same degree) centered at another desired point. This proves that, for any $n\in\mathbb N$, the $\mathbb F$\--vector space of $g$\--polynomials of degree at most $n$ is generated by the $g$\--monomials $g_{s,k}$, $k=0,\dots,n$ centered at any given point $s$ (or even at different points). Denote as $\operatorname{P}_n$ this $\mathbb F$\--vector space. Denote as $\operatorname*{P}$ the $\mathbb F$\--vector space of all $g$\--polynomials. When necessary, we will write $\operatorname*{P}(\bR)$ or $\operatorname*{P}(\bC)$ to indicate the field and, if a particular domain $X$ needs to be considered, we will write $\operatorname*{P}(X,\bR)$ or $\operatorname*{P}(X,\bC)$.

Given a vector space $ Y$ over $\mathbb F$ we denote its \emph{dual space} as $Y^*$.
If $T:V\longrightarrow W$ is a linear mapping between finite dimensional $\bF$--vector spaces we denote as $[T]_{\cal{B}\cal{C}}$ the matrix that reads $T$ on coordinates where $\cal{B}$ and $\cal{C}$ are bases of $W$ and $V$ respectively. If $W=\bF^n$ and $V=\bF^m$ for some $n,m\in\bN$, denote just as $[T]\equiv [T]_{\cal{B}\cal{C}}$ where $\cal{B}$ and $\cal{C}$ are the canonical basis. Also, $[T]_{i,j}$ will denote the $i,j$\--th entry of the matrix $[T]$.

\begin{dfn}
 Given $x_0\in \mathbb R$ and $n\in\mathbb{N}$, denote by $\operatorname{L}_{x_0}:\mathbb{F}^{n+1}\to \operatorname{P}_n$ the linear map such that $\operatorname{L}_{x_0}(\l_0,\l_1,\dots,\l_n)=\sum_{k=0}^n \lambda_{k}g_{x_0,k}$.

 Take $x_1\in\mathbb R$ and define $\operatorname{D},\operatorname{M}_{x_0 \to x_1}:\mathbb F^{n+1}\to\mathbb F^{n+1}$ the linear maps
given by the matrices
\[
[\operatorname{D}] = \(i\delta_i^{j-1}\)_{i,j\in\{1,\dots,n+1\}}, \quad [\operatorname{M}_{x_0 \to x_1}] =\(
 {j-1 \choose i-1}g_{x_0,j-i}(x_1)\)_{i,j\in\{1,\dots,n+1\}}
\]
where we consider $g_{x_0,k}=0$ for $k<0$ and $\delta_i^{j}$ denotes the Kronecker delta. The operator $\operatorname{D}$ provides the expression of the formal  $g$\--derivative of polynomials just in terms of the coefficients of those polynomials with respect to the monomial generators, whereas $\operatorname{M}_{x_0 \to x_1}$ gives the change of coordinates from the generators $\{g_{x_0,k}\}_{k=0}^n$ to $\{g_{x_1,k}\}_{k=0}^n$.
\end{dfn}

Note that, since $[\operatorname{M}_{x_0 \to x_1}]$ is triangular with ones in the diagonal, $\operatorname{M}_{x_0 \to x_1}$ is bijective. On the other hand, $\{g_{x_0,k}\}_{k=0}^n$ is a set of generators of $\operatorname{P}_n$, but not necessarily a basis, so $\operatorname{L}_{x_0}$ may not be an isomorphism of vector spaces. It is, if and only if, $\{g_{x_0,k}\}_{k=0}^n$ is a linearly independent set of $\mathbb F^{\mathbb R}$.

Applying Proposition~\ref{center}, we obtain $
\operatorname{L}_{x_1}\operatorname{M}_{x_0 \to x_1}(\lambda)(x) = \operatorname{L}_{x_0}(\lambda)(x)$ for every $x\in \mathbb R$ and $\lambda\in \mathbb F^{n+1}$.
 Hence, $
\operatorname{L}_{x_1}\circ\operatorname{M}_{x_0 \to x_1} = \operatorname{L}_{x_0}$.
Also, for all $\lambda\in \mathbb F^{n+1}$, and all $x,y\in\mathbb R$ such that $x<y$,
\[
\operatorname{L}_{x_0}(\lambda)(y)-\operatorname{L}_{x_0}(\lambda)(x) = \int_{[x,y)}\operatorname{L}_{x_0}\operatorname{D}(\lambda)(t)\operatorname{d}\mu_g.
\]
Thus, by the Fundamental Theorem of Calculus,
\begin{equation}\label{eq:operatorDmasintegral}
\operatorname{L}_{x_0}(\lambda)'_g(x) = \operatorname{L}_{x_0}\operatorname{D}(\lambda)(x)
\end{equation}
for $g$\--a.a. $x\in\mathbb R$.

The necessity of operating formally on the coefficients of the $g$\--polynomials instead of the $g$\--polynomials themselves is illustrated in the following example.
\begin{exa}\label{exa:gconstant}
 Take $g$ as a constant function. Then, $\mu_g(\mathbb R)=0$ and all functions are null $g$\--almost everywhere. Hence, $g_{x_0,1}(x)=0$ for all $x_0,x\in\mathbb R$. For $n\geq 1$, the map
 \[ \begin{tikzcd}[row sep= 0em]
 g_{x_0,k}\in \{g_{x_0,k}\}_{k=0}^n \arrow[mapsto]{r} & kg_{x_0,k-1}\in\operatorname{P}_n
 \end{tikzcd}
 \]
 may not be extended linearly to all $\operatorname{P}_n$ since it is mapping a zero function ($g_{x_0,1}$) to a nonzero function ($g_{x_0,0}=1$). However, equation~\eqref{eq:operatorDmasintegral} still holds.
\end{exa}

Operating formally in the coefficients to obtain a derivation can be done on the polynomial ring of an arbitrary field, see the \emph{Hasse derivative} for example in \cite[Section 3]{LiLin2026ConfluentVandermonde}.
We now prove some properties the operators $\operatorname{M}_{x_0 \to x_1}$ and $\operatorname{D}$ satisfy.
\begin{lem}\label{lem:commutetativeoperators}
The following properties hold:
 \vspace{-1em}\begin{enumerate}[noitemsep]
 \item[I.] For all $x_0,x_1,x_2\in\mathbb R$, $
 \operatorname{M}_{x_0 \to x_1}\operatorname{M}_{x_1 \to x_2} = \operatorname{M}_{x_0 \to x_2}$.
 \item[II.] For all $x_0\in\mathbb R$, $
 \operatorname{M}_{x_0 \to x_0} = \operatorname{Id}_{\mathbb F^{n+1}}$.
 \item[III.] For all $x_0,x_1\in\mathbb R$, the operators $\operatorname{D}$ and $\operatorname{M}_{x_0 \to x_1}$ commute.
\end{enumerate}
\end{lem}
\begin{proof}
I. $[\operatorname{M}_{x_0 \to x_1}][\operatorname{M}_{x_1 \to x_2}]$ is an upper triangular matrix with ones on its diagonal. Let $i,j\in\{1,\dots,n+1\}$ with $i\leq j$. Then,
 \begin{align*}
 ([\operatorname{M}_{x_0 \to x_1}][\operatorname{M}_{x_1 \to x_2}])_{i,j} &= \sum_{k=1}^{n+1}[\operatorname{M}_{x_0 \to x_1}]_{i,k} [\operatorname{M}_{x_1 \to x_2}]_{k,j} = \sum_{k=i}^{j}[\operatorname{M}_{x_0 \to x_1}]_{i,k} [\operatorname{M}_{x_1 \to x_2}]_{k,j}\\
 &= \sum_{k=i}^{j}{k-1 \choose i-1}g_{x_0,k-i}(x_1){j-1 \choose k-1}g_{x_1,j-k}(x_2)\\
 &= \sum_{k=i}^{j}{j-1 \choose i-1}{j-i \choose k-i}g_{x_0,k-i}(x_1)g_{x_1,j-k}(x_2)\\
 &= {j-1 \choose i-1}\sum_{k=0}^{j-i}{j-i \choose k}g_{x_0,k}(x_1)g_{x_1,j-i-k}(x_2)\\
 &= {j-1 \choose i-1}g_{x_0,j-i}(x_2) = [\operatorname{M}_{x_0 \to x_2}]_{i,j}.
 \end{align*}

 II. This is immediate.

 III. Clearly, both $[\operatorname{D}][\operatorname{M}_{x_0 \to x_1}]$ and $[\operatorname{M}_{x_0 \to x_1}][\operatorname{D}]$ are upper triangular matrices with zeros on their diagonals. Take $i,j\in\{1,\dots,n+1\}$ with $i<j$. Then,
 \begin{align*}
 ([\operatorname{D}][\operatorname{M}_{x_0 \to x_1}])_{i,j} &= \sum_{k=1}^{n+1}[\operatorname{D}]_{i,k}[\operatorname{M}_{x_0 \to x_1}]_{k,j} = \sum_{k=1}^{j}[\operatorname{D}]_{i,k}[\operatorname{M}_{x_0 \to x_1}]_{k,j}\\
 & =\sum_{k=1}^{j}i\delta_{i}^{k-1}[\operatorname{M}_{x_0 \to x_1}]_{k,j} = i {j-1 \choose i}g_{x_0,j-i-1}(x_1),
 \end{align*}
 and
 \begin{align*}
 ([\operatorname{M}_{x_0 \to x_1}][\operatorname{D}])_{i,j} &= \sum_{k=1}^{n+1}[\operatorname{M}_{x_0 \to x_1}]_{i,k}[\operatorname{D}]_{k,j} = \sum_{k=i}^{n+1}[\operatorname{M}_{x_0 \to x_1}]_{i,k}[\operatorname{D}]_{k,j}\\
 & =\sum_{k=i}^{n+1}[\operatorname{M}_{x_0 \to x_1}]_{i,k}k\delta_{k}^{j-1} = (j-1) {j-2 \choose i-1}g_{x_0,j-1-i}(x_1).
 \end{align*}
 It is easy to show that $
 i {j-1 \choose i} = (j-1) {j-2 \choose i-1}$.
 Therefore, $[\operatorname{D}][\operatorname{M}_{x_0 \to x_1}]= [\operatorname{M}_{x_0 \to x_1}][\operatorname{D}]$.
\end{proof}

\section{Domain of a uniformly $g$\--continuous function}\label{sec:gdomain}

We will extend uniformly $g$\--continuous functions to a new, bigger, domain.
Usually, uniformly $g$\--continuous functions are studied by a change of variable, see \cite[Theorem~3.13]{Fernandez2024} and \cite[Lemma~2.8]{weier}. We will take another approach. The topology $\tau_g$ is an initial topology. Let us discuss initial topologies briefly.

Let $X$ be a set and a mapping $f: X\longrightarrow \mathbb F$. The \emph{initial topology} $\tau_f = \{ f^{-1}(U)\,: \, U\in\tau_u\}$ is the smallest topology that makes the mapping $f$ continuous. The topology $\tau_f$ also coincides with the topology given by the pseudometric $
d_f(x_1,x_2) = |f(x_1)-f(x_2)|$, $x_1,x_2\in X.
$

\begin{thm}\label{th:changeofvariablestateofmind} Let $X$ be a set and consider a function $f:X\to\bF$. Then, every (uniformly) continuous function $h:(X,\tau_f) \longrightarrow \mathbb F$ is of the form $\phi\circ f$ where $\phi:f(X)\longrightarrow \mathbb F$ is (uniformly) continuous.
\end{thm}
\begin{proof}
	Consider the $T_0$--identification of $X$, i.e., the quotient space $X/\mathord{\sim}$ given by $
	x_1\sim x_2: \Leftrightarrow  f(x_1) = f(x_2)$.
	Since $d_f$ induces a metric passing through the quotient, $X/\mathord{\sim}$ is a metric space. The mapping $f$ induces an isometric homeomorphism between $X/\mathord{\sim}$ and $f(X)\subset \mathbb F$ (the latter being equipped with the usual topology). Hence, every (uniformly) continuous mapping $h: (X,\tau_f) \longrightarrow \mathbb F$ factors through the quotient $X/\mathord{\sim}$, cf. \cite[Theorem~3.13]{Fernandez2024} and \cite[Lemma~2.8]{weier}.
\end{proof}
Similarly, if $A\subset X$, the subspace topology on $A$ coincides with the initial topology given by $\restr{f}{A}$.
The $T_0$--identification, also known as the Kolmogorov quotient, may be defined for arbitrary topological spaces; see \cite{kolmogorovquotients}. The quotient projection $q \colon X \to X/\mathord{\sim}$ is surjective, continuous, open, and closed. Moreover, every open set of $X$ is saturated, that is, $q^{-1}(q(U)) = U$ for all open sets $U\ss X$; see \cite[Theorem 3.16]{kolmogorovquotients}. The Kolmogorov quotient preserves many topological properties, that is, $X$ has such a property if and only if $X/\mathord{\sim}$ has it. Let us show that connectedness and compactness are examples of this.

\begin{thm}
Let $q \colon X \to Y$ be a surjective, continuous, open map between topological spaces, and assume that every open subset of $X$ is saturated, i.e., $q^{-1}(q(U))=U$, for all open $U\ss X$. Then:
\vspace{-2ex}\begin{enumerate}[noitemsep]
	\item $X$ is compact if and only if $Y$ is compact.
	\item $X$ is connected if and only if $Y$ is connected.
\end{enumerate}
\end{thm}
\begin{proof}
	Since $q$ is continuous we have $Y$ is compact or connected if $X$ is. Let us show I. Take $\{U_i \,:\, i\in I\}$ an open cover of $X$. Then, $\{q(U_i) \,:\, i\in I\}$ is an open cover of $Y$. Since $Y$ is compact, there exist $i_1,\dots,i_n\in I$ such that $Y = q(U_{i_1})\cup\cdots\cup q(U_{i_n})$. Taking inverse images, since open sets are saturated, $X= q^{-1}(Y) = q^{-1}(q(U_{i_1}))\cup\cdots\cup q^{-1}(q(U_{i_n})) = U_{i_1}\cup\cdots\cup U_{i_n}$. This implies $X$ is compact. We shall prove II. Take $U,V\ss X$ open sets of $X$ such that $U\cap V=\emptyset$ and $U\cup V =X$. Hence, $q(U),q(V)\ss Y$ are open sets of $Y$ such that $q(U)\cup q(V)=Y$ and $\emptyset = q^{-1}(q(U)\cap q(V))$ from where we deduce $q(U)\cap q(V)=\emptyset$. Since $Y$ is connected $q(U),q(V)\in\{Y,\emptyset\}$. But that implies $U,V\in\{X,\emptyset\}$. Indeed, if $q(V)=Y$, then $V=q^{-1}(q(V))=X$ and if $q(V)=\emptyset$ then $V=\emptyset$.
\end{proof}

Returning to our previous setting, let $X$ be a set and $f\colon X\to\bF$. If $X$ is endowed with the initial topology $\tau_f$ then $X$ is compact (connected) if and only if $f(X)$ is compact (connected).

\begin{dfn}\label{def:Rmas}
 Let $
 \mathbb R_+ := (\mathbb R\times \{0\})\sqcup(D_g\times\{1\})\subset \mathbb R\times\{0,1\}.
$
 Notice that $\mathbb R_+$ is the disjoint union of $\mathbb R$ and $D_g$. We will denote by $D_g^+$ the set $D_g\times\{1\}$. Define the mapping $G : \mathbb R_+\longrightarrow \mathbb R$ as follows
 \[
 G(x,0)=g(x), \;\; x\in\mathbb R,\quad G(x,1) = \lim_{t\to x^+} g(t)\equiv g(x^+), \;\; x\in D_g.
 \]
 Endow $\mathbb R_+$ with the initial topology given by $G$ and consider its associated pseudometric $d_G$.

 Let $a,b\in \mathbb R$ with $a<b$. Set
$
 \lbrack a,b\rbrack_+: = (\lbrack a,b\rbrack\times\{0\})\sqcup [(D_g\cap \lbrack a,b))\times\{1\}]\subset \mathbb R_+.
$
\end{dfn}

We shall occasionally identify $[a,b]$ with $[a,b]\times\{0\}$ and $\restr{G}{[a,b]_+}$ with $G$ by a slight abuse of notation. Also, we will denote $D_g^+\equiv(D_g\cap [a,b))\times\{1\}$. With this convention, $
[a,b]_+ = [a,b]\sqcup D_g^+,
$
and we use the analogous decomposition for $\mathbb{R}_+$. Since $G$
extends $g$, the space $([a,b],\tau_g)$ inherits the subspace topology
from $([a,b]_+,\tau_G)$, which is a subspace of
$(\mathbb{R}_+,\tau_G)$.

\begin{lem}
 \label{clausuragab}
Given $a,b\in\mathbb R$, $a<b$. Then, $\overline{g([a,b])}=g([a,b])\cup \{g(t^+):t\in D_g\cap[a,b)\}$,
where $\overline{g([a,b])}$ denotes the usual closure of $g([a,b])\subset \mathbb R$.
\end{lem}
\begin{proof} Clearly,
$g([a,b])\cup \{g(t^+):t\in D_g\cap[a,b)\}\subset \overline{g([a,b])}$,
since $g(t^+)\in \overline{g([a,b])}$ for all $t\in [a,b)$.
Take now $x\in \overline{g([a,b])}$. There exists a sequence $(t_n)_{n\in\mathbb N}\subset [a,b]$ such that $g(t_n)\rightarrow x$ as $n$ tends to infinity. Since $[a,b]$ is a usual compact set, passing to a subsequence if necessary, we can assume that $t_n\rightarrow t$ as $n$ tends to infinity for some $t\in[a,b]$. Thus, either $t_n\leq t$ for infinitely many indices or $t_n > t$ for infinitely many indices. Hence, the sequence $(t_n)_{n\in\mathbb N}$ approximates $t$ either from the left or from the right. We obtain that $x\in\{g(t),g(t^+)\}$.
\end{proof}

The above lemma proves that $G([a,b]_+) = \overline{g([a,b])}$. Hence, we have the following commutative diagram
\[
\begin{tikzcd}[row sep=2em, column sep=3em]
	([a,b],d_g) \arrow[hookrightarrow]{d} \arrow[r, "g"] & (g([a,b]),d_u)\arrow[hookrightarrow]{d} \\
	([a,b]_+,d_{G}) \arrow[r, "G"] & (\overline{g([a,b])},d_u)
	\arrow[phantom, from=1-1, to=2-2, "\circlearrowright" description]
\end{tikzcd}
\]
where the horizontal mappings induce isometric homeomorphisms once factored through the quotient. Since the $T_0$\--identification of $[a,b]_+$ is homeomorphic to the compact set $\overline{g([a,b])}$, the pseudometric space $([a,b]_+,d_G)$ is compact. Also, $g([a,b])$ is dense on $\overline{g([a,b])}$ hence, $[a,b]$ is dense in $[a,b]_+$ with respect to the topology $\tau_G$.

\begin{dfn} \label{def:contispaces}
 Let $X\subset \mathbb R_+$.
 We denote by $\operatorname{C}_g(X,\mathbb F)$ the vector \emph{space of $\tau_G$\--continuous functions} and as $\operatorname{UC}_g(X,\mathbb F)$ the \emph{space of uniformly $d_G$\--continuous functions}.
\end{dfn}

If $X\subset \mathbb R\subset \mathbb R_+$, the subspace  topology and pseudometric of $X$ in $\bR_+$ is the same as the $\tau_g$ topology and $d_g$ pseudometric respectively, and the notation of Definition~\ref{def:contispaces} is coherent with the standard notation in the literature, check for example, \cite[Definition 2.1]{Fernandez2024}. Note that since $[a,b]_+$ is compact, $\operatorname{UC}_g([a,b]_+,\mathbb F) = \operatorname{C}_g([a,b]_+,\mathbb F)$.

The following is a classical result; see, for instance, \cite[Theorem 10.9.1]{OSSearcoid2007}. Although it is stated there for metric spaces, the arguments remain valid in the pseudometric setting.

\begin{thm}
 \label{extensionpseudometrix}
 Let $M$ be a pseudometric space, $D$ a dense subset of $M$ and $Y$ a complete metric space. If $f: D\rightarrow Y$ is a uniformly continuous function then $f$ admits a unique continuous extension $\tilde f: M\rightarrow Y$. Even more, $\tilde f$ is uniformly continuous.
\end{thm}

In our setting, every uniformly $g$\--continuous function $f:[a,b]\longrightarrow \mathbb F$ admits a unique extension to a uniformly continuous function $\tilde f: [a,b]_+\longrightarrow\mathbb{F}$, defined by
\begin{equation}
 \label{extensiondefuncaabg}
\restr{\tilde f}{[a,b]} = f,\quad \tilde f(t) = \lim_{y\to t^+} f(y), \;t\in D_g^+.
\end{equation}
Throughout the paper, given $f\in \operatorname{UC}_g([a,b],\mathbb{F})$, we shall identify $f$ with its unique extension to $[a,b]_+$ whenever no confusion may arise. Thus, from now on we will identify $G$ with $g$. Moreover, for $t\in D_g\cap [a,b)$, we denote by $t^+$ its corresponding copy in $D_g^+$. With this convention, the evaluation of a uniformly $g$\--continuous function at $t^+$ is well defined and consistent with~\eqref{extensiondefuncaabg}.

By Theorem~\ref{extensionpseudometrix} and Theorem~\ref{th:changeofvariablestateofmind}, we have the following corollary that concerns the spaces of $g$\--continuous (C), $g$\--uniformly continuous (UC), bounded $g$\--continuous (BC) and bounded uniformly $g$\--continuous (BUC) functions.
\begin{cor}
 \label{isomgcontandtododemas}
Let $X\subset \mathbb R_+$ and $\overline X$ denote the closure of $X$ in $\mathbb R_+$.
The following maps are vector space isomorphisms:
\[
\begin{tikzcd}[row sep=0em,column sep=1em]
 \operatorname{C}_{\operatorname{id}}(g(X),\mathbb F) \arrow[r, "\Phi"] & \operatorname{C}_g(X,\mathbb F) \\
 f \arrow[r, mapsto] & f\circ g
\end{tikzcd}\quad
 \begin{tikzcd}[row sep=0em]
 \operatorname{UC}_{\operatorname{id}}(g(X),\mathbb F) \arrow[r, "\Psi"] & \operatorname{UC}_g(X,\mathbb F) \\
 f \arrow[r, mapsto] & f\circ g
\end{tikzcd}\quad
\begin{tikzcd}[row sep=0em]
 \operatorname{UC}_{g}\left(\overline X,\mathbb F\right) \arrow[r, "\Theta"] & \operatorname{UC}_g(X,\mathbb F) \\
 f \arrow[r, mapsto] & \restr{f}{X}
\end{tikzcd}\]
\[
\begin{tikzcd}[row sep=0em,column sep=1em]
	\operatorname{BC}_{\operatorname{id}}(g(X),\mathbb F) \arrow[r, "\til\Phi"] & \operatorname{BC}_g(X,\mathbb F) \\
	f \arrow[r, mapsto] & f\circ g
\end{tikzcd}
\begin{tikzcd}[row sep=0em]
	\operatorname{BUC}_{\operatorname{id}}(g(X),\mathbb F) \arrow[r, "\til\Psi"] & \operatorname{BUC}_g(X,\mathbb F) \\
	f \arrow[r, mapsto] & f\circ g
\end{tikzcd}
\begin{tikzcd}[row sep=0em]
	\operatorname{BUC}_{g}\left(\overline X,\mathbb F\right) \arrow[r, "\til\Theta"] & \operatorname{BUC}_g(X,\mathbb F) \\
	f \arrow[r, mapsto] & \restr{f}{X}
\end{tikzcd}
\]

$\til\Phi$, $\til\Psi$ and $\til\Theta$ are isometric isomorphisms of Banach spaces equipped with the supremum norm. If $X$ is compact, $\Phi=\til\Phi = \Psi=\til\Psi$ and $\Theta=\til\Theta$. If $X$ is relatively compact, $\Psi=\til\Psi$ and $\Theta=\til\Theta$.

In particular, if $X=[a,b]$, then, the following maps are isometric isomorphisms of Banach spaces:
\[
 \begin{tikzcd}[row sep=0em, column sep=3em]
 \operatorname{C}_g([a,b]_+,\mathbb F) \arrow[r, "\Theta"] & \operatorname{UC}_g([a,b],\mathbb F) & \operatorname{C}_{\operatorname{id}}(\overline{g([a,b])},\mathbb F)\arrow[r, "\Phi"] & \operatorname{C}_g([a,b]_+,\mathbb F) \\
 f \arrow[r, mapsto] & \restr{f}{[a,b]} & f \arrow[r, mapsto] & f\circ G
\end{tikzcd}
\]
\end{cor}
\begin{proof}
 It is easy to show that all maps defined are isometries.
 	Clearly, $\Phi$ is a linear mapping. Since $g$ maps $X$ onto $g(X)$ surjectively the injectivity of $\Phi$ follows. Hence, $\Phi$ is bijective thanks to Theorem~\ref{th:changeofvariablestateofmind}. $\Psi$ is proven analogously and $\Theta$ follows in a similar way using Theorem~\ref{extensionpseudometrix} instead.

	The spaces appearing on the definitions of $\til\Phi$, $\til\Psi$ and $\til\Theta$ are of bounded functions. Furthermore, they all are closed subspaces of the Banach space of bounded $\bF$--valued functions on $X$ and, therefore, Banach spaces.

	When $X$ is compact, any continuous function on $X$ is bounded and uniformly continuous so, in that case, $\Phi=\til\Phi = \Psi=\til\Psi$ and $\Theta=\til\Theta$. When $X$ is relatively compact, $\overline{X}$ is compact and, using $\Theta$, we see that the elements of $\operatorname{UC}_g(X,\mathbb F)$ are bounded.
\end{proof}

\cite[Corollary~2.11]{weier} provides a related result, while Corollary~\ref{isomgcontandtododemas} offers further insight into this setting. Since $\operatorname{P}_g(X,\mathbb{F}) \subset \operatorname{C}_g(X,\mathbb{F})$, it follows that
$
\operatorname{dim} \operatorname{P}_g (X,\mathbb{F})\leq \operatorname{dim} \operatorname{C}_g(X,\mathbb{F}).
$
If $g(X)$ is finite, then both $\operatorname{C}_g(X,\mathbb{F})$ and $\operatorname{P}_g(X,\mathbb{F})$ are finite--dimensional. Consequently, the set of $g$\--monomials must be linearly dependent.


\begin{lem}
 \label{extremallemma}
Let $f\in \operatorname{UC}_g([a,b],\mathbb F)$. Then $f$ is bounded and, when $\bF = \bR$, it attains its maximum and minimum on $[a,b]_+$. That is, either on points of $[a,b]$, its usual domain, or as the right hand side limit on discontinuity points of $g$. In particular, there exist $t\in[a,b]_+$ such that $|f(t)| = \left\lVert f\right\rVert_\infty.$
\end{lem}
\begin{proof}
It follows immediately from the compactness of $[a,b]_+$.
\end{proof}

The topology $\tau_G$ will not distinguish between two points $x,y\in\mathbb R_+$ for which $g(x)=g(y)$. Hence, the following definition.
\begin{dfn}
 Let $x,y\in \mathbb R_+$. We say $x$ and $y$ are\emph{ $g$\--distinct} if and only if
 $g(x)\neq g(y)$ and we write $x \neq_g y$.
Also, we denote $x<_g y :\Leftrightarrow g(x)< g(y)$, $x\leq_g y :\Leftrightarrow g(x)\leq g(y)$.

$\leq_g$ is a total preorder over $\mathbb R_+$. We say that a sequence of points $(t_i)_{i=0}^n\subset \mathbb R_+$, where $n\in\mathbb N$, is \emph{ordered from left to right} if
$
t_0\leq_g t_1\leq_g \cdots\leq_g t_n$, that is, $g(t_0)\leq g(t_1)\leq \cdots\leq g(t_n)$.

We define the intervals
 \[
 \begin{array}{l}
 [x,y)^g = \{z\in \mathbb{R} \,:\, g(x)\leq g(z) < g(y)\} = \{z\in\mathbb{R} \,:\, x\leq_g z <_g y\},\\
 \left[ x,y \right)^{g}_{+} = \{z\in\mathbb{R}_+ \,:\, g(x)\leq g(z) < g(y)\} = \{z\in\mathbb{R}_+ \,:\, x\leq_g z <_g y\}.
 \end{array}
 \]
The intervals $(x,y]^g$, $(x,y]^g_+$, $(x,y)^g$, $(x,y)^g_+$, $[x,y]^g$ and $[x,y]^g_+$ are defined analogously.
\end{dfn}
Let $[a,b]$ be an interval and let $x,y\in [a,b]$ satisfy $x <_g y$. In general, it is not true that $[x,y)^g \subset [a,b]$. In fact, one can show that
\[
[x,y)^g \setminus [a,b]
\subset \{ z\in(-\infty,a) : g(z)=g(a) \},
\]
cf. \cite[Lemmas~4.3 and 4.7.1]{Tojo2025}. The latter set, however, has zero $g$\--measure (see \cite[Lemma~4.7.3]{Tojo2025} and \cite[Lemma~2.14]{weier}) and each point in it is topologically indistinguishable from the point $a\in [a,b]$. For this reason, when working within a fixed interval $[a,b]$ or $[a,b]_+$, we shall simply interpret $[x,y)^g$ as $[x,y)^g\cap [a,b]$. The same convention will be adopted for $[x,y)^g_+$ relative to $[a,b]_+$. We have the following result in the line of \cite[Lemma~4.7.3]{Tojo2025}.
\begin{pro}
 \label{intervalsabmas}
 Let $t_1,t_2\in [a,b]_+$ such that $t_1<_g t_2$.
 \vspace{-1ex}\begin{enumerate}[noitemsep]
 \item[I.] If $t_1,t_2\in[a,b]$, then $\mu_g([t_1,t_2)^g \,\triangle\, [t_1,t_2) ) = 0$.
 \item[II.] If $t_1=s_1^+$ with $s_1\in D_g\cap[a,b)$ and $t_2\in[a,b]$, then $\mu_g([t_1,t_2)^g \,\triangle\, (s_1,t_2))=0$.
 \item[III.] If $t_2=s_2^+$ with $s_2\in D_g\cap[a,b)$ and $t_1\in[a,b]$, then $\mu_g([t_1,t_2)^g \,\triangle\, [t_1,s_2])=0$.
 \item[IV.] If $t_1=s_1^+$ and $t_2=s_2^+$ with $s_1,s_2\in D_g\cap[a,b)$, then $[t_1,t_2)^g = (s_1,s_2]$.
 \end{enumerate}
 In all cases, $\mu_g([t_1,t_2)^g)=g(t_2)-g(t_1)>0$.
\end{pro}
\begin{proof}
I.
 $[t_1,t_2)^g \,\triangle\, [t_1,t_2)$ is the union of the sets
 \[\{z\in[a,b] : g(z)=g(t_1) \text{ and } z<t_1\},\qquad\{z\in[a,b] : g(z)=g(t_2) \text{ and } z<t_2\}.\]
 Thus, by \cite[Lemma 2.14]{weier}, the symmetric difference has $\mu_g$ null measure.

II. We have that $[t_1,t_2)^g\subset (s_1,t_2)$. Hence, \[[t_1,t_2)^g \,\triangle\, [t_1,t_2)=\{z\in[a,b] : g(z)=g(t_2) \text{ and } z<t_2\}.\] By \cite[Lemma 2.14]{weier}, $\mu_g([t_1,t_2)^g \,\triangle\, [t_1,s_2])=0$.

III. $[t_1,t_2)^g \,\triangle\, [t_1,t_2)=\{z\in[a,b] : g(z)=g(t_1) \text{ and } z<t_1\}$ and, by \cite[Lemma 2.14]{weier}, $\mu_g((s_1,t_2)\setminus[t_1,t_2)^g)=0$.

IV. For $z\in [t_1,t_2)^g$, we have that $
 g(s_1)<g(s_1^+)\leq g(z) < g(s_2^+)$, which implies that $z\in(s_1,s_2]$. It is immediate to show that $(s_1,s_2]\subset [t_1,t_2)^g$, so we get the equality.

As a direct consequence of the previous results, $\mu_g([t_1,t_2)^g)=g(t_2)-g(t_1)>0$.
\end{proof}

From Proposition~\ref{intervalsabmas} we derive Barrow's rule for points in $[a,b]_+$.

\begin{pro}[Barrow's rule]
 \label{intervalsabmasandmeasure}
 Let $f:[a,b]\longrightarrow \mathbb F$ be a $g$\--absolutely continuous function. Then, for all $t_1,t_2\in [a,b]_+$ such that $t_1\leq_g t_2$,
 \[
 f(t_2)-f(t_1)= \int_{[t_1,t_2)^g}f'_g\operatorname{d}\mu_g.
 \]
\end{pro}
\begin{proof} Note that $[t_1,t_2)^g$ is an interval over the real and thus a  $g$\--measurable set. If $g(t_1)=g(t_2)$ then $f(t_1)=f(t_2)$ and $[t_1,t_2)^g=\emptyset$, so the result holds. Assume $t_1<_gt_2$. We have to consider the same cases as in Proposition~\ref{intervalsabmas}. Using the $g$\--absolute continuity we deduce the following:\\
 \noindent $\bullet$ Case 1: $t_1,t_2\in[a,b]$.
 \[
 \int_{[t_1,t_2)^g}f'_g\operatorname{d}\mu_g = \int_{[t_1,t_2)}f'_g\operatorname{d}\mu_g = f(t_2)-f(t_1).
 \]
 \noindent $\bullet$ Case 2: $t_1=s_1^+$ with $s_1\in D_g\cap[a,b)$ and $t_2\in[a,b]$.
 \[
 \int_{[t_1,t_2)^g}f'_g\operatorname{d}\mu_g = \int_{(s_1,t_2)}f'_g\operatorname{d}\mu_g = f(t_2)-f(s_1^+) = f(t_2)-f(t_1).
 \]
 \noindent $\bullet$ Case 3: $t_2=s_2^+$ with $s_2\in D_g\cap[a,b)$ and $t_1\in[a,b]$.
 \[
 \int_{[t_1,t_2)^g}f'_g\operatorname{d}\mu_g = \int_{[t_1,s_2]}f'_g\operatorname{d}\mu_g = f(s_2^+)-f(t_1) = f(t_2)-f(t_1).
 \]
 \noindent $\bullet$ Case 4: $t_1=s_1^+$ and $t_2=s_2^+$ with $s_1,s_2\in D_g\cap[a,b)$.
 \[
 \int_{[t_1,t_2)^g}f'_g\operatorname{d}\mu_g = \int_{(s_1,s_2]}f'_g\operatorname{d}\mu_g = f(s_2^+)-f(s_1^+) = f(t_2)-f(t_1).\qedhere
 \]
\end{proof}

\section{Roots of a uniformly $g$\--continuous functions}
\label{sectionunifcont}
We now introduce a notion of root for uniformly $g$\--continuous functions. Under this definition, not only zeros but also points at which the function changes sign are regarded as roots. This broader concept allows us to establish results analogous to the classical theorems of Bolzano and Rolle.

\begin{dfn}
 \label{def:roots}
 Let $t\in[a,b]_+$ and $f\in \operatorname{UC}_g([a,b],\mathbb R)$. We say $f$ has a \emph{root} at $t$ of
 \vspace{-1ex}\begin{enumerate}[noitemsep]
 \item[I.] \emph{type I} if $f(t)=0$.
 \item[II.] \emph{type II} if $t\in D_g\cap[a,b)$ and $f(t)f(t^+)<0$. 
 \end{enumerate}
 Note that both cases cannot hold simultaneously. A similar concept appears in \cite[Definition~4.43]{BohnerPeterson2003}. 
\end{dfn}
The following results are extensions of classical theorems to the new context of the set $[x,y)^g$.

\begin{pro}[Bolzano's Theorem]
 \label{bolzano}
 Let $f\in \operatorname{UC}_g([a,b],\mathbb R)$ and $x,y\in [a,b]$ with $x<y$.
 Suppose $f(x)f(y)<0$. Then $f$ has a root in $[x,y)^g_+$.
\end{pro}
\begin{proof}
 Assume without loss of generality that $f(x)<0$. Note that $f(x)f(y)< 0$ implies $g(x)<g(y)$ and, thus, $[x,y)^g_+$ is nonempty. $\{ z\in [x,y] : f(z)< 0 \}$, is a nonempty and bounded $g$\--open set of $[x,y]$. Let $t := \sup \{ z\in [x,y] : f(z)< 0 \}$.
 By the left\--continuity of $f$, $t<y$ and $f(t)\leq 0$. Hence, $x \leq_g t <_g y$ and $f(t^+)\geq 0$. (In this case, $f(t^+)$ denotes just the right hand side limit of $f$ at $t$. Note that maybe $t\notin D_g$ so $t^+$ does not exist in $\bR_+$).

Now, either $f(t)=0$, and $f$ has a root of type I at $t$ or $f(t)\neq 0$. Thus, $t\in D_g\cap[x,y)$ since $f(t)\neq f(t^+)$. Hence, if $f(t^+)=0$, $x \leq_g t <_g t^+ <_g y$, $f$ has a root of type I at $t^+$ and we are done. If $f(t^+)>0$, then, $f(t)f(t^+)<0$, and $f$ has a root of type II at $t$.
\end{proof}

The above result is stated just for points in $[a,b]$. An analogous result holds for points in $[a,b]_+$.

\begin{cor}
	 \label{bolzanomas}
 Let $f\in \operatorname{UC}_g([a,b],\mathbb R)$ and $x,y\in [a,b]_+$ with $x<_gy$.
 Suppose $f(x)f(y)<0$. Then $f$ has a root in $[x,y)^g_+$.
\end{cor}
\begin{proof}
	We consider cases. If $x,y\in [a,b]$, the result follows from Proposition~\ref{bolzano}. Assume $x=s_1^+$ with $s_1\in D_g\cap [a,b)$ and $y\in[a,b]$. Thus, $g(s_1)<g(s_1^+)<g(y)$ and $s_1<y$. Since $f(s_1^+)\neq 0$, there exists $t\in (s_1,y)$ such that $f(t)f(s_1^+)>0$. Hence, $f(t)f(y)<0$ and $g(s_1^+)\leq g(t)<g(y)$. We can apply Proposition~\ref{bolzano} to obtain that $f$ has a root on $[t,y)^g_+\subset [x,y)^g_+$. Assume $y=s_2^+$ with $s_2\in D_g\cap [a,b)$ and $x\in[a,b]$. From $g(x)<g(y)$ we obtain $g(x)\leq g(s_2)$ and $x\leq s_2$. If $f(s_2) = 0$ or either $f(s_2)f(s_2^+)<0$ we obtain that $f$ has a root of type I or type II at $s_2\in [x,y)^g_+$. If $f(s_2)f(s_2^+)>0$, then $f(x)f(s_2)<0$ and Proposition~\ref{bolzano} shows $f$ has a root on $[x,s_2)^g_+\subset [x,y)^g_+$. The case $x=s_1^+$ with $s_1\in D_g\cap [a,b)$ and $y=s_2^+$ with $s_2\in D_g\cap [a,b)$ follows combining the arguments of the previous cases.
\end{proof}

\begin{pro}[Rolle's Theorem]
 \label{rolle}
 Let $h \in \operatorname{UC}_g([a,b],\mathbb R)$ and $t_1,t_2\in [a,b]_+$ such that $t_1<_g t_2$. If
 $\int_{[t_1,t_2)^g}h\operatorname{d}\mu_g = 0$, then, $h$ has a root in $[t_1,t_2)^g_+$.
\end{pro}
\begin{proof}
 If $h$ has constant sign on $[t_1,t_2)^g$, then either $h\geq 0$ or $h\leq 0$ on $[t_1,t_2)^g$. On both cases we deduce that $h(t)=0$ for $g$\--almost all $t\in [t_1,t_2)^g$ since the integral over $[t_1,t_2)^g$ vanishes. Thus, since $\mu_g([t_1,t_2)^g)\neq 0$, there exists $t\in [t_1,t_2)^g$ such that $h(t)=0$ and $t$ is a root of $h$.

 If $h$ changes sign on $[t_1,t_2)^g$, there exist $x_1,x_2\in [t_1,t_2)^g$ with $x_1<_g x_2$ such that $h(x_1)h(x_2)<0$. Applying Proposition~\ref{bolzano} we get that $h$ has a root in $[x_1,x_2)^g_+\subset [t_1,t_2)^g_+$.
\end{proof}

Let us show how this new concept of root relates to the above results.

\begin{pro}
 \label{rolle2}
 Let $h\in\operatorname{UC}_g([a,b],\mathbb R)$ and $f:[a,b]\longrightarrow \mathbb R$ be defined as
 \[f(x): = f(a)+ \int_{[a,x)}h\operatorname{d}\mu_g, \quad x\in[a,b].\]
 Let $t_1,t_2\in [a,b]_+^g$ be two $g$\--distinct roots of $f$ such that $t_1<_g t_2$, then $h$ has a root in $[t_1,t_2)^g_+$.
\end{pro}
\begin{proof}
 We need to consider cases by the type of root $t_1$ and $t_2$ are.

\noindent $\bullet$ \emph{Case 1: $f(t_1)=f(t_2)=0$.} Since $f$ is $g$\--absolutely continuous, by Proposition~\ref{intervalsabmasandmeasure},
 \[
 0 = f(t_2) - f(t_1) = \int_{[t_1,t_2)^g} h \operatorname{d}\mu_g.
 \]
 Applying Proposition~\ref{rolle}, we get that $h$ has a root in $[t_1,t_2)^g_+$.

\noindent $\bullet$ \emph{Case 2: $f(t_1)\neq 0$ and $f(t_2)=0$.} In this case, $t_1\in D_g\cap[a,b)$ and $f(t_1)f(t_1^+)<0$. Hence,
\[h(t_1) = \frac{f(t_1^+)-f(t_1)}{\Delta g(t_1)}\neq 0.\]
Assume $h(t_1)>0$. Hence, $f(t_1^+)>0$ and
 \[
 0 > f(t_2) - f(t_1^+) = \int_{[t_1^+,t_2)^g} h \operatorname{d}\mu_g.
 \]
 Thus, $\mu_g([t_1^+,t_2)^g)\neq 0$ and there exists some $t_3\in [t_1^+,t_2)^g$ such that $h(t_3)<0$. Note that $g(t_1)<g(t_1^+)\leq g(t_3)$. Hence, by Proposition~\ref{bolzano}, $h$ has a root in $[t_1,t_3)^g_+\subset [t_1,t_2)^g_+$.

\noindent $\bullet$ \emph{Case 3: $f(t_1) = 0$ and $f(t_2)\neq 0$.} Now, $t_2\in D_g\cap[a,b)$ and $f(t_2)f(t_2^+)<0$. Hence,
\[h(t_2) = \frac{f(t_2^+)-f(t_2)}{\Delta g(t_2)}\neq 0.\]
Assume $h(t_2)>0$. Thus, $ f(t_2^+)>0$ and
 \[
 0 > f(t_2) - f(t_1) = \int_{[t_1,t_2)^g} h \operatorname{d}\mu_g.
 \]
 We infer that $\mu_g([t_1,t_2)^g)\neq 0$ and the existence of some $t_3\in [t_1,t_2)^g$ such that $h(t_3)<0$. Hence, by Proposition~\ref{bolzano}, $h$ has a root in $[t_3,t_2)^g_+\subset [t_1,t_2)^g_+$.

\noindent $\bullet$ \emph{Case 4: $f(t_i)\neq 0$ for $i=1,2$.} Then, for $i=1,2$, $t_i\in D_g\cap[a,b)$ and
\[h(t_i) = \frac{f(t_i^+)-f(t_i)}{\Delta g(t_i)}\neq 0.\]
If $h(t_1)h(t_2)<0$, we apply Proposition~\ref{bolzano} to get a root of $h$ in $[t_1,t_2)^g_+$. If $h(t_1),h(t_2)>0$, then
 \[
 0 > f(t_2) - f(t_1^+) = \int_{[t_1^+,t_2)^g} h \operatorname{d}\mu_g.
 \]
 Thus, $\mu_g([t_1^+,t_2)^g)\neq 0$ and there exists $t_3\in [t_1^+,t_2)^g$ with $h(t_3)<0$ and $g(t_1)<g(t_1^+)\leq g(t_3)$. Applying Proposition~\ref{bolzano}, we get that $h$ has a root in $[t_1,t_3)^g_+\subset [t_1,t_2)^g_+$. The case $h(t_1),h(t_2)<0$ is analogous.
\end{proof}

The following theorem shows that a real $g$\--polynomial of order $n$ has at most $n$ $g$\--distinct roots.
\begin{thm}
 \label{rootsofpol}
 Let $n\in\mathbb N$ and $x_0\in\mathbb R$. Take $\lambda\in\mathbb R^{n+1}$ and consider $p=\operatorname{L}_{x_0}(\lambda)\in\operatorname{P}_n(\bR)$. If $p$ has $n+1$ $g$\--distinct roots, then $\lambda=0$.
\end{thm}
\begin{proof}
 We have that
 \begin{equation}\label{polyabscont}
 p(x) - p(y) = \int_{[y,x)} \operatorname{L}_{x_0}\operatorname{D}(\lambda)\operatorname{d}\mu_g, \quad x,y\in\mathbb R \quad y<x.
 \end{equation}
 We proceed by induction on $n$. Consider the case $n=0$. In this case, $\lambda=\lambda_0$ and $p \equiv \lambda$ is a constant function. By hypothesis, $p$ has a root $t\in\mathbb R_+$. Since $p$ is constant, the root $t$ cannot be of type II, so $0=p(t)=\lambda$.

 Let $n\in\mathbb N$ and assume induction hypothesis. There exists a sequence of $n+1$ roots of $p$ $t_1<_g t_2 <_g\dots <_gt_{n+1}$ in $\bR_+$.
 By Proposition~\ref{rolle2} and equation~\eqref{polyabscont} we have that the polynomial of degree $n-1$, $\operatorname{L}_{x_0}\operatorname{D}(\lambda)$
 has a root $s_i\in [t_i,t_{i+1})^g_+, \text{ for } i\in\{1,\dots,n\}$, which implies $s_i<_g s_j$ for $i<j$. By the induction hypothesis, $\lambda_k = 0$ for $k\in\{1,\dots,n\}$. Therefore, $p=\lambda_0$ is constant function with $n+1$ roots, so $p=0$ and $\lambda=0$.
\end{proof}

As a consequence, we have the following Lemma.

\begin{lem}
 \label{lem:independencegmnomialsX}
 Take $X\subset \mathbb R_+$, $n\in\mathbb N$ and $x_0\in\bR$. Let $\operatorname{P}_n(X,\bF):= \left\{ \restr{p}{X}\,:\, p\in \operatorname{P}_n(\bF)
 \right\}$. The following statements are equivalent.
 \vspace{-1em}\begin{enumerate}
 \item[I.] Let $i\colon \bF^{\bR}\to \bF^X$ be the linear mapping $i(f)=\restr{f}{X}$. The operator $i\circ\operatorname{L}_{x_0}\colon \bF^{n+1}\to \operatorname{P}_n(X,\bF)$ is an isomorphism of $\bF$--vector spaces. (Note this implies that $\operatorname{L}_{x_0}$ is an isomorphism).
 \item[II.] The set $\{g_{x_0,k}\}_{k=0}^n$ is a linearly independent subset of $\operatorname{P}_n(X,\bF)$. (Note this implies that $\{g_{x_0,k}\}_{k=0}^n$ is a basis of $\operatorname{P}_n(\bF)$).
 \item[III.] $\# g(X)\geq n+1$.
 \end{enumerate}
 As a consequence, $
 \operatorname*{dim} \operatorname{P}(X,\bF) = \operatorname*{dim} \operatorname{C_g(X,\mathbb F)} = \#g(X)$, where $\operatorname{P}(X,\bF) = \{ \restr{p}{X}\,:\, p\in \operatorname{P}(\bF)
 \}$.
\end{lem}
\begin{proof}
 Studying the dimension of the spaces involved, it is clear that I and II are equivalent.

 If we assume II, then $\{g_{x_0,k}\}_{k=0}^n$ is a linearly independent subset of $\operatorname{C}_g (X,\mathbb F)$. By Corollary~\ref{isomgcontandtododemas}, $\operatorname{dim} \operatorname{C}_g (X,\mathbb F) \leq \# g(X)$ and III holds.

Assume III, take $\lambda\in \mathbb C^{n+1}$ such that $\restr{\operatorname{L}_{x_0}(\lambda)}{X}=0$ and consider the real and imaginary parts of $\l$, $\operatorname{Re}\lambda, \operatorname{Im}\lambda \in\mathbb R^{n+1}$, such that $\lambda = \operatorname{Re}\lambda + i \operatorname{Im}\lambda$. Thus,
 \[
 \operatorname{L}_{x_0}(\lambda)(x) = \operatorname{L}_{x_0}(\operatorname{Re}\lambda + i \operatorname{Im}\lambda)(x) = \operatorname{L}_{x_0}(\operatorname{Re}\lambda)(x) + i \operatorname{L}_{x_0}(\operatorname{Im}\lambda)(x)=0,\quad \forall x\in X.
 \]
 Hence, $ \operatorname{L}_{x_0}(\operatorname{Re}\lambda)(x) = 0$ and $\operatorname{L}_{x_0}(\operatorname{Im}\lambda)(x)=0,\quad \forall x\in X$.
 By III, $\operatorname{L}_{x_0}(\operatorname{Re}\lambda)$ and $\operatorname{L}_{x_0}(\operatorname{Im}\lambda)$ have more that $n+1$ $g$\--distinct roots. Therefore, by Theorem~\ref{rootsofpol}, $\lambda=0$.
\end{proof}

The following corollary shows that a real $g$\--polynomial of degree $n$ can only change sign at most $n$ times.

\begin{cor}\label{cor:signchanges}
 Let $n\in\mathbb N$ and $p\in\operatorname{P}_n(\bR)$. Then, $p$ changes sign at most $n$ times on $\mathbb R_+$. That is, there do not exist $t_1<_g t_2 <_g\dots <_gt_{n+2}$ in $\bR_+$ such that $p(t_i)p(t_{i+1})< 0$ for $i\in\{1,\dots,n+1\}$.
\end{cor}
\begin{proof}
 Take $x_0\in\mathbb R$ and $\lambda\in\mathbb R^{n+1}$ such that $\operatorname{L}_{x_0}(\lambda)=p$. Suppose there exist $t_1<_g t_2 <_g\dots <_gt_{n+2}$ in $\bR_+$ such that $p(t_i)p(t_{i+1})< 0$ for $i\in\{1,\dots,n+1\}$.
 By Corollary~\ref{bolzanomas}, $p$ has a root $s_i\in [t_i,t_{i+1})^g_+$ for $i\in\{1,\dots,n+1\}$. By Theorem~\ref{rootsofpol}, we deduce that $\lambda_k=0$ for all $0\leq k\leq n$. That however implies $p=0$ and contradicts the fact that $p(t_i)p(t_{i+1})< 0$ for $i\in\{1,\dots,n+1\}$, so we get the result.
\end{proof}

If we have a $g$\--polynomial with $n$ roots then the leading coefficient gives us information about the sign of the polynomial.

\begin{cor}
 \label{cor:signoffirstchangeofsign}
 Let $n\in\mathbb N$, $x_0\in\mathbb R$, $\lambda\in\mathbb R^{n+1}$ and consider $p=\operatorname{L}_{x_0}(\lambda)\in\operatorname{P}_n(\bR)$. Let $t_1<_g t_2 <_g\dots <_gt_{n}$ be roots of $p$ in $\bR_+$. Then $\operatorname{sgn} (\lambda_n )(-1)^n p(x)\geq 0$, for all $x\in\mathbb R_+$ such that $g(x)\leq g(t_1)$.
\end{cor}
\begin{proof}
 If $\lambda_n = 0$, then $p\in\operatorname{P}_{n-1}$ and by Theorem~\ref{rootsofpol}, $\lambda=0$, so $p=0$.

 Suppose then that $\lambda_n \neq 0$. We proceed by induction on $n$. The case $n=1$ follows from the fact that $g$ is nondecreasing. Let $n\in\mathbb N$ with $n>1$. By Proposition~\ref{rolle2}, $\operatorname{L}_{x_0}\operatorname{D}(\lambda)$ has a root $s_i\in [t_i,t_{i+1})^g_+, \text{ for } i\in\{1,\dots,n-1\}$. Take $x\in\mathbb R_+$ such that $g(x)\leq g(t_1)$, by Proposition~\ref{intervalsabmasandmeasure},
 \begin{equation}\label{eq:signoolo}
p(x) = p(t_1)- \int_{[x,t_1)^g}\operatorname{L}_{x_0}\operatorname{D}(\lambda)\operatorname{d}\mu_g.
 \end{equation}
 However, $[x,t_1)^g\subset [x,s_1)^g$ and, by the induction hypothesis, $\operatorname{sgn} (n\lambda_n )(-1)^{n-1}\operatorname{L}_{x_0}\operatorname{D}(\lambda)(s)\geq 0$ for all $s\in\mathbb R_+$ such that $g(s)\leq g(s_1)$. Notice that $
 \operatorname{sgn} (n\lambda_n ) = \operatorname{sgn} (\lambda_n )$.
 We consider two cases. If $t_1$ is a root for $p$ of type I, multiply by $\operatorname{sgn}(\lambda_n)(-1)^{n}$ on both sides of \eqref{eq:signoolo}, hence
 \[
 \operatorname{sgn}(\lambda_n)(-1)^{n}p(x) = \int_{[x,t_1)^g}\operatorname{sgn} (n\lambda_n )(-1)^{n-1}\operatorname{L}_{x_0}\operatorname{D}(\lambda)\operatorname{d}\mu_g\geq 0
 \]
 since the latter is the integral of a nonnegative function. If $t_1$ is a root of type II, we have that $
 p(t_1^+)-p(t_1) = \operatorname{L}_{x_0}\operatorname{D}(\lambda)(t_1)\Delta g(t_1)
$. We deduce $p(t_1)$ has sign $\operatorname{sgn} (\lambda_n )(-1)^{n}$ or equivalently,
\[
\operatorname{sgn} (\lambda_n )(-1)^{n}p(t_1)\geq 0.
\]
Multiply by $\operatorname{sgn} (\lambda_n )(-1)^{n}$ on both sides of \eqref{eq:signoolo} again.
\end{proof}


\section{Lagrange interpolation}
\label{lagrangeinterpolsec}

In this section, we study the Lagrange interpolation problem for Stieltjes polynomials and introduce a natural generalization: we allow the interpolation nodes to lie in $\mathbb{R}_+$. In this way, Stieltjes polynomials form a class of functions that admit interpolation not only at point evaluations but also at right\--hand limits. This will be fundamental in Section~\ref{sec:chebishev}. More generally, as we will show, they allow interpolation with respect to certain convex combinations of the evaluation functionals.

\begin{dfn}[Lagrange interpolation problem] Let $n\in\mathbb N$ and consider any sequence of $n+1$ $g$\--distinct points $x_0,\dots,x_n\in \mathbb R_+$, called the \emph{interpolation nodes}, and any $n+1$ values $t_0,\dots,t_n\in\mathbb F$. We call the system of equations
\begin{equation}\label{ip}p(x_i)=t_i\,,\quad i\in\{0,\dots,n\},\end{equation}
the \emph{Lagrange interpolation problem}, for which we try to find a \emph{polynomial solution} $p\in \operatorname{P}_n(\bF)$.
\end{dfn}
\begin{rem}
 \label{remarknodos}
For the interpolation problem~\eqref{ip} to have a solution it is necessary that $g(x_i)\neq g(x_j)$ for every $i,j\in\{0,\dots,n\}$ such that $t_i\neq t_j$. Indeed, since any $g$\--polynomial $p$ is $g$\--continuous, $p(x)=p(y)$ for all $x,y\in\mathbb R_+$ such that $g(x)=g(y)$. On the other hand, interpolation nodes on which $g$ takes the same value are redundant (in that case, we can reduce problem~\eqref{ip} to have strictly less than $n+1$ equations making the solution no longer unique). This justifies the assumption that $g$ takes different values on every node.
\end{rem}
\begin{dfn}
 For $x\in\mathbb R_+$, define $\operatorname{ev}_{x}:\bF^{\bR_+} \to \mathbb F$ as the linear map $\operatorname{ev}_{x}(f)=f(x)$ for $f\in \bF^{\bR_+}$.
\end{dfn}
\begin{dfn}\label{def:Vandermondemapping}
	 Let $X$ and $Y$ be two sets, $A\subset Y^X$ and $T=\{x_1,\dots,x_n\}\ss X$. We define the \emph{Vandermonde mapping associated to $T$} as the map $V_A^T:A\to Y^n$ such that $V(p)=(p(x_1),\dots,p(x_n))$ for $p\in A$. If $Y$ has a $\bF$--vector space structure and $A$ is a subspace of $Y^X$ the mapping $V_A^T$ is $\bF$--linear. If $T$ (and respectively $A$) is clear from the context, we will write $V\equiv V_A\equiv V_A^T$.

For $A\ss \operatorname{R}([a,b],\bF)$, we can consider the following extension of the definition.
Given any $T=\{x_1,\dots,x_n\}\ss [a,b]_+$, we define the \emph{Vandermonde mapping} as the map $V_A^T: A \to \mathbb{F}^{n+1}$ such that $V_A^T(p)=(\til p(x_1),\dots,\til p(x_n))$ for $p\in A$, where $\til p$ is defined in \eqref{extensiondefuncaabg}.
\end{dfn}

\begin{rem}\label{rem:Vandermonde}
	Note that if $A\subset Y^X$, we can define $H = \{ \restr{f}{T}\,:\, f\in A\}\subset Y^T$. Apparently, there is no difference between $V^T_{H}$ and $V^T_A$. So we could just consider $T=X$ on the above definition. However, note the mapping $V^T_{H}$ is always injective. In fact, $V^T_{Y^T}$ is always a bijective map. If we define $i:Y^X\to Y^T$ as $i(f):=\restr{f}{T}$, then, $V^T_A = \restr{V^T_{Y^T}\circ i}{A}$.
\end{rem}

\begin{lem}
\label{lemequival}
Consider a set $T$ of $n+1$ $g$\--distinct points, $x_0,\dots,x_n\in\mathbb R_+$. The following statements are equivalent:
\begin{enumerate}[noitemsep, itemsep=.5ex]
\item $V\equiv V_{\operatorname{P}_n(\bF)}^T$ is an isomorphism of vector spaces.
\item The Lagrange interpolation problem~\eqref{ip} has unique solution for any $t_0,\dots,t_n\in\mathbb F$.
\item $\{\operatorname{ev}_{x_j}\}_{0\leq j\leq n}$ is a linearly independent subset of $\operatorname{P}_n(\bF)^*$, the dual of $\operatorname{P}_n(\bF)$.
\item The matrix
\begin{equation}
\label{vmatrix}[V\circ\operatorname{L}_{y_0} ]=
 \(g_{y_0,k}(x_j)\)_{j,k=0}^n=\begin{pmatrix}
 1 & g_{y_0,1}(x_0) &\dots & g_{y_0,n}(x_0) \\
 1 & g_{y_0,1}(x_1) &\dots & g_{{y_0,n}}(x_1)\\
 \vdots & \vdots & \ddots & \vdots \\
 1 & g_{y_0,1}(x_n) &\dots & g_{y_0,n}(x_n)
 \end{pmatrix}
\end{equation}
has nonzero determinant for some $y_0\in\mathbb R$.
\end{enumerate}
\end{lem}
\begin{proof}
 Clearly, II is equivalent to the map $V$ being bijective, so I and II are equivalent. If I holds, then $V$ is surjective and the adjoint mapping $V^*$ is injective. Since $\{\operatorname{ev}_{x_j}\}_{0\leq j\leq n}$ is the image through $V^*$ of the canonical basis induced on the dual space $(\mathbb F^{n+1})^*$, it must be a linearly independent set, so III holds.

 Assume III. Then, $n+1\leq \operatorname*{dim} \operatorname{P}_n(\bF)^* = \operatorname*{dim} \operatorname{P}_n(\bF)\leq n+1$. Thus, both $\{g_{y_0,k}\}_{0\leq k\leq n}$ and $\{\operatorname{ev}_{x_j}\}_{0\leq j\leq n}$ are bases and the matrix in~\eqref{vmatrix} must be invertible, so IV holds. On the other hand, if IV holds, it says $[V\circ\operatorname{L}_{y_0}]$ is invertible which implies that both $V$ and $\operatorname{L}_{y_0}$ are isomorphisms, so III holds.
\end{proof}

Note that in the case $g=\operatorname{id}$, matrix~\eqref{vmatrix} is the Vandermonde matrix, see \cite[Chapter II]{davis74}, \cite{LiLin2026ConfluentVandermonde}.

It turns out that the choice of center for the $g$\--monomials does not matter when it comes to compute the determinant of~\eqref{vmatrix}.

\begin{lem}
\label{reduction}
Let $y_0,y_1\in\mathbb R$ and $n\in \mathbb N$. Consider $n+1$ arbitrary points $x_0, \dots, x_{n}\in\mathbb R_+$, then
\[
V\circ\operatorname{L}_{y_0} = V\circ\operatorname{L}_{y_1}\operatorname{M}_{y_0 \to y_1},
\]
which implies
\[
|[V\circ \operatorname{L}_{y_0}]| = |[V\circ\operatorname{L}_{y_1}]|.
\]
\end{lem}
\begin{proof}
We know $\operatorname{L}_{y_1}\circ \operatorname{M}_{y_0 \to y_1} = \operatorname{L}_{y_0}$, so the first equality holds. Since $[\operatorname{M}_{y_1 \to y_0}]$ is an upper triangular matrix with ones on the diagonal, its determinant equals 1. The result now follows.
\end{proof}

As a consequence of Lemma~\ref{reduction}, the map
\[ x \mapsto |[V\circ\operatorname{L}_{x}]| = \begin{vmatrix}
 1 & g_{x,1}(x_0) &\dots & g_{x,n}(x_0)\\
 1 & g_{x,1}(x_1) &\dots & g_{x,n}(x_1)\\
 \vdots & \vdots & \ddots & \vdots \\
 1 & g_{x,1}(x_{n}) &\dots & g_{x,n}(x_{n}) \\
 \end{vmatrix},
\]
for $x\in\mathbb R$, is constant. Reasoning as in the proof of Lemma~\ref{lem:independencegmnomialsX} we can show that the Lagrange interpolation is possible.

\begin{thm}[Lagrange Interpolation]
\label{pith}
Fix $n\in \mathbb N$ and take $x_0,\dots,x_n\in\mathbb R_+$ $n+1$ $g$\--distinct points ($g(x_i)\neq g(x_j)$ for $i\neq j$). The Lagrange interpolation problem~\eqref{ip} has a unique solution. Furthermore, if the points $x_0,\dots,x_n$ are ordered from left to right ($g(x_i)<g(x_j)$ if $i<j$), then, for all $y_0\in\mathbb R$,
\begin{equation}
 \label{matrixlagrangeprueba}
|[V\circ\operatorname{L}_{y_0}]|=\begin{vmatrix}
 1 & g_{y_0,1}(x_0) &\dots & g_{y_0,n}(x_0)\\
 1 & g_{y_0,1}(x_1) &\dots & g_{y_0,n}(x_1)\\
 \vdots & \vdots & \ddots & \vdots \\
 1 & g_{y_0,1}(x_n) &\dots & g_{y_0,n}(x_n)
\end{vmatrix}>0.
\end{equation}
\end{thm}
\begin{proof}
Denote $T=\{x_0,x_1,\dots,x_n\}\subset \mathbb R_+$ and fix $y_0\in\mathbb R$. We aim to show that $V^T_{\operatorname{P}_n(\bF)}\circ\operatorname{L}_{y_0}$ is an isomorphism. Let $\lambda\in\mathbb F^{n+1}$ such that $\operatorname{L}_{y_0}(\lambda)\in \operatorname{ker}V$. That is, $
 \operatorname{L}_{y_0}(\lambda)(x_j)=0$, $j\in\{0,\dots,n\}$. Take $\operatorname{Re}\lambda,\operatorname{Im}\lambda \in\mathbb R^{n+1}$ such that $\lambda = \operatorname{Re}\lambda + i \operatorname{Im}\lambda$. Therefore,
 \[
 \operatorname{L}_{y_0}(\lambda)(x_j) = \operatorname{L}_{y_0}(\operatorname{Re}\lambda)(x_j) + i \operatorname{L}_{y_0}(\operatorname{Im}\lambda)(x_j)=0,\quad j\in\{0,\dots,n\}.
 \]
 We obtain
 $
 \operatorname{L}_{y_0}(\operatorname{Re}\lambda)(x_j) = 0$ and $ \operatorname{L}_{y_0}(\operatorname{Im}\lambda)(x_j)=0$ for $j\in\{0,\dots,n\}
 $.
 Since $\operatorname{L}_{y_0}(\operatorname{Re}\lambda)$ and $\operatorname{L}_{y_0}(\operatorname{Im}\lambda)$ have $n+1$ $g$\--distinct roots, we conclude, by Theorem~\ref{rootsofpol}, that $\lambda=0$. This shows that both $V$ and $\operatorname{L}_{y_0}$ are isomorphisms.
 Let us prove inequality~\eqref{matrixlagrangeprueba}. Define the $g$\--polynomial,
 \[ p(x) =
 \begin{vmatrix}
 1 & g_{y_0,1}(x) &\dots & g_{y_0,n}(x)\\
 1 & g_{y_0,1}(x_1) &\dots & g_{y_0,n}(x_1)\\
 \vdots & \vdots & \ddots & \vdots \\
 1 & g_{y_0,1}(x_n) &\dots & g_{y_0,n}(x_n)
\end{vmatrix}.
 \]
 $x_j$ is a root of type I of $p$ for $j\in\{1,\dots,n\}$. By Lemma~\ref{lemequival}, $p(x_0)\neq 0$. We now proceed by induction on $n$. The case $n=1$ follows after an easy computation. For $n>1$, the leading coefficient of $p$ is
 \[
 (-1)^{n}\begin{vmatrix}
 1 & g_{y_0,1}(x_1) &\dots & g_{y_0,n-1}(x_1)\\
 1 & g_{y_0,1}(x_2) &\dots & g_{y_0,n-1}(x_2)\\
 \vdots & \vdots & \ddots & \vdots \\
 1 & g_{y_0,1}(x_n) &\dots & g_{y_0,n-1}(x_n)
\end{vmatrix}
 \]
 which has sign $(-1)^n$ by induction hypothesis. By Corollary~\ref{cor:signoffirstchangeofsign}, $p(x_0)>0$.
\end{proof}

It is not a coincidence that Theorem~\ref{pith} and Lemma~\ref{lem:independencegmnomialsX} have almost identical proofs. We can deduce Theorem~\ref{pith} from Lemma~\ref{lem:independencegmnomialsX} in the following way.

\noindent \textit{Alternative proof of Theorem~\ref{pith}}. Take $X = \{x_j\,:\,j\in\{0,\dots,n\}\}\subset \mathbb R_+$. By Lemma~\ref{lem:independencegmnomialsX}, we have that $\operatorname{dim} \operatorname{P}_n(X,\bF) = n+1$. Therefore, $\operatorname{P}_n(X,\bF) = \operatorname{C}_g(X,\mathbb F)=\operatorname{C}_{\operatorname{id}}(g(X),\mathbb F)$ by Corollary~\ref{isomgcontandtododemas}.
Since $g(X)$ is discrete, $\operatorname{C}_{\operatorname{id}}(g(X),\mathbb F) = \bF^{g(X)}$, which gives $\smash{V^X_{\operatorname{P}_n(X,\bF)} = V^{g(X)}_{\bF^{g(X)}}}$. But the latter is bijective, see Remark~\ref{rem:Vandermonde}.
Hence, $\smash{V^X_{\operatorname{P}_n(X,\bF)}}$ must be an isomorphism. Thus, $\smash{V^X_{\operatorname{P}_n(\bF)} = V^X_{\operatorname{P}_n(X,\bF)}\circ i }$ is an isomorphism since $i\colon \operatorname{P}_n(\bF)\to \operatorname{P}_n(X,\bF)$, $i(p)=\restr{p}{X}$ is a vector space isomorphism. \qed

The mapping $W:\mathbb R^{n}_+\longrightarrow \mathbb R$,
\[
W(t_1,\dots,t_n) = \begin{vmatrix}
 1 & g_{x_0,1}(t_1) &\dots & g_{x_0,n-1}(t_1)\\
 1 & g_{x_0,1}(t_2) &\dots & g_{x_0,n-1}(t_2)\\
 \vdots & \vdots & \ddots & \vdots \\
 1 & g_{x_0,1}(t_n) &\dots & g_{x_0,n-1}(t_n)
\end{vmatrix}
\]
only vanishes on the set $\{(t_1,\dots,t_n)\in\mathbb R^n_+:\exists i\neq j\text{ such that }g(t_i)=g(t_j) \}$. So we can say $W$ is nonzero on any $n$\--tuple of $n$ $g$\--distinct points. In fact, its sign depends only on the parity of the permutation that orders the $t_j$ from left to right.

We can give a formula for $p$ using determinant calculus.
\begin{pro}
\label{formadeterminante}
Let $g$ be a derivator, $y_0\in\mathbb R$, $x_0,\dots,x_n\in\mathbb R_+$ $n+1$ $g$\--distinct points and $t_0,\dots,t_n\in\mathbb F$. Then the $g$\--polynomial
\[
p(x)=-
 \begin{vmatrix}
 1 & g_{y_0,1}(x_0) &\dots & g_{y_0,n}(x_0)\\
 1 & g_{y_0,1}(x_1) &\dots & g_{y_0,n}(x_1)\\
 \vdots & \vdots & \ddots & \vdots \\
 1 & g_{y_0,1}(x_{n}) &\dots & g_{y_0,n}(x_{n}) \\
 \end{vmatrix}
^{-1}
 \begin{vmatrix}
 0 & 1 & g_{y_0,1}(x) &\dots & g_{y_0,n}(x)\\
 t_ 0 & 1 & g_{y_0,1}(x_0) &\dots & g_{y_0,n}(x_0)\\
 t_ 1 & 1 & g_{y_0,1}(x_1) &\dots & g_{y_0,n}(x_1)\\
 \vdots &\vdots & \vdots & \ddots & \vdots \\
 t_n & 1 & g_{y_0,1}(x_{n}) &\dots & g_{y_0,n}(x_{n}) \\
 \end{vmatrix}
\]
solves the Lagrange interpolation problem~\eqref{ip}.
\end{pro}
\begin{proof}
For a proof of the determinant formula, see \cite[Theorem 2.5.2]{davis74}. By expanding the right hand side determinant by the first row we see $p$ is indeed a $g$\--polynomial.
\end{proof}

\subsection{Convex Lagrange interpolation}\label{subsec:cli}

One problem of Definition~\ref{def:roots} is that functions with roots of type II do not form a linear space. This suggests the following idea. Take $f\in\operatorname{UC}_g(\mathbb R,\mathbb R)$ such that $f$ has a root of type II on $t\in D_g$. Since $f(t)f(t^+)<0$, there exists $\alpha\in (0,1)$ such that
\begin{equation}
 \label{eq:lambdaevaluar}
 \alpha f(t) + (1-\alpha)f(t^+)=0.
\end{equation}
Conversely, if this is the case, then either $f$ has a root of type II on $t$ or two roots of type I at $t$ and $t^+$. Also, the set of functions that satisfy~\eqref{eq:lambdaevaluar} is a vector space. Note we can write equation~\eqref{eq:lambdaevaluar} as
$
(\alpha\operatorname{ev}_t + (1-\alpha)\operatorname{ev}_{t^+})(f)=0
$.
This suggests the following idea.
\begin{dfn}\label{dfng*}
 Let $\mathbb R_{++} = \mathbb R_+ \sqcup [(0,1)\times D_g] $ and define $g^* : \mathbb R_{++}\longrightarrow \mathbb R$ as
 \[
 g^*(x) =\begin{dcases}
 g(x), & x\in\mathbb R_+,\\
 \alpha g(t) + (1-\alpha)g(t^+), & x = (\alpha,t)\in(0,1)\times D_g.
 \end{dcases}
 \]
 Endow $\mathbb R_{++}$ with the initial topology $\tau_{g^*}$. Note that $(0,1)\times D_g$ is an open subset of $\mathbb R_{++}$ and therefore $\mathbb R_+$ is a closed subset.
\end{dfn}

It is easy to show that $g^*(\mathbb R_{++})$ is connected and therefore, $\mathbb R_{++}$ is also connected.
Define the mapping $F: \mathbb R_{++}\to\operatorname{R}(\mathbb R,\mathbb F)^*$ given by
\[ F(x) =
\left\{\begin{array}{ll}
 \operatorname{ev}_x, & x\in\mathbb R_+,\\
 \alpha\operatorname{ev}_t + (1-\alpha)\operatorname{ev}_{t^+}, & x=(\alpha,t)\in (0,1)\times D_g.
 \end{array}\right.
\]
 Analogous to Definition~\ref{def:Rmas}, given $a,b\in\mathbb R$, $a<b$, set
\[
[a,b]_{++} = [a,b]_+\sqcup [(0,1)\times (D_g\cap[a,b))]\subset \mathbb R_{++}.
\]
Hence, we can restrict $F$ to $[a,b]_{++}$ so for each $t\in[a,b]_{++}$, $F(t)$ can act on $\operatorname{R}([a,b],\mathbb F)$.
We are going to generalize Lagrange interpolation through the mapping $F$.

From now on, for $x\in\mathbb R_{++}$ and $f\in\operatorname{R}(\mathbb R,\mathbb F)$, we will abuse notation and write $
f(x) \equiv F(x)(f)$.
Thus, we can evaluate regulated functions at points of $\mathbb R_{++}$. This way, for all $x\in\mathbb R_{++}$,
\begin{equation}\label{eqF}
F(x)(f)= \alpha f(y)+(1-\alpha)f(y^+), \quad \forall f\in\operatorname{R}(\mathbb R,\mathbb F),
\end{equation}
for some $\alpha\in[0,1]$ and $y\in\mathbb R$. Notice the abuse of notation here, $y^+$ can be a point of $\mathbb R_+$. In this case, it is a notation for the limit from the right $
f(y^+) = \lim_{x\to y^+}f(x)$.
Evaluating at $y^+$ and taking the limit from the right coincide, as we showed on Section~\ref{sec:gdomain}.
\begin{dfn}
 Let $n\in\mathbb N$ and $X = (x_i)_{i=0}^n \subset \mathbb{R}_{++}$ a sequence of $n+1$ points on $\mathbb{R}_{++}$. We say $X$ is a \emph{valid sequence for interpolation} if:
 \vspace{-1.5ex}\begin{enumerate}[noitemsep]
 \item[I.] $\#g^*(X)\geq n+1$. In other words, $g^*(x_i)\neq g^*(x_j)$ for $i\neq j$.
 \item[II.] $\#(g^*(X)\cap [g(t),g(t^+)])\leq 2$ for all $t\in D_g$.
\end{enumerate}
\end{dfn}
It is worth noting that if $X = (x_i)_{i=0}^n \subset \mathbb R_+$, $X$ is a valid sequence for interpolation if and only if $X$ is a sequence of $n+1$ $g$\--distinct points.
\begin{dfn}[Convex Lagrange interpolation problem] For $n\in\mathbb N$, consider a valid sequence for interpolation of $n+1$ nodes $X = (x_i)_{i=0}^n \subset \mathbb{R}_{++}$ and any $n+1$ values $t_0,\dots,t_n\in\mathbb F$. We call the system of equations
	\begin{equation}\label{ipstar}F(x_i)(p)=t_i\,,\quad i\in\{0,\dots,n\},\end{equation}
	the \emph{convex Lagrange interpolation problem}, for which we try to find a \emph{polynomial solution} $p\in \operatorname{P}_n(\bF)$.
\end{dfn}

Analogues of Lemmas~\ref{lemequival} and~\ref{reduction} hold with identical proofs.

\begin{dfn}
	We are going to extend Definition~\ref{def:Vandermondemapping} in the following sense. Let $X = \{x_1,\dots,x_n\}\subset [a,b]_{++}$ and $A\ss \operatorname{R}([a,b],\bF)$ a subspace. We denote as $V^{X}_{A}\colon A\to \bF^{n}$ the mapping $V^{X}_{A}(p)=(F(x_1)(p),\dots,F(x_n)(p))$, for all $p\in A$.
\end{dfn}

\begin{rem}\label{rem:ipstar}
 Let $x,y\in\mathbb R_{++}$ such that $g^*(x),g^*(y)\in [g(t),g(t^+)]$ for some $t\in D_g$ with $g^*(x)\neq g^*(y)$. Hence there exist $\alpha,\beta\in[0,1]$, $\alpha\neq \beta$ such that
 \[ \begin{array}{l}
 F(x)(f) = \alpha f(t)+(1-\alpha)f(t^+),\\
 F(y)(f) = \beta f(t)+(1-\beta)f(t^+),
 \end{array}
 \]
 for all $f\in\operatorname{UC}_g(\mathbb R,\mathbb F)$. Recall $F(t)(f)=f(t)$ and $F(t^+)(f)=f(t^+)$ so the above reads
\[ \begin{pmatrix}
\alpha & 1-\alpha \\
\beta & 1-\beta
\end{pmatrix}
 \begin{pmatrix}
F(t)(f) \\
F(t^+)(f)
\end{pmatrix}
=
\begin{pmatrix}
F(x)(f)\\
F(y)(f)
\end{pmatrix}.
\]
This implies that we can write $\{F(t),F(t^+)\}$ as a linear combination of $\{F(x),F(y)\}$ and vice versa, since the matrix above has determinant $\a-\b\ne 0$.

Let $X = (x_i)_{i=0}^n \subset \mathbb{R}_{++}$ be a sequence of $n+1$ points such that $\#(g^*(X)\cap [g(t),g(t^+)])= 2$ for some $t\in D_g$. Take $i,j\in\{0,\dots,n\}$ such that $g^*(X)\cap [g(t),g(t^+)]=\{g^*(x_i),g^*(x_j)\}$. Then, $\smash{V^{X}_{\operatorname{P}_n(\bF)}}$ is an isomorphism if and only if $\smash{V^{X'}_{\operatorname{P}_n(\bF)}}$ is, where $X'= (x_k')_{k=0}^n$ with
\[
x_k'=\begin{dcases}
 x_k, & k\neq i,j,\\
 t,& k=i,\\
 t^+,& k=j.
\end{dcases}
\]
Hence, if $X$ is such that $\#(g^*(X)\cap [g(t),g(t^+)])= 2$ we can assume the intersection is precisely $\{g(t),g(t^+)\}$. Besides, $X$ is a valid sequence for interpolation if and only if $X'$ is as well. This also explains why problem~\eqref{ipstar} does not have a solution in general (and if it does it is not unique) if $X$ is a sequence of nodes such that $\#(g^*(X)\cap [g(t),g(t^+)])> 2$ for some $t\in D_g$.
\end{rem}

\begin{thm}[Convex Lagrange interpolation]\label{th:cli}
 Let $n\in\mathbb N$ and a valid sequence for interpolation of $n+1$ nodes $X = (x_i)_{i=0}^n \subset \mathbb{R}_{++}$. Then, problem~\eqref{ipstar} has a unique solution.

 Furthermore, let $y_0\in\mathbb R$, if the sequence $(x_i)_{i=0}^n$ is ordered from left to right, that is, satisfies $g^*(x_i)< g^*(x_j)$ for $i<j$, then
 \[|[V^{X}_{\operatorname{P}_n}\circ\operatorname{L}_{y_0}]| = \begin{vmatrix}
 1 & g_{y_0,1}(x_0) &\dots & g_{y_0,n}(x_0)\\
 1 & g_{y_0,1}(x_1) &\dots & g_{y_0,n}(x_1)\\
 \vdots & \vdots & \ddots & \vdots \\
 1 & g_{y_0,1}(x_n) &\dots & g_{y_0,n}(x_n)
\end{vmatrix}>0.
 \]
\end{thm}
\begin{proof}
 Take $y_0\in\mathbb R$. We are going to show that $V^{X}_{\operatorname{P}_n(\bF)}\circ \operatorname{L}_{y_0}$ is an isomorphism. Let $\lambda\in\mathbb C^{n+1}$ such that $\lambda\in\operatorname{ker}V^{X}_{\operatorname{P}_n(\bF)}\circ\operatorname{L}_{y_0}$. As in Theorem~\ref{pith}, by taking real and imaginary parts, we get $\operatorname{L}_{y_0}(\operatorname{Re}\lambda),\operatorname{L}_{y_0}(\operatorname{Im}\lambda)\in\operatorname{ker}V^{X}_{\operatorname{P}_n(\bF)}$. It suffices to assume that $\lambda\in\mathbb R^{n+1}$. Take $i\in\{0,\dots,n\}$, we have that $
 F(x_i)(\operatorname{L}_{y_0}(\lambda))=0$.
 If $x_i\in\mathbb R_+$, $\operatorname{L}_{y_0}(\lambda)$ just has a root of type I on $x_i$. If $x_i\notin\mathbb R_+$, then $x_i = (\alpha,t)\in (0,1)\times D_g$. We consider two cases:

 \noindent $\bullet$\emph{ Case 1: $\#(g^*(X)\cap [g(t),g(t^+)]) = 2$}. Following Remark~\ref{rem:ipstar}, we get that $\operatorname{L}_{y_0}(\lambda)$ has two roots of type I at $t$ and $t^+$.

 \noindent $\bullet$ \emph{Case 2: $\#(g^*(X)\cap [g(t),g(t^+)]) = 1$}. Then, $
 \alpha \operatorname{L}_{y_0}(\lambda)(t)+(1-\alpha)\operatorname{L}_{y_0}(\lambda)(t^+)=0$.
 As we said at the start of the section, in this case, either $\operatorname{L}_{y_0}(\lambda)$ has a root of type II on $t$ or two roots of type I on $t$ and $t^+$.

Therefore, $\operatorname{L}_{y_0}(\lambda)$ has at least $n+1$ $g$\--distinct roots and by Theorem~\ref{rootsofpol}, $\lambda=0$.

 Take $i\in\{0,\dots,k\}$ such that $x_i\in\bR_{++}\setminus \bR_{+}$. Thus, $x_i = (\alpha_i,t_i)\in(0,1)\times D_g$. Therefore, the $i$--th row of $[V^{X}_{\operatorname{P}_n}\circ\operatorname{L}_{y_0}]$ can be written as the convex combination
\[\begin{pmatrix}
1 \\ g_{y_0,1}(x_i) \\\vdots \\ g_{y_0,n}(x_i)
\end{pmatrix}= \alpha_i \begin{pmatrix}
1 \\ g_{y_0,1}(t_i) \\\vdots \\ g_{y_0,n}(t_i)
\end{pmatrix}+ (1-\alpha_i)\begin{pmatrix}
1 \\ g_{y_0,1}(t_i^+) \\\vdots \\ g_{y_0,n}(t_i^+)
\end{pmatrix},
\]
where the  $g$\--monomials are evaluated at points of $\bR_+$. Hence, applying the multilinearity of the determinant, $|[V^{X}_{\operatorname{P}_n}\circ\operatorname{L}_{y_0}]|$ is a conical combination (linear combination with nonnegative scalars) of determinants of the type of $|[V^{X'}_{\operatorname{P}_n}\circ \operatorname{L}_{y_0}]|$ where $X'\ss \bR_+$ and the sequence $X'$ is ordered from left to right. From Theorem~\ref{pith}, $|[V^{X'}_{\operatorname{P}_n}\circ \operatorname{L}_{y_0}]|\geq 0$. Hence, $|[V^{X}_{\operatorname{P}_n}\circ\operatorname{L}_{y_0}]|\geq 0$. Since it must be nonzero, it is positive.
\end{proof}

\subsection{Interpolatory projections}

Let $n\in\mathbb N$ and $X = (x_i)_{i=0}^n \subset \mathbb{R}_{++}$ be a valid sequence for interpolation. Through the mapping $F$, $X$ induces a subset of linear forms over $\operatorname{R}(\mathbb R,\mathbb F)$. Thus, we have the following linear mapping $F_X:\operatorname{R}(\mathbb R,\mathbb F)\to \mathbb{F}^{n+1} $ such that $F_X(f)=(F(x_0)(f),\dots,F(x_n)(f))$ for $f\in \operatorname{R}(\mathbb R,\mathbb F)$.
Note that $\restr{F_X}{\operatorname{P}_n} = V^{X}_{\operatorname{P}_n(\bF)}$. Furthermore, if $X\ss [a,b]_{++}$ then $F_X$ may be defined over the Banach space $\operatorname{R}([a,b],\mathbb F)$. Note that in this case $F_X$ is bounded. Thus, we can define the following.

\begin{dfn}
 \label{defL}
 Let $n\in\mathbb N$ and $X = (x_i)_{i=0}^n \subset \mathbb{R}_{++}$ be a valid sequence for interpolation. Define $	\operatorname{I}_n:\operatorname{R}(\mathbb R,\mathbb F)\to \operatorname{R}(\mathbb R,\mathbb F) $ such that, for $f\in \operatorname{R}(\mathbb R,\mathbb F)$,
$\operatorname{I}_n(f)$ is the unique $g$\--polynomial of degree at most $n$ that solves the convex Lagrange interpolation problem
$
F(x_j)(p)=F(x_j)(f)$, $j\in\{0,\dots,n\}
$.
Define also $\operatorname{R}_n:\operatorname{R}(\mathbb R,\mathbb F)\to \operatorname{R}(\mathbb R,\mathbb F)$ given by $\operatorname{R}_n(f)=f-\operatorname{I}_n(f)$ for $f\in \operatorname{R}(\mathbb R,\mathbb F)$.
\end{dfn}

\begin{pro}
 \label{projectionmappingpol}
$\operatorname{I}_n$ is a projection of $\operatorname{R}(\bR,\operatorname F)$
onto $\operatorname{P}_n$. If $X\ss [a,b]_{++}$, then $\operatorname{I}_n:\operatorname{R}([a,b],\mathbb F)\to\operatorname{P}_n$ is bounded.
\end{pro}
\begin{proof}
By Theorem~\ref{th:cli}, $V^{X}_{\operatorname{P}_n(\bF)}$ is an isomorphism. We have that $\operatorname{I}_n = (V^{X}_{\operatorname{P}_n(\bF)})^{-1}\circ F_X$. Therefore, if $X\ss [a,b]_{++}$ and we consider the restriction of $\operatorname{I}_n$ to $\operatorname{R}([a,b],\mathbb F)$, $\operatorname{I}_n$ is a bounded operator. Moreover, $\restr{\operatorname{I}_n}{\operatorname{P}_n} = (V^{X}_{\operatorname{P}_n(\bF)})^{-1}\circ \restr{F_X}{\operatorname{P}_n}= \restr{\operatorname{Id}}{\operatorname{P}_n}$.
Hence, $\operatorname{I}_n$ is a projection.
\end{proof}

Of course, the operators $\operatorname{I}_n$ and $\operatorname{R}_n$ could be restricted to $\operatorname{UC}_g(\mathbb R,\bF)$ or $\operatorname{UC}_g([a,b],\bF)$ if $X\subset [a,b]_{++}$.
Even more, if $X\subset K$ with $K$ a compact subset of $\bR_{++}$ then $\operatorname{I}_n$ is a continuous projection of $\operatorname{C}(K,\mathbb F)$ onto $\operatorname{P}_n$.
In fact, if $X\ss\bR$, $\operatorname{I}_n$ and $\operatorname{R}_n$ can be extended to the Banach space of bounded $\mathbb{F}$\--valued functions, while Proposition~\ref{projectionmappingpol} remains valid.

\section{Best approximation, equioscillation and T\--Systems}
\label{sec:chebishev}
Building on the results obtained in \cite{weier}, which provide a generalization of the Weierstrass Approximation Theorem for derivators with finitely many discontinuities, we aim to recover, in the Stieltjes setting, two fundamental approximation properties of classical polynomials: the uniqueness of the best uniform approximant (see \cite[Theorem~7.5.6]{davis74} or \cite[\S~6.3\--5]{kreyszig}) and the Chebyshev Equioscillation Theorem \cite[Theorem~7.6.2]{davis74}.

\begin{dfn}
 Let $Y$ be a normed space and $X$ a vector subspace of $Y$. We say that $X$ is a \emph{Chebyshev space} if, for $v\in Y$, there exists a unique $x\in X$ such that $
 \left\lVert v-x\right\rVert_{Y} = \operatorname*{inf}_{y\in X} \left\lVert v-y\right\rVert_Y$.
\end{dfn}

Thus, for a Chebyshev space we need both existence and uniqueness for the minimizer of the distance. The following guarantees the existence for finite dimensional subspaces.
\begin{pro}[{\cite[\S~6.1-1]{kreyszig}}]
 Let $ Y$ be a normed space and $X$ a finite dimensional subspace of $Y$. Then, for all $v\in Y$, there exists $x\in X$ such that $
 \left\lVert v-x\right\rVert_{ Y} = \operatorname*{inf}_{y\in X} \left\lVert v-y\right\rVert_Y $.
\end{pro}

Under strictly convex norms any best approximant out of a given subspace must be unique \cite[\S\ 6.2\--3]{kreyszig} and, thus, any finite dimensional space is a Chebyshev space. However, the uniform norm is not strictly convex, and we cannot proceed in this manner. It turns out that, under the uniform norm, finite dimensional Chebyshev spaces are characterized.

\begin{dfn}\label{def:haarspace}
 Let $Q$ be a Hausdorff topological space and $ Y\ss \operatorname{C}(Q,\mathbb F)$ a vector subspace of dimension $n\in\bN$ of the space of $\mathbb F$\--valued continuous functions. We say that $Y$ is a \emph{Haar space} if for any given set of $n$ distinct points $T\subset Q$, the Vandermonde mapping $V_Y^T$
 is an isomorphism.
\end{dfn}
Observe that $ Y$ is a Haar space if and only if every nonzero function on $ Y$ has at most $n-1$ zeros. The existence of a Haar space of dimension $n\in\mathbb N$ implies $\# Q\geq n$. If $Q$ is compact, $\operatorname{C}(Q,\mathbb F)$ is a Banach space under the supremum norm.

\begin{thm}[Haar's Theorem]\label{th:haar}
 Let $Q$ be a compact Hausdorff space and $ Y$ a finite dimensional space of $\operatorname{C}(Q,\mathbb F)$. Then, $ Y$ is a Chebyshev space if and only if it is a Haar space.
\end{thm}

The above theorem is due to A. Haar \cite{haar1918minkowskische}. See also \cite[Theorem 2.7]{alimov2021geometric} and \cite[Section 43]{achieser1992theory}. For the special case, $Q=[a,b]\subset \mathbb R$ and $\mathbb F=\mathbb R$, see \cite[Theorem 2.6]{alimov2021geometric} or \cite[\S~6.3\--4]{kreyszig}. The general case follows from \cite[Section 4]{cheney1966introduction}.

The existence of Chebyshev subspaces imposes huge restrictions on the topology of $Q$. In fact, the following holds.

\begin{thm}[Mairhuber's Theorem {\cite[Theorem 2.16]{alimov2021geometric}}]
Let $Q$ be a compact metrizable space. Assume $\operatorname{C}(Q,\mathbb R)$ contains a Chebyshev space of dimension greater than two, then $Q$ is homeomorphic to a subset of the unit sphere $\mathbb S^1$.
\end{thm}

We need to define Haar spaces on the Stieltjes setting.
\begin{dfn}
Let $X\subset \mathbb R_+$ and $ Y\subset \operatorname{UC}_g(X,\mathbb F)$ be a vector space of finite dimension $n\in\bN$. We say $ Y$ is a \emph{$g$\--Haar space} if, given a set of $n$ $g$\--distinct points $T\subset X$, the 
Vandermonde mapping $V_Y^T$
is an isomorphism.
\end{dfn}
For instance, by Theorem~\ref{pith}, $\operatorname{P}_{n-1}$ is a $g$\--Haar space, see Lemma~\ref{lem:pnghaar}. The above definition is almost identical to Definition~\ref{def:haarspace} however, since $X$ is not, in general, Hausdorff we must impose the points to be distinguishable by the $g$\--topology.

\begin{thm}
 \label{haaruniqueness}
 Let $X$ be a compact subset of $\mathbb R_+$ and $ Y\subset \operatorname{C}_g(X,\mathbb F)$ a finite dimensional vector space. Then, $ Y$ is a Chebyshev space if and only if it is a $g$\--Haar space.
\end{thm}
\begin{proof}
 Since $X$ is compact and $g$ is continuous on $X$, $g(X)$ is a compact Hausdorff subset of $\mathbb R$. By  Corollary~\ref{isomgcontandtododemas}, the map $I:f\in\operatorname{C}_{\operatorname{id}}(g(X),\mathbb F) \to f\circ g\in \operatorname{C}_g(X,\mathbb F)$ is an isometric isomorphism. Let $ Y':= I^{-1}( Y)$, $n=\dim Y$. Since $I$ is isometric, $ Y'$ is a Chebyshev space if and only if $ Y$ is a Chebyshev space. Let us show that $ Y'$ is a Haar space if and only if $ Y$ is a $g$\--Haar space. $\# g(X)\geq n$, so we have a correspondence between sets of $n$ distinct points $T'=(t_i)_{i=1}^n\subset g(X)$ and sets of $n$ $g$\--distinct points $T=(y_i)_{i=1}^n\subset X$ such that $g(y_i)=t_i$, $1\leq i\leq n$.
It is easy to show that $V_Y^T\circ \restr{I}{ Y'} = V_{Y'}^{T'}$ and $V_{Y'}^{T'}\circ \restr{I^{-1}}{ Y} = V_Y^T$. Thus, $V_Y^T$ is an isomorphism if and only if $V_{Y'}^{T'}$ is an isomorphism. The result now follows from Theorem~\ref{th:haar}.
\end{proof}

We can relax the compactness hypothesis in the following way.

\begin{cor}\label{cor:relativecompact}
 Let $X\subset \mathbb R_+$ be a relatively compact subset and $ Y\subset \operatorname{UC}_g(X,\mathbb F)$ a finite dimensional vector space. Then, $ Y$ is a Chebyshev space if and only if it is a $g$\--Haar space on $\overline{X}$. Here $\overline{X}$ denotes the closure of $X$ on the $g$\--topology.
\end{cor}
\begin{proof}
 We use that $\operatorname{UC}_g(X,\mathbb F)$ is isometrically isomorphic to $\operatorname{C}_g\left(\overline{X},\mathbb F\right)$ by Theorem~\ref{extensionpseudometrix} (see Corollary~\ref{isomgcontandtododemas}) and then apply Theorem~\ref{haaruniqueness}.
\end{proof}

Note that since $[a,b]_+$ is compact, we can apply Theorem~\ref{haaruniqueness} to any closed subset $X$ of $[a,b]_+$. We make the following remark.

\begin{lem}\label{lem:pnghaar}
 Let $n\in\mathbb{N}$. Then, $\operatorname{P}_n$ is a $g$\--Haar space over $\mathbb R_+$, and thus, over any $X\subset \mathbb{R}_+$.
\end{lem}
\begin{proof}
 If $\#g(\mathbb R_+)\geq n+1$, the result follows from Theorem~\ref{pith}. If $\#g(\mathbb R_+)\leq n$, then $\operatorname{P}_n = \operatorname{P}_m$ where $m +1 = \# g(\mathbb R_+)$. Apply Theorem~\ref{pith} again, also, in this case, we have $\operatorname{P}_m = \operatorname{UC}_g( \mathbb R_+,\mathbb F)$.
\end{proof}

By applying Theorem~\ref{haaruniqueness} to $[a,b]_+$, or Corollary~\ref{cor:relativecompact} to $[a,b]$, to get the following corollary.

\begin{cor}[Uniqueness of best approximating Stieltjes polynomial]
 Let $n\in\mathbb{N}$. For any uniformly $g$\--continuous function $f\in\operatorname{UC}_g([a,b],\mathbb F)$ there exists a unique Stieltjes polynomial $p\in \operatorname{P}_n$ such that
 \[
 \left\lVert f-p\right\rVert_\infty = \inf_{q\in\operatorname{P}_n}\left\lVert f-q\right\rVert_\infty.
 \]
\end{cor}

We have recovered the uniqueness of the best uniform polynomial approximant. We will recover now the Chebyshev Equioscillation Theorem.

\begin{dfn}
 Let $Q$ be a set and $ Y$ a finite dimensional subspace of $\mathbb R$\--valued functions defined over $Q$. Fix $n = \operatorname{dim} Y\in\mathbb N$. We say that $ Y$ is a \emph{T\--system} or \emph{Chebyshev system} if for any set of $n$ distinct points $S \subset Q$
 the Vandermonde mapping $V_Y^S$ is an isomorphism.
\end{dfn}

Hence, Haar spaces are just continuous T\--systems over Hausdorff spaces. See \cite{didio2024introductiontsystemsspecial} and \cite{karlin1966tchebycheff} for an extensive study on T\--systems and their properties. Continuous T\--systems over $[a,b]$ with the usual topology (that is, Haar spaces over compact intervals) have very unique behaviour see \cite[Chapters 4, 7 and 8]{didio2024introductiontsystemsspecial}. In fact,
%
%
exploiting the properties of continuous T\--systems over compact intervals a stronger version of Haar's Theorem can be proven.

\begin{thm}[Alternation Theorem]\label{th:alternation}
 Let $ Y\ss \operatorname{C}_{\operatorname{id}}([a,b],\mathbb R)$ be an $n$\--dimensional Chebyshev subspace. Take $f\in\operatorname{C}_{\operatorname{id}}([a,b],\mathbb R)$ and $v\in Y$. Then, the following statements are equivalent:
 \begin{enumerate}[noitemsep]
 \item[I.] $v$ is the best approximation of $f$ in $ Y$, that is,
$
 \left\lVert f-v\right\rVert_{\infty} = \inf_{v'\in Y}\left\lVert f-v'\right\rVert_{\infty}$.
 \item[II.] There exist $t_1<\cdots< t_{n+1}$ in $[a,b]$ such that
 \begin{align*}
	|f(t_i)-v(t_i)| = & \left\lVert f-v\right\rVert_\infty, \quad i\in\{1,\dots,n+1\},\\
	f(t_i)-v(t_i) = & -(f(t_{i+1})-v(t_{i+1})), \quad i\in\{1,\dots,n\}.
\end{align*}
 \end{enumerate}
\end{thm}

Theorem~\ref{th:alternation}, when stated for $Y=\operatorname{P}_n$, is usually referred as the \emph{Chebyshev Equioscillation theorem}.

The exact statement of Theorem~\ref{th:alternation} appears on \cite[Theorem 2.8]{alimov2021geometric}. For a proof see \cite[Section 48]{achieser1992theory} or \cite[Chapter 9, Theorem 1.1]{karlin1966tchebycheff}. A more general statement appears on \cite[Section 4, Alternation Theorem]{cheney1966introduction} which proves Theorem~\ref{th:alternation} for the restriction of continuous T\--systems on $[a,b]$ over closed subsets of $[a,b]$.

We could try to use \cite[Section 4, Alternation Theorem]{cheney1966introduction} to get to an Alternation theorem for the Stieltjes setting.
If $X\subset \mathbb R_+$ is compact, $\operatorname{UC}_g(X,\mathbb R) = \operatorname{C}_{\operatorname{id}}(g(X),\mathbb R)$. Then, we could insert $g(X)$ in the smallest interval $[a,b]$ that contains $g(X)$. Therefore, we could think of $\operatorname{C}_{\operatorname{id}}(g(X),\mathbb R)$ as a closed subspace of $\operatorname{C}_{\operatorname{id}}([a,b],\mathbb R)$ by extending the functions defined on $g(X)$ linearly to $[a,b]$ by the Tietze extension theorem. It would be left to show that the image of $\operatorname{P}_n$ through this extension is a continuous T\--system over $[a,b]$ and finally apply \cite[Section 4, Alternation Theorem]{cheney1966introduction} to the closed subset of $[a,b]$, $g(X)$.

However, the image of $\operatorname{P}_n$ is not, in general, a T\--system over $[a,b]$, Section~\ref{subsec:cli} hints at why. The extension to $[a,b]$ of some function defined over $g(X)$ is just affine on the connected components of $[a,b]\setminus g(X)$ and thus, a family of these functions cannot interpolate on an arbitrary set of three or more points.


Hence, we need to give an intrinsic characterization of the subspaces of $\operatorname{C}_{\operatorname{id}}(Q,\mathbb R)$ that satisfy the alternation theorem where $Q$ is any \emph{usual} (that is, with the usual topology) compact subset of $\mathbb R$. This is precisely what the authors do in \cite{deutsch1980weak}.

\begin{dfn}\label{def:weakvheby}
 Let $X\subset \mathbb R$ be a usual compact set and $ Y\subset \operatorname{C}_{\operatorname{id}}(X,\mathbb R)$ vector subspace of dimension $\n$. We say $ Y$ is a \emph{weak Chebyshev} subspace if every $v\in Y$ has at most $n-1$ sign changes, i.e., there do not exist $t_1<\cdots< t_{n+1}$ in $X$ such that
 $
 v(t_i)v(t_{i+1})<0$ for $i\in\{1,\dots,n\}$.
\end{dfn}

This property is denoted by (W\--4) in \cite{deutsch1980weak}, where it is proven equivalent to other criteria. The term weak Chebyshev comes from the fact that if $X=[a,b]$, then any finite dimensional Chebyshev subspace is always weak Chebyshev. This is not true in general, see \cite[Example~3.3]{deutsch1980weak}.

\begin{lem}\label{lem:weakchebysolsigneduniq}
 Let $X\subset \mathbb R$ be a usual compact set and $ Y\subset \operatorname{C}_{\operatorname{id}}(X,\mathbb R)$ a weak Chebyshev subspace of dimension $n\in\mathbb N$. Let $f\in \operatorname{C}_{\operatorname{id}}(X,\mathbb R)$, $v\in Y$ and suppose there do exist $t_1<\cdots< t_{n+1}$ in $X$ such that
 \begin{align*}
 	|f(t_i)-v(t_i)| = & \left\lVert f-v\right\rVert_\infty, \quad i\in\{1,\dots,n+1\},\\
 	f(t_i)-v(t_i) = & -(f(t_{i+1})-v(t_{i+1})), \quad i\in\{1,\dots,n\}.
 \end{align*}
 Then, $\left\lVert f-v\right\rVert_\infty = \inf_{w\in V}\left\lVert f-w\right\rVert_\infty$.
\end{lem}
\begin{proof}
 Assume there exists $v'\in V$ such that $\left\lVert f-v'\right\rVert_\infty < \left\lVert f-v\right\rVert_\infty$. Thus, for all $i\in\{1,\dots,n+1\}$,
 \[
 |f(t_i)-v'(t_i)|\leq \left\lVert f-v'\right\rVert_\infty < \left\lVert f-v\right\rVert_\infty = |f(t_i)-v(t_i)|.
 \]
 Therefore, $v'-v = (f-v)-(f-v')$ has the same sign at $t_i$ as $f-v$. Hence,
 \[
 (v'-v)(t_i)(v'-v)(t_{i+1})<0,\quad i\in\{1,\dots,n\}.
 \]
 However, this is a contradiction since $v'-v\in Y$ and $ Y$ is a weak Chebyshev subspace.
\end{proof}

Definition~\ref{def:weakvheby} has its equivalent on the Stieltjes setting.

\begin{dfn}
 Let $X$ a compact subset of $\mathbb R_+$ and $ Y\subset \operatorname{C}_g(X,\mathbb R)$ a finite dimensional subspace. Take as $n = \operatorname{dim} Y$. We say $ Y$ is a \emph{$g$\--weak Chebyshev} subspace if every $v\in Y$ has at most $n-1$ sign changes, that is, there do not exist $t_1<_g\cdots<_g t_{n+1}$ in $X$ such that
 $
 v(t_i)v(t_{i+1})<0$ for $i\in\{1,\dots,n\}$.
\end{dfn}

Note that $\operatorname{P}_n$ is always $g$\--weak Chebyshev for any compact subset of $\mathbb{R}_+$ and all $n\in\mathbb N$ thanks to Corollary~\ref{cor:signchanges}. The following simple result relates the concepts of weak Chebyshev and $g$\--weak Chebyshev.

\begin{lem}\label{lem:weakchebyequiv}
 Let $X$ be a compact subset of $\mathbb{R}_+$ and consider the isometric isomorphism $I:f\in\operatorname{C}_{\operatorname{id}}(g(X),\mathbb F) \to f\circ g\in \operatorname{C}_g(X,\mathbb F)$.
Let $ Y\subset \operatorname{C}_g(X,\mathbb R)$ be a finite dimensional subspace and $ Y' = I^{-1}( Y)$. Then, $ Y$ is $g$\--weak Chebyshev if and only if $ Y'$ is weak Chebyshev.
\end{lem}

In \cite{deutsch1980weak}, the following characterization of the alternation theorem was obtained. The result is stated in a broader setting, where the compactness assumption is replaced by the weaker hypothesis that $X\subset \mathbb{R}$ is locally compact.

\begin{thm}[{\cite[Theorem 4.1]{deutsch1980weak}}]\label{th:delahostiaweakalternation}
 Let $X$ be a usual compact subset of $\mathbb{R}$, $n\in\mathbb N$ and $ Y\subset \operatorname{C}_{\operatorname{id}}(X,\mathbb R)$ an $n$\--dimensional subspace. Suppose $\#X\geq n+1$. The following is equivalent:
 \begin{enumerate}[noitemsep]
 \item[I.] $ Y$ is weak Chebyshev.
 \item[II.] For all $f\in \operatorname{C}_{\operatorname{id}}(X,\mathbb R)$, there exists a best approximation $v\in Y$ to $f$ such that there exist $t_1<\cdots< t_{n+1}$ in $X$ satisfying
 \begin{align*}
 |f(t_i)-v(t_i)| = & \left\lVert f-v\right\rVert_\infty, \quad i\in\{1,\dots,n+1\},\\
 f(t_i)-v(t_i) = & -(f(t_{i+1})-v(t_{i+1})), \quad i\in\{1,\dots,n\}.
 \end{align*}
 \end{enumerate}
\end{thm}

\begin{cor}\label{cor:gweakcheby}
 Let $X$ be a compact subset of $\mathbb{R}_+$, $n\in\mathbb N$ and $ Y\subset \operatorname{C}_{g}(X,\mathbb R)$ an $n$\--dimensional subspace. Suppose $\#g(X)\geq n+1$. The following is equivalent:
 \begin{enumerate}[noitemsep]
 \item[I.] $ Y$ is $g$\--weak Chebyshev.
 \item[II.] For all $f\in \operatorname{C}_{g}(X,\mathbb R)$, there exists a best approximation $v\in Y$ to $f$ such that there exist $t_1<_g\cdots<_g t_{n+1}$ in $X$ satisfying
 \begin{align*}
 	|f(t_i)-v(t_i)| = & \left\lVert f-v\right\rVert_\infty, \quad i\in\{1,\dots,n+1\},\\
 	f(t_i)-v(t_i) = & -(f(t_{i+1})-v(t_{i+1})), \quad i\in\{1,\dots,n\}.
 \end{align*}
 \end{enumerate}
\end{cor}
\begin{proof}
 It is enough to use Lemma~\ref{lem:weakchebyequiv} and Theorem~\ref{th:delahostiaweakalternation}.
\end{proof}

We finally arrive at the generalization of the alternation theorem for the Stieltjes setting.

\begin{thm}\label{th:chebymasweakchebyforg}
 Let $X$ be a compact subset of $\mathbb{R}_+$, $n\in\mathbb N$ and $ Y\subset \operatorname{C}_{g}(X,\mathbb R)$ an $n$\--dimensional subspace. Suppose $\#g(X)\geq n+1$. Assume $ Y$ is a $g$\--Haar space (a Chebyshev space by Theorem~\ref{haaruniqueness}) and $g$\--weak Chebyshev. Take $f\in\operatorname{C}_{g}(X,\mathbb R)$ and $v\in Y$. Then, the following is equivalent:
 \begin{enumerate}[noitemsep]
 \item[I.] $v$ is the best approximation of $f$ in $ Y$, that is, $
 \left\lVert f-v\right\rVert_{\infty} = \inf_{w\in Y}\left\lVert f-w\right\rVert_{\infty}$.
 \item[II.] There exist $t_1<_g\cdots<_g t_{n+1}$ in $X$ such that
 \begin{align*}
 	|f(t_i)-v(t_i)| = & \left\lVert f-v\right\rVert_\infty, \quad i\in\{1,\dots,n+1\},\\
 	f(t_i)-v(t_i) = & -(f(t_{i+1})-v(t_{i+1})), \quad i\in\{1,\dots,n\}.
 \end{align*}
 \end{enumerate}
\end{thm}
\begin{proof}
I implies II from Corollary~\ref{cor:gweakcheby} and the uniqueness of best approximant, given that $Y$ is Chebyshev. II implies I from Lemma~\ref{lem:weakchebysolsigneduniq}.
\end{proof}

The reader can find a characterization of the finite dimensional subspaces of $\operatorname{C}_{\operatorname{id}}(X,\mathbb R)$ that are Chebyshev and weak Chebyshev in \cite[Theorem 5.1]{deutsch1980weak}.

\begin{rem}
 In the hypotheses of Theorem~\ref{th:chebymasweakchebyforg}, if $ Y\subset \operatorname{C}_{g}(X,\mathbb R)$ is an $n$\--dimensional subspace, then necessarily $\#g(X)\geq n$. Even more, if $\#g(X) = n$, then $ Y = \operatorname{C}_{g}(X,\mathbb R)$.
\end{rem}

Since for all $n\in\mathbb N$, $\operatorname{P}_n$ is simultaneously Chebyshev and weak Chebyshev, we have the following result.

\begin{thm}[Chebyshev Equioscillation theorem]\label{th:tschebyequiostheorem}
 Let $X$ be a compact subset of $\mathbb{R}_+$ and $n\in\mathbb N$. Suppose $\#g(X)\geq n+2$. Take $f\in\operatorname{C}_{g}(X,\mathbb R)$ and $p\in \operatorname{P}_n$. Then, the following statements are equivalent:
 \begin{enumerate}[noitemsep]
 \item[I.] $p$ is the best approximation of $f$ in $\operatorname{P}_n$, that is, $
 \left\lVert f-p\right\rVert_{\infty} = \inf_{q\in \operatorname{P}_n}\left\lVert f-q\right\rVert_{\infty}$.
 \item[II.] There exist $t_1<_g\cdots<_g t_{n+2}$ in $X$ such that
 \begin{align*}
 	|f(t_i)-p(t_i)| = & \left\lVert f-p\right\rVert_\infty, \quad i\in\{1,\dots,n+2\},\\
 	f(t_i)-p(t_i) = & -(f(t_{i+1})-p(t_{i+1})), \quad i\in\{1,\dots,n+1\}.
 \end{align*}
 \end{enumerate}
\end{thm}

\section{Divided differences}\label{sec:divdiff}

In this section our first objective will be to show that there is always a monic $g$\--polynomial with prescribed  $g$\--distinct roots. This will allow us to define the Newton polynomials in the Stieltjes setting and, as a consequence, divided differences. We begin with the following definition.

\begin{dfn}
 Take $n\in\mathbb N$ and assume $\# g(\mathbb R)\geq n+1$. Let $p\in\operatorname{P}_n$ and $\lambda\in\mathbb F^{n+1}$. We say $p$ is \emph{monic} if $\lambda_n=1$ and
 $p=\operatorname{L}_{x_0}(\lambda)$ for some $x_0\in\mathbb R$.
\end{dfn}
\begin{rem}
We assume $\# g(\mathbb R)\geq n+1$, so that $\operatorname{L}_{x_0}$ is an isomorphism. Otherwise, a polynomial could have multiple expressions in which one of them is not monic. For instance, $g_{x_0,1}=0$ in Example~\ref{exa:gconstant}.

The above definition does not depend on the center $x_0$. Indeed, if $x_1\in\mathbb R$ and $
\operatorname{M}_{x_0\to x_1}(\lambda)=\beta
$,
then $\beta_n=1$, as $[\operatorname{M}_{x_0\to x_1}]$ is an upper triangular matrix with ones on the diagonal.
\end{rem}
\begin{lem}\label{lemmon}
 Let $n\in\mathbb{N}$ and $X=(x_i)_{i=1}^n\subset \mathbb R_{++}$ be a valid sequence for interpolation. Assume $\#g(\mathbb R)\geq n+1$. Then, there exists a unique monic polynomial $p\in\operatorname{P}_n$ such that
 $
 F(x_i)(p) = 0$ for $i=1,\dots,n$.
\end{lem}
\begin{proof}
 Since $\#g(\mathbb R)\ge n+1$, it is possible to add an additional node $x_{n+1}\in\mathbb R_{++}$ to $X$, that is $X' = X\cup \{x_{n+1}\}$, such that $X'$ is a valid sequence for interpolation. Take $q\in\operatorname{P}_{n}$ the unique solution to the convex Lagrange interpolation problem
 \begin{equation}\label{eq:dmlv1}
 F(x_j)(q)=\delta_{j}^{n+1}, \qquad j\in\{1,\dots,n+1\}.
 \end{equation}
 Let $y_0\in\mathbb R$ and $\lambda\in \mathbb R^{n+1}$ such that $\operatorname{L}_{y_0}(\lambda) = q$. Then, $\lambda_n\neq 0$. Indeed, if not, $q\in\operatorname{P}_{n-1}$ and the zero polynomial is the only polynomial of degree at most $n-1$ that satisfies the first $n$ equations of~\eqref{eq:dmlv1}. Hence, $q=0$, but that contradicts $F(x_{n+1})(q)=1$. Thus, the monic polynomial $p=\frac{1}{\lambda_{n}}q\in\operatorname{P}_{n}$ satisfies $
 F(x_i)(p) = 0$ for $i=1,\dots,n$.

 We have shown such a polynomial exists, let us prove it is unique. Let $p_1,p_2\in\operatorname{P}_n$ be two monic polynomials such that $
 F(x_i)(p_1)=F(x_i)(p_2) = 0$ for $i=1,\dots,n$. Take $\lambda,\beta\in \mathbb F^{n+1}$ such that $\operatorname{L}_{y_0}(\lambda) = p_1$ and $\operatorname{L}_{y_0}(\beta) = p_2$. Now, note that $c_j = F(x_{n+1})(p_j)\neq 0$ for $j\in\{1,2\}$. Hence, $\frac{1}{c_1}p_1$ and $\frac{1}{c_2}p_2$ both solve~\eqref{eq:dmlv1}. Since there is only one solution, $\frac{p_1}{c_1} = \frac{p_2}{c_2}$. We have that $\operatorname{L}_{y_0}$ is an isomorphism. Hence, $\frac{1}{c_2}\lambda=\frac{1}{c_2}\beta$ but $\lambda_n=\beta_n=1$. Therefore, $c_1=c_2$ and $p_1=p_2$.
\end{proof}
\begin{rem}
In the proof of Lemma~\ref{lemmon}, we have shown that the unique monic polynomial $p\in\operatorname{P}_n$ such that $
F(x_i)(p) = 0$ for $i=1,\dots,n$ has real coefficients.
\end{rem}
\begin{dfn}
 Let $n\in\mathbb{N}$ and $X=(x_i)_{i=1}^n\subset \mathbb R_{++}$ be a valid sequence for interpolation. If $\#g(\mathbb R)\geq n+1$, define the \emph{Newton polynomial} associated with the nodes $\{x_i\}_{i=1}^n$ as the unique monic
$g$\--polynomial $N[x_1,\dots,x_n]\in\operatorname{P}_{n}$ such that $
F(x_i)(N[x_1,\dots,x_n])=0,$ for $1\leq i\leq n$.
If $\#g(\mathbb R)=n$, just define $N[x_1,\dots,x_n]$ as the zero polynomial.
\end{dfn}
\begin{exa}
If $g=\operatorname{id}$, then, for $(x_1)_{i=1}^n\subset\mathbb R_{++}=\mathbb R$, $
N[x_1,\dots,x_n](x) = \prod_{j=1}^{n}(x-x_j)$.
Note that in this case $N[x_1,\dots,x_i]N[x_{i+1},\dots,x_n]=N[x_1,\dots,x_n]$. In general, $\operatorname{P}_g$ is not an algebra and we do not have such an expression for the Newton polynomials.
\end{exa}

As a consequence of Proposition~\ref{formadeterminante}, we get the following corollary.
\begin{cor}
 For all $x\in\mathbb R_+$,
 \[
N[x_0,\dots,x_n](x)=(-1)^{n+1}
 \begin{vmatrix}
 1 & g_{y_0,1}(x_0) &\cdots & g_{y_0,n}(x_0) \\
 1 & g_{y_0,1}(x_1) &\cdots & g_{y_0,n}(x_1)\\
 \vdots & \vdots & \ddots & \vdots \\
 1 & g_{y_0,1}(x_n) &\cdots & g_{y_0,n}(x_n) \\
 \end{vmatrix}
^{-1}
 \begin{vmatrix}
 1 & g_{y_0,1}(x_0) &\cdots & g_{y_0,n+1}(x_0) \\
 1 & g_{y_0,1}(x_1) &\cdots & g_{y_0,n+1}(x_1)\\
 \vdots & \vdots & \ddots & \vdots \\
 1 & g_{y_0,1}(x_n) &\cdots & g_{y_0,n+1}(x_n) \\
 1 & g_{y_0,1}(x) &\cdots & g_{y_0,n+1}(x) \\
 \end{vmatrix}.
\]
\end{cor}

Let $n\in\mathbb N$ and $X=(x_i)_{i=1}^n\subset \mathbb R_{++}$ be a valid sequence for interpolation. Assume $\#g(\mathbb R)\geq n+1$. The set $ \{1,N[x_1],N[x_1,x_2]\dots,N[x_1,\dots,x_{n}]\}$
is a basis of $\operatorname{P}_{n}$. This basis induces a lower triangular matrix when reading the mapping $V^{T}_{\operatorname{P}_n}$ on coordinates, where $T= X\cup \{x_{n+1}\}$ for some $x_{n+1}\in \mathbb R_+$. This avoids recomputing all the coefficients of the interpolating polynomial when a new node is added. This representation is usually known as the \emph{Newton representation for interpolation}, see \cite[Chapter~2.6]{davis74}.
\begin{dfn}
\label{divdef}
Let $f\in\operatorname{UC}_g(\mathbb R,\mathbb F)$, $n\in\mathbb{N}$, $y_0\in\bR$ and $X=(x_i)_{i=0}^n\subset \mathbb R_{++}$ be a valid sequence for interpolation. Recall the operator $\operatorname{I}_n$ from Definition~\ref{defL}. Take the unique $\l=(\l_0,\dots,\l_{n})\in\mathbb F^{n+1}$ such that $
\operatorname{I}_n(f)=\operatorname{L}_{y_0}(\lambda) $. We define the \emph{divided difference of $f$ relative to $x_0,\dots,x_n$} as $
f[x_0,\dots,x_n]:=\lambda_{n}$.
\end{dfn}
\begin{rem}
Note that $f[x_0,\dots,x_n]$ does not depend on $y_0$. Denote the projection $\pi_n\colon\lambda = (\l_0,\dots,\l_{n})\in\mathbb F^{n+1}\to\l_n\in\bF$.
We can express the divided differences operator relative to $x_0,\dots,x_n$ as the linear form
\[
\begin{tikzcd}[row sep= 0em]
\bigcdot\ [x_0,\dots,x_n]: \operatorname{UC}_g(\mathbb R,\mathbb F) \arrow[r,"\operatorname{I}_n"] &\operatorname{P}_n \arrow[r,"\operatorname{L}_{y_0}^{-1}"] & \mathbb F^{n+1}\arrow[r,"\pi_n"] &\mathbb F.
\end{tikzcd}
\]
\end{rem}
\begin{rem}
The divided differences are linear forms that satisfy a relation of biorthonormality with respect to the Newton polynomials. Indeed, let $n\in\mathbb N$ and $X=(x_i)_{i=0}^n\subset \mathbb R_{++}$ be a valid sequence for interpolation. For $j,k\in\{0,\dots,n\}$, $
(N[x_0,\dots,x_{j-1}])[x_0,\dots,x_k]=\delta_{j}^k$.
This is usually known as\emph{ biorthonormality of Newton type}. It turns out there is a unique way of obtaining this biorthonormality on general, see \cite[Theorem 2.6.1]{davis74}.
\end{rem}
\begin{rem}
We have that $
\operatorname{I}_n(f)= \sum_{k=0}^n f[x_0,\dots,x_{k}] N[x_0,\dots,x_{k-1}]$ by the biorthonormality of Newton type,
analogously to the classical Newton form of the interpolating polynomial.
\end{rem}

Proposition~\ref{formadeterminante} can be used to derive a determinant formula for the calculation of $f[x_0,\dots,x_n]$.

\begin{pro}
\label{divdifformula}
Let $f\in\operatorname{UC}_g(\mathbb R,\mathbb F)$, $n\in\mathbb{N}$ and $X=(x_i)_{i=0}^n\subset \mathbb R_{++}$ be a valid sequence for interpolation. Then,
\[
f[x_0,\dots,x_n]= (-1)^n\left|
 \begin{array}{cccc}
 1 & g_{y_0,1}(x_0) &\dots & g_{y_0,n}(x_0)\\
 1 & g_{y_0,1}(x_1) &\dots & g_{y_0,n}(x_1)\\
\vdots & \vdots & \ddots & \vdots \\
 1 & g_{y_0,1}(x_{n}) &\dots & g_{y_0,n}(x_{n}) \\
 \end{array}
\right|^{-1}\left|\begin{array}{ccccc}
 f(x_ 0) & 1 & g_{y_0,1}(x_0) &\dots & g_{y_0,n-1}(x_0)\\
 f(x_ 1) & 1 & g_{y_0,1}(x_1) &\dots & g_{y_0,n-1}(x_1)\\
\vdots & \vdots & \ddots & \vdots \\
 f(x_n) & 1 & g_{y_0,1}(x_{n}) &\dots & g_{y_0,n-1}(x_{n}) \\
 \end{array}\right|.
\]
\end{pro}

Taking into account that $\operatorname{I}_n(p)=p$ for  $p\in\operatorname{P}_{n-1}$, we get the following lemma.
\begin{lem}
Let $n\in\mathbb N$ and $X=(x_i)_{i=1}^n\subset \mathbb R_{++}$ be a valid sequence for interpolation. Then for $p\in\operatorname{P}_{n-1}$, we have that $ p[x_0,\dots,x_n]=0$.
\end{lem}
For the next result we consider $g^*$ as defined on Definition~\ref{dfng*}.
\begin{pro}\label{pro:divdifformulaextraña}
Let $n\in\mathbb N$ and $X=(x_i)_{i=0}^n\subset \mathbb R_{++}$ be a valid sequence for interpolation. Let $f\in\operatorname{UC}_g(\mathbb R,\mathbb F)$. Take any $x\in \mathbb R_{+}$ such that $g(x)\neq g^*(x_j)$ for $0\leq j\leq n$. Then,
\begin{equation*}
\label{errorformula}
f(x)-\operatorname{I}_n(f)(x)=f[x_0,\dots,x_n,x]N[x_0,\dots,x_n](x).
\end{equation*}
If $g(x)=g^*(x_j)$ for some $j$, then $f(x)-\operatorname{I}_n(f)(x)=0$.
\end{pro}
\begin{proof} By Remark~\ref{rem:ipstar}, we can assume, without loss of generality, that if $\#(g^*(X)\cap[g(t),g(t^+)])=2$ then $g^*(X)\cap[g(t),g(t^+)] = \{g(t),g(t^+)\}$.
If $g(x)=g^*(x_j)$ for some $j$, the result holds. Assume that $g(x)\neq g^*(x_j)$ for $0\leq j\leq n$. Then, the sequence $X'=(x_i)_{i=0}^{n+1}$ where $x_{n+1}=x$ is a valid sequence for interpolation.
Now, the unique polynomial $\operatorname{I}_{n+1}(f)\in \operatorname{P}_{n+1}$ that interpolates $f$ relative to the nodes $X'$, is given by
$
\operatorname{I}_{n+1}(f) = \operatorname{I}_{n}(f) + f[x_0,\dots,x_n,x]N[x_0,\dots,x_n]
$,
which means that
\[
0=f(x)-\operatorname{I}_{n+1}(f)(x)=f(x)-\operatorname{I}_{n}(f)(x)-f[x_0,\dots,x_n,x]N[x_0,\dots,x_n](x),
\]
as we wanted to prove.
\end{proof}

\section{Taylor's Formula}\label{sec:taylor}

In this section we prove some basic properties of $g$\--differentiable functions and a Taylor's formula in a similar way to \cite{Agarwal1999}, where the authors study time scales.
\begin{lem}\label{lemeq}Let $n\in\mathbb N$ and $f,\varphi_k:{\mathbb R}\to{\mathbb F}$ with $k\in\{0,\dots,n\}$ be such that $f$ is $n$\--times $g$\--differentiable and the $\varphi_k$ are $g$\--differentiable and satisfy
 \begin{equation}\label{eqh}
 (\varphi_{k+1})'_g=\varphi_{k}+(\varphi_{k})'_g\Delta g,\quad (\varphi_0)'_g = 0,
 \end{equation}
	for $k=0,\dots,n-1$. Then,
	\[ \left( \sum_{k=0}^{n-1}(-1)^kf^{(k)}_g\varphi_k\right) '_g=(-1)^{n-1}f^{(n)}_g[\varphi_{n-1}+(\varphi_{n-1})'_g\Delta g]
	+f(\varphi_{0})'_g=(-1)^{n-1}f^{(n)}_g(\varphi_{n})'_g.\]
\end{lem}
\begin{proof}Using the product rule, we have that
	\begin{align*}
 \left( \sum_{k=0}^{n-1}(-1)^kf^{(k)}_g\varphi_k\right) '_g= & \sum_{k=0}^{n-1}(-1)^k\left( f^{(k+1)}_g\varphi_k+f^{(k)}_g(\varphi_k)'_g+f^{(k+1)}_g(\varphi_k)'_g\Delta g\right)
		\\ = & \sum_{k=0}^{n-1}(-1)^k\left( f^{(k+1)}_g[\varphi_k+(\varphi_k)'_g\Delta g]+f^{(k)}_g(\varphi_{k})'_g\right)
		\\ = & \sum_{k=0}^{n-1}(-1)^k\left( f^{(k+1)}_g(\varphi_{k+1})'_g+f^{(k)}_g(\varphi_{k})'_g\right)
		\\[ 0.8em] = & (-1)^{n-1}f^{(n)}_g(\varphi_{n})'_g
		+f(\varphi_{0})'_g=(-1)^{n-1}f^{(n)}_g(\varphi_{n})'_g.\tag*{\qedhere}
	\end{align*}
\end{proof}
Fix $t_0\in\mathbb R$. We will construct a sequence of functions $h_{t_0,n}:{\mathbb R}\longrightarrow{\mathbb R}$, $n\in\mathbb N$, that satisfy~\eqref{eqh}. We start by defining $h_{t_0,0}(t)=1$ for all $t\in\mathbb R$. Given $n\in\mathbb N$, we set $h_{t_0,n}:{\mathbb R}\longrightarrow{\mathbb R}$ recursively as
\[
h_{t_0,n}(t)=
 n\int_{t_0}^t[h_{t_0,n-1}+(h_{t_0,n-1})_g'\Delta g]\operatorname{d}\mu_g
\]
for $t\in\mathbb R$. Note the similarity with Definition~\ref{def:pol}. It is important to notice that $h_{t_0,n}$ is well defined as every previously defined function $h_{t_0,n-1}$ is $g$\--absolutely continuous, see Theorem~\ref{tf}.

\begin{lem}\label{lem:formulahdivnfac}
 For every $t_0\in\mathbb R$ the sequence of functions $\left(\frac{1}{n!}h_{t_0,n}\right)_{n\in\mathbb N}$ satisfies equation~\eqref{eqh}.
\end{lem}
\begin{proof}
 For $n\in\mathbb N$, set $\varphi_n = \frac{h_{t_0,n}}{n!}$. By Theorem~\ref{tf}, for all $k\in\mathbb N$, with $k\geq 1$,
 \begin{align*}
 (\varphi_k)'_g & = \left(\frac{h_{t_0,k}}{k!}\right)'_g = \frac{k}{k!}[h_{t_0,k-1}+(h_{t_0,k-1})_g'\Delta g]\\
 &=\frac{h_{t_0,k-1}}{(k-1)!}+\left(\frac{h_{t_0,k-1}}{(k-1)!}\right)'_g\Delta g =\varphi_{k-1}+(\varphi_{k-1})'_g\Delta g. \tag*{\qedhere}
 \end{align*}
\end{proof}

We are ready to prove Taylor's Theorem.

\begin{thm}[Taylor's Theorem]\label{Taylor} Let $n\in\mathbb N$ and $f:{\mathbb R}\to{\mathbb F}$ an $n-1$ times $g$\--differentiable function where $f^{(k)}_g$ is $g$\--absolutely continuous for $k\in\{0,\dots,n-1\}$, then, for $t,t_0\in\mathbb R$,
	\[ f(t)=\sum_{k=0}^{n-1}(-1)^kf^{(k)}_g(t_0)\frac{1}{k!}h_{t,k}(t_0)+\int_{t_0}^t(-1)^{n-1}f^{(n)}_g(s)\left(\frac{1}{n!}h_{t,n}\right)_g'(s)\operatorname{d}\mu_g(s).\]
\end{thm}
\begin{proof}
 Let $t,t_0\in\mathbb R$. The function
 \[
 \sum_{k=0}^{n-1}(-1)^kf^{(k)}_g\frac{1}{k!}h_{t,k}
 \]
 is $g$\--absolutely continuous.
 Assume $t\geq t_0$.
	We have that $h_{t,k}(t)=0$ for $k\in{\mathbb N}$. Thus, by Lemmas~\ref{lemeq} and~\ref{lem:formulahdivnfac} and Theorem~\ref{tf},
	\begin{align*}
 \int_{[t_0,t)}(-1)^{n-1}f^{(n)}_g(s)\left(\frac{1}{n!}h_{t,n}\right)_g'(s)\operatorname{d}\mu_g(s)= & \sum_{k=0}^{n-1}(-1)^kf^{(k)}_g(t)\frac{1}{k!}h_{t,k}(t)-\sum_{k=0}^{n-1}(-1)^kf^{(k)}_g(t_0)\frac{1}{k!}h_{t,k}(t_0)\\ =
		& f(t)-\sum_{k=0}^{n-1}(-1)^kf^{(k)}_g(t_0)\frac{1}{k!}h_{t,k}(t_0),\tag*{\qedhere}\end{align*}
	and the result follows (the proof is analogous for $t<t_0$).
	%
\end{proof}

The functions $h_{t_0,n}$ and the $g$\--monomials are closely related.

\begin{lem}\label{lemrgh}Let $n\in\mathbb N$, $g_{t_0,n}(t)=(-1)^nh_{t,n}(t_0)$, for all $t_0,t\in{\mathbb R}$.
\end{lem}
\begin{proof}
	Since $g_{t_0,n}(t_0)=0$ for $n\in{\mathbb N}$ and $g_{t_0,0}=1$ by definition, using Theorem~\ref{Taylor}, we have that, for $t\in \bR$,
	\begin{align*}g_{t_0,n}(t)=& \sum_{k=0}^{n}(-1)^k(g_{t_0,n})^{(k)}_g(t_0)\frac{1}{k!}h_{t,k}(t_0)+\int_{t_0}^t(-1)^{n}(g_{t_0,n})^{(n+1)}_g(s)\left(\frac{1}{(n+1)!}h_{t,n+1}\right)_g'(s)\operatorname{d}\mu_g(s)\\
 = & \sum_{k=0}^{n}(-1)^k\frac{n!}{(n-k)!}g_{t_0,n-k}(t_0)\frac{1}{k!}h_{t,k}(t_0)\tag*{\qedhere}
 = (-1)^nh_{t,n}(t_0).\end{align*}
\end{proof}
With Lemma~\ref{lemrgh} we can express Theorem~\ref{Taylor} in terms of $g$\--monomials.
\begin{cor}[Taylor's Theorem]\label{taylor}Let $n\in\mathbb N$ and $f:{\mathbb R}\to{\mathbb F}$ an $n-1$ times $g$\--differentiable function where $f^{(k)}_g$ is $g$\--absolutely continuous for $k\in\{0,\dots,n-1\}$, then, for $t,t_0\in\mathbb R$,
	\[ f(t)=\sum_{k=0}^{n-1}\frac{f^{(k)}_g(t_0)}{k!}g_{t_0,k}(t)-\frac{1}{n!}\int_{t_0}^tf^{(n)}_g(s)\frac{\partial g_{s,n}}{\partial_gs}(t)\operatorname{d}\mu_g(s).
	\]
\end{cor}

Now we can obtain Cauchy's formula for repeated integration as an immediate corollary.
\begin{cor}[Cauchy's formula for repeated integration]\label{cor:Cauchyformularepint}
	Take $t_0\in\mathbb R$. Given a $g$\--integrable function $f:{\mathbb R}\to{\mathbb F}$, for every $t\in{\mathbb R}$,
	\[ \int _{t_0}^{t}\int _{t_0}^{s _{1}}\cdots \int _{t_0}^{s _{n-1}}f(s _{n})\operatorname{d}\mu_g( s _{n})\cdots \operatorname{d}\mu_g( s _{2})\operatorname{d}\mu_g( s _{1})=-\frac{1}{n!}\int_{t_0}^tf(s)\frac{\partial g_{s,n}}{\partial_gs}(t)\operatorname{d}\mu_g(s).\]
\end{cor}
\begin{proof}Just consider the function $F:{\mathbb R}\to{\mathbb F}$ given by
	\[ F(t):=\int _{t_0}^{t}\int _{t_0}^{s _{1}}\cdots \int _{t_0}^{s _{n-1}}f(s _{n})\operatorname{d}\mu_g( s _{n})\cdots \operatorname{d}\mu_g( s _{2})\operatorname{d}\mu_g( s _{1})\]
	and apply Corollary~\ref{taylor} to it, taking into account that $F^{(k)}_g(t_0)=0$ for $k=0,\dots,n-1$ and $F^{(n)}_g(t)=f(t)$, to get that
	\[ F(t)=-\frac{1}{n!}\int_{t_0}^tf(s)\frac{\partial g_{s,n}}{\partial_gs}(t)\operatorname{d}\mu_g(s).\qedhere\]
\end{proof}
\begin{rem} Applying Cauchy's formula for repeated integration to the constant function $1$ we get
	\[ \frac{1}{n!}g_{t_0,n}(t)=\int _{t_0}^{t}\int _{t_0}^{s _{1}}\cdots \int _{t_0}^{s _{n-1}}1\operatorname{d}\mu_g( s _{n})\cdots \operatorname{d}\mu_g( s _{2})\operatorname{d}\mu_g( s _{1})=-\frac{1}{n!}\int_{t_0}^t\frac{\partial g_{s,n}}{\partial_gs}(t)\operatorname{d}\mu_g(s),
\]
for $t,t_0\in\mathbb R$ and $n\in\mathbb N$, $n\geq 1$.
\end{rem}

We can compute the $g$\--derivatives of the functions $h_{t_0,n}$ recursively.
\begin{lem}\label{lem:hder}
 For all $t_0\in\mathbb R$, $n\in\mathbb N$ and  $g$\--a.a. $t\in\mathbb R$,
 \begin{equation*}\label{eq:hder}
 (h_{t_0,n})'_g(t) = \sum_{k=0}^{n-1}\frac{n!}{k!}h_{t_0,k}(t)\Delta g^{n-1-k}(t).
 \end{equation*}
\end{lem}
\begin{proof}
 We proceed by induction. The case $n=1$ follows by the $g$\--absolute continuity.
 Assume the induction hypothesis. By Theorem~\ref{tf}, for $g$\--a.a $t\in\mathbb R$,
 \begin{align*}
 (h_{t_0,n+1})'_g(t) & = (n+1)[h_{t_0,n}(t)+(h_{t_0,n})(t)'_g\Delta g(t)]\\
 & = (n+1)\left[h_{t_0,n}(t)+ \left( \sum_{k=0}^{n-1}\frac{n!}{k!}h_{t_0,k}(t)\Delta g^{n-1-k}(t)\right) \Delta g(t)\right]\\
 & = (n+1)h_{t_0,n}(t)+ \left( \sum_{k=0}^{n-1}\frac{(n+1)!}{k!}h_{t_0,k}(t)\Delta g^{n-k}(t)\right)\\
 &= \sum_{k=0}^{n}\frac{(n+1)!}{k!}h_{t_0,k}(t)\Delta g^{n-k}(t).\tag*{\qedhere}
 \end{align*}
 Hence the result follows.
\end{proof}

\section{Peano's kernel theorem}\label{sec:peano}

We establish several Peano\--type results for the interpolation problems developed in the previous sections, making use of the Taylor formula proved above. Hence, we generalize not only the existing Peano's kernel theorem to the Stieltjes setting but also to the new types of interpolation developed in Section~\ref{lagrangeinterpolsec}.

\begin{lem}\label{lem:partialgs}
 Let $n\in\mathbb N$, for all $x\in\mathbb R$ and $g$\--a.a. $s\in\mathbb R$,
 \begin{equation}\label{eq:partialgs}
 \frac{\partial g_{s,n}}{\partial_g s}(x) = n!(-1)^n\sum_{k=0}^{n-1}\frac{(-1)^k\Delta g^{n-1-k}(s)}{k!}g_{s,k}(x).
 \end{equation}
\end{lem}
\begin{proof}
 We have $g_{s,n}(x)=(-1)^n h_{x,n}(s)$ for all $s,x\in\mathbb R$, see Lemma~\ref{lemrgh}. Thus, by Lemma~\ref{lem:hder}, for $g$\--a.a. $s\in\mathbb R$,
 \begin{align*}
 \frac{\partial g_{s,n}}{\partial_g s}(x) & = (-1)^n (h_{x,n})'_g(s) = (-1)^n\sum_{k=0}^{n-1}\frac{n!}{k!}h_{x,k}(s)\Delta g^{n-1-k}(s)\\
 & = n!(-1)^n\sum_{k=0}^{n-1}\frac{(-1)^k\Delta g^{n-1-k}(s)}{k!}g_{s,k}(x).\tag*{\qedhere}
 \end{align*}
\end{proof}

The advantages of using the right hand side of equation~\eqref{eq:partialgs} is that it exists for all $s,x\in\mathbb R$. We use this fact to define the following kernel.

\begin{dfn}
 Let $a,b\in\mathbb R$, $a<b$ and $n\in\mathbb N$. We define the kernel
 \[
 K^+_n(x,s)=\begin{dcases}n!(-1)^n\sum_{k=0}^{n-1}\frac{(-1)^k\Delta g^{n-1-k}(s)}{k!}g_{s,k}(x), & s\in[a,x), \\
 0,& s\in[x,b],
 \end{dcases}
 \]
 for $x,s\in[a,b]$.
\end{dfn}
\begin{rem}
The above is a generalization of the typical Peano kernel, see \cite[Section 13.2]{scott2011numerical}. Indeed, for $g=\operatorname{id}$,
 \[
 K^+_n(x,s)=\begin{dcases}-n(x-s)^{n-1}, & s\in[a,x), \\
 0,& s\in[x,b].
 \end{dcases}
\]
\end{rem}
Let us show some of the properties that $K^+_n$ has.

\begin{pro}\label{pro:kernelproperties}
 Let $a,b\in\mathbb R$, $a<b$ and $n\in\mathbb N$. Then:
 \vspace{-1em}\begin{enumerate}[noitemsep]
 \item[I.] For all $f\in\cL^1_g([a,b))$, the function
 \[
 H: x\in[a,b]\to \int_{[a,b)}f(s)K^+_n(x,s)\operatorname{d}\mu_g(s)
 \]
 is $n-1$ times $g$\--differentiable and $H^{(k)}_g$ is $g$\--absolutely continuous for $k\in\{0,\dots,n-1\}$.
 \item[II.] If $n>1$, the function $
 K^+_n(\cdot,s): x\in[a,b]\to K^+_n(x,s)
 $
 is $g$\--absolutely continuous. If $n=1$, $ K^+_1(x,s) = -\chi_{(s,b]}(x)$,
 where $\chi$ denotes the indicator function, which is not $g$\--continuous in general but is regulated.
 \item[III.] $K^+_n$ is bounded on $[a,b]^2$.
 \end{enumerate}
\end{pro}
\begin{proof}
I. Let $x\in[a,b]$. By Lemma~\ref{lem:partialgs},
 \[
 H(x) = \int_{[a,b)}f(s)K^+_n(x,s)\operatorname{d}\mu_g(s) = \int_{[a,x)}f(s)\frac{\partial g_{s,n}}{\partial_g s}(x)\operatorname{d}\mu_g(s),
 \]
 so, I follows from Corollary~\ref{cor:Cauchyformularepint}.

II. Fix $s\in[a,b]$, then
 \[
 K^+_n(x,s)=\begin{dcases} 0,& x\in[a,s],\\
 n!(-1)^n\sum_{k=0}^{n-1}\frac{(-1)^k\Delta g^{n-1-k}(s)}{k!}g_{s,k}(x), & x\in(s,b].
 \end{dcases}
 \]
 Note that $K^+_n(\cdot,s)$ is a $g$\--polynomial on $(s,b]$. Define $r:[a,b]\to\mathbb R$ as
 \[
 r(x) = \begin{dcases} 0,& x\in[a,s),\\
 n!(-1)^n\sum_{k=1}^{n-1}\frac{(-1)^k\Delta g^{n-1-k}(s)}{(k-1)!}g_{s,k-1}(x), & x\in[s,b].
 \end{dcases}
 \]
 Observe that $r\in\cL^1_g([a,b))$. Also, if $\Delta g(s)\neq 0$,
 \[
 K^+_n(x,s) = \int_{[a,x)}(r(t)+n!(-1)^n\Delta g^{n-2}(s)\chi_{\{s\}}(t))\operatorname{d} \mu_g(t).
 \]
 If $n>1$ and $\Delta g(s)=0$,
 \[
 K^+_n(x,s) = \int_{[a,x)}r(t)\operatorname{d} \mu_g(t).
 \]
 Therefore, if $n>1$, $K^+_n(\cdot,s)$ is $g$\--absolutely continuous. The case $n=1$ is straightforward.

 III. Let us prove that for $x,s\in[a,b]$ with $s<x$, $
 g_{s,k}(x)\leq g_{a,k}(b)$
 for all $k\in\mathbb N$. By Proposition~\ref{center},  since $g_{a,k}$ is nondecreasing on $[a,\infty)$,
 \[
 g_{a,k}(x)=\sum_{j=0}^{k}{k\choose j}g_{a,j}(s)g_{s,k-j}(x) = g_{s,k}(x) + \sum_{j=1}^{k}{k\choose j}g_{a,j}(s)g_{s,k-j}(x)\geq g_{s,k}(x).
 \]
 Furthermore, $\Delta g$ is bounded on $\mathbb R$. Hence, we have that, for  fixed $x,s\in[a,b]$,
 \[
 |K^+_n(x,s)|\leq n!\sum_{k=0}^{n-1}\frac{\left\lVert\Delta g\right\rVert_{\infty}^{n-1-k}}{k!}g_{a,k}(b).\qedhere
 \]
\end{proof}

In the usual setting, the interpolatory projections commute with the integral. We shall prove this as well for the general interpolatory projections in the Stieltjes context.

\begin{pro}\label{pro:integralLcommutative}
Let $h:[a,b]\times[a,b]\rightarrow \mathbb R$ be a function such that for any fixed $x\in[a,b]$, $h(x,\cdot)$ is $g$\--integrable on $[a,b)$. Define
\[
H : x\in[a,b]\to \int_{[a,b)}h(x,s)\operatorname{d}\mu_g(s).
\]
Suppose that $h(\cdot,s)$ is regulated for $g$\--a.a. $s\in[a,b)$. Also assume that for $g\text{\--a.a. }s\in[a,b]$, $|h(x,s)|\leq f(s) \text{ for all }x\in[a,b]$,
where $f\in\cL^1_g([a,b],[0,\infty))$. Let $n\in\mathbb N$ and $X =(x_i)_{i=0}^n \subset [a,b]_{++}$ be a valid sequence for interpolation and consider the operator $\operatorname{I}_n$ we defined on Definition~\ref{defL}. Then, $H$ is regulated and
\[
\operatorname{I}_n(H)=\int_{[a,b)}\operatorname{I}_n(h(\cdot,s))\operatorname{d}\mu_g(s).
\]
In particular,
\[
H[x_0,\dots,x_n]=\int_{[a,b)}h(\cdot,s)[x_0,\dots,x_n]\operatorname{d}\mu_g(s).
\]
\end{pro}
\begin{proof}
 Take $x\in [a,b)$. Set $(y_k)_{k\in\mathbb N}\subset (x,b]$ such that $y_k\to x$ as $k$ tends to infinity. Thus, $h(y_k,\cdot)$ is a sequence of $g$\--measurable functions. Even more, for $g$\--a.a. $s\in[a,b]$, $
 h(y_k,s)\to h(x^+,s)$
 as $k$ tends to infinity.
 Hence, the function $h(x^+,\cdot)$ is defined $g$\--almost everywhere on $[a,b]$ and is $g$\--measurable. By hypothesis, the functions $h(y_k,\cdot)$ are bounded $g$\--almost everywhere by the $g$\--integrable function $f$. Hence, by the dominated convergence theorem,
 \[
 H(y_k)\to \int_{[a,b)}h(x^+,s)\operatorname{d}\mu_g(s)
 \]
 as $k$ tends to infinity.
 From this we deduce the existence of $H(x^+)$ and
 \[
 H(x^+)=\int_{[a,b)}h(x^+,s)\operatorname{d}\mu_g(s).
 \]
 The existence of left hand side limits is analogous.
 Therefore, $H$ is regulated and, for $t\in[a,b]_{++}$,
 \[
 F(t)(H)=\int_{[a,b)}F(t)(h(\cdot,s))\operatorname{d}\mu_g(s).
 \]
 Denote by $\{\ell_j\}_{j=0}^n$ the \emph{Lagrange basis} of $\operatorname{P}_n$ which is the inverse image of the canonical basis of the Vandermonde mapping $V^X_{\operatorname{P}_n}$. Thus, for $x\in \mathbb R$,
 \begin{align*}
 \operatorname{I}_n(H)(x) &= \sum_{k=0}^n F(x_k)(H)\ell_k(x)= \sum_{k=0}^n\left(\int_{[a,b)} F(x_k)(h(\cdot,s))\operatorname{d}\mu_g(s)\right)\ell_k(x)\\
 &=\int_{[a,b)}\left(\sum_{k=0}^n F(x_k)(h(\cdot,s))\ell_k(x)\right)\operatorname{d}\mu_g(s)=\int_{[a,b)}\operatorname{I}_n(h(\cdot,s))(x)\operatorname{d}\mu_g(s).
 \end{align*}

 We have shown that the function $
 s\in[a,b]\to F(x_k)(h(\cdot,s))$
 is well defined $g$\--almost everywhere on $[a,b]$ and is $g$\--integrable. Take $y_0\in\mathbb{R}$. Thus, the mapping
 \[
 \begin{tikzcd}[row sep= 0em]
 [a,b]\arrow{r}&\mathbb R^{n+1}\arrow[r,"(\operatorname{L}_{y_0})^{-1}\cdot( V^X_{\operatorname{P}_n})^{-1}"]&\mathbb R^{n+1}\\
 s\arrow[r,mapsto]& \begin{pmatrix}
F(x_0)(h(\cdot,s)) \\
\vdots \\
F(x_n)(h(\cdot,s))
\end{pmatrix}\arrow[r,mapsto]&\begin{pmatrix}
a_0(s) \\
\vdots \\
a_n(s)
\end{pmatrix}
 \end{tikzcd}
 \]
is well defined for $g$\--a.a. $s\in[a,b]$, is $g$\--measurable and belongs to $\cL^1_g([a,b),\mathbb R^{n+1})$. Hence, for $g$\--a.a. $s\in[a,b]$,
\[
\operatorname{I}_n(h(\cdot,s)) = \sum_{k=0}^{n}a_k(s)g_{y_0,k}.
\]
Then, for $x\in\mathbb R$,
\begin{align*}
\operatorname{I}_n(H)(x) & = \int_{[a,b)}\operatorname{I}_n(h(\cdot,s))(x)\operatorname{d}\mu_g(s) = \int_{[a,b)}\sum_{k=0}^{n}a_k(s)g_{y_0,k}(x)\operatorname{d}\mu_g(s)\\
& = \sum_{k=0}^{n}\left( \int_{[a,b)}a_k(s)\operatorname{d}\mu_g(s)\right) g_{y_0,k}(x).
\end{align*}
From this we deduce that
\[
H[x_0,\dots,x_n]= \int_{[a,b)}a_n(s)\operatorname{d}\mu_g(s) = \int_{[a,b)}h(\cdot,s)[x_0,\dots,x_n]\operatorname{d}\mu_g(s).\qedhere
\]
\end{proof}

As an immediate result we obtain the following.

\begin{cor}
Consider $h:[a,b]\times[a,b]$ in the hypotheses of Proposition~\ref{pro:integralLcommutative}. Let $n\in\mathbb N$ and $X =(x_i)_{i=0}^n \subset [a,b]_{++}$ be a valid sequence for interpolation. We have that
\[
\operatorname{R}_n\left(\int_{[a,b)}h(\cdot,s)\operatorname{d}\mu_g(s)\right)=\int_{[a,b)}\operatorname{R}_n(h(\cdot,s))\operatorname{d}\mu_g(s).
\]
\end{cor}

Let $f\in\cL^1_g([a,b))$ and $m\in\mathbb N$. Define $h:[a,b]\times[a,b]\longrightarrow \mathbb R$ as $
h(x,s) = f(s)K^+_m(x,s)$.
Then, $h$ satisfies the hypotheses of Proposition~\ref{pro:integralLcommutative}. Hence, with the above, a Peano type result is possible using the Stieltjes polynomials and the generalized convex Lagrange interpolation.

\begin{thm}[Peano's Kernel Theorem]
Let $\operatorname{L}$ denote some $\operatorname{L}\in \operatorname{R}([a,b],\mathbb F)^*$ or denote a linear endomorphism $\operatorname{L}:\operatorname{R}([a,b],\mathbb F)\to \operatorname{R}([a,b],\mathbb F)$.
Take $n\in\mathbb N$. Suppose $\operatorname{L}$ satisfies that:
\begin{enumerate}[noitemsep, itemsep=.1cm]
		\item For all $f\in\cL^1_g([a,b))$,
 \[ \operatorname{L}
 \left(\int_{[a,b)}f(s)K^+_n(\cdot,s)\operatorname{d}\mu_g(s)\right)=\int_{[a,b)}f(s)\operatorname{L}(K^+_n(\cdot,s))\operatorname{d}\mu_g(s).\]
		\item If $p\in\operatorname{P}_{n-1}$, then $\operatorname{L}(p)=0$.
\end{enumerate}
Then, if $f:[a,b]\rightarrow \mathbb R$ is an $n-1$ times $g$\--differentiable function where $f^{(k)}_g$ is $g$\--absolutely continuous for $k\in\{0,\dots,n-1\}$,
\[
\operatorname{L}(f)=-\frac{1}{n!}\int_{[a,b)} f^{(n)}_g(s)\operatorname{L}(K^+_n(\cdot,s))\operatorname{d}\mu_g(s).
\]
\end{thm}
\begin{proof}
By Corollary~\ref{taylor}, we have that
\[
f(x)=\sum_{k=0}^{n-1}\frac{f^{(k)}_g(a)}{k!}g_{a,k}(x)-\frac{1}{n!}\int_{[a,x)} f^{(n)}_g(s)\frac{\partial g_{s,n}}{\partial_g s}(x)\operatorname{d}\mu_g(s),
\]
for all $x\in[a,b]$. Thus,
\[
f(x)=\sum_{k=0}^{n-1}\frac{f^{(k)}_g(a)}{k!}g_{a,k}(x)-\frac{1}{n!}\int_{[a,b)} f^{(n)}_g(s)K^+_n(x,s)\operatorname{d}\mu_g(s),
\]
by the definition of $K^+_n$, for all $x\in[a,b]$. Therefore, by the linearity of $\operatorname{L}$ and the fact that $\operatorname{L}$ takes $g$\--polynomials of degree less or equal to $n-1$ to the zero function,
\[
\operatorname{L}(f)=-\frac{1}{n!}\operatorname{L}\left(\int_{[a,b)} f^{(n)}_g(s)K^+_n(\cdot,s)\operatorname{d}\mu_g(s)\right).
\]
Now, by hypothesis I,
\[ \operatorname{L}(f)=-\frac{1}{n!}\int_{[a,b)} f^{(n)}_g(s)\operatorname{L}(K^+_n(\cdot,s))\operatorname{d}\mu_g(s).\qedhere\]
\end{proof}

We state now a list of results that follow immediately from the above.
\begin{cor}
Let $n\in\mathbb N$ and $X =(x_i)_{i=0}^n \subset [a,b]_{++}$ be a valid sequence for interpolation. Suppose $f:[a,b]\rightarrow \mathbb R$ is an $n$ times $g$\--differentiable function where $f^{(k)}_g$ is $g$\--absolutely continuous for $k\in\{0,\dots,n\}$. Then,
\[
\operatorname{R}_n(f)=-\frac{1}{(n+1)!}\int_{[a,b)} f^{(n+1)}_g(s)\operatorname{R}_n(K^+_{n+1}(\cdot,s))\operatorname{d}\mu_g(s).
\]
\end{cor}

\begin{cor}\label{col:divdifpeano}
Let $n\in\mathbb N$ and $X =(x_i)_{i=0}^n \subset [a,b]_{++}$ be a valid sequence for interpolation. Suppose $f:[a,b]\rightarrow \mathbb R$ is an $n-1$ times $g$\--differentiable function where $f^{(k)}_g$ is $g$\--absolutely continuous for $k\in\{0,\dots,n-1\}$. Then,
\[
f[x_0,\dots,x_n]=-\frac{1}{n!}\int_{[a,b)} f^{(n)}_g(s)K^+_n(\cdot,s)[x_0,\dots,x_n]\operatorname{d}\mu_g(s).
\]
\end{cor}

\begin{cor}
Let $n\in\mathbb N$ and $X =(x_i)_{i=0}^n \subset [a,b]_{++}$ be a valid sequence for interpolation.
Suppose $f:[a,b]\longrightarrow \mathbb R$ is an $n$ times $g$\--differentiable function where $f^{(k)}_g$ is $g$\--absolutely continuous for $k\in\{0,\dots,n\}$. Take $x\in[a,b]_+$ such that $g(x)\neq g^*(x_j)$ with $j\in\{0,\dots,n\}$. Then,
\[
f(x)-\operatorname{I}_n(f)(x)=-N[x_0,\dots,x_n](x)\frac{1}{(n+1)!}\int_{[a,b)} f^{(n+1)}_g(s)K^+_{n+1}(\cdot,s)[x_0,\dots,x_n,x]\operatorname{d}\mu_g(s).
\]
\end{cor}
\begin{proof}
 Use Corollary~\ref{col:divdifpeano} together with Proposition~\ref{pro:divdifformulaextraña}.
\end{proof}

\section{$g$\--Monomials centered at right hand side limits}\label{sec:gmoncenteredRmas}

Fix $t\in\mathbb R$. Recall that on Lemma~\ref{lemrgh}, we showed the function $
x\in\mathbb R
\to g_{x,k}(t)$
is $g$\--absolutely continuous for all $k\in\mathbb N$. Hence, they are uniformly $g$\--continuous and may be evaluated over $\mathbb R_+$. This gives the idea of centering the $g$\--monomials at points on $\mathbb R_+$. We have not shown this earlier as it was not strictly necessary. However, to be able to center  $g$\--polynomials on $\bR_+$ will be fundamental in Section~\ref{sec:hermite}.

\begin{dfn}
	Let $(M_i,d_i)$ be metric spaces for $i=1,\dots,n$, $n\in\bN$. We define the product metric $\prod_{i=1}^n d_i $ on the cartesian product $\prod_{i=1}^n M_i$ as
	\[
	\prod_{i=1}^n d_i \left((x_1,\dots,x_n),(y_1,\dots,y_n) \right) = \max\{d_1(x_1,y_1),\dots,d_n(x_n,y_n)\}
	\]
	for all $(x_1,\dots,x_n),(y_1,\dots,y_n)\in \prod_{i=1}^n M_i$. This metric induces the product topology on $\prod_{i=1}^n M_i$.
\end{dfn}

Before continuing note this technical remark.

\begin{rem}\label{rem:recordatorioopens}
	For $a,b\in\bR$, $a\leq b$. We have that
	\[
	\overline{[a,b]} = g^{-1}([g(a),g(b)]) = [a,b]^g_+
	\]
	where $\overline{[a,b]}$ denotes the closure of $[a,b]$ on $\bR_+$, thus,
	\[
	(a,b)^g_+\ss\overline{[a,b]}
	\]
	where $(a,b)^g_+$ is an open set of $\bR_+$.
\end{rem}

We are going to use the following gluing result.

\begin{thm}\label{th:extensiononopensetsmaps}
	Let $X$ and $Y$ be topological spaces and $A_m\ss X$ such that $X=\bigcup_{m\in\bN} \operatorname{Int}(A_m)$. Consider a sequence of continuous mappings $f_m : A_m\longrightarrow Y$ such that $f_m(x)=f_k(x)$ for all $x\in A_m\cap A_k$ and $m,k\in \bN$ with $m\neq k$. Then, there exists a continuous mapping $F:X\longrightarrow Y$ such that $\restr{F}{A_m}=f_m$ for all $m\in \bN$.
\end{thm}
\begin{proof}
	Let $x\in X$, there exists some $m\in\bN$ such that $x\in \operatorname{Int}(A_m)$. Define $F(x)=f_m(x)$. The hypotheses assure the mapping $F$ is well defined. This implies $\restr{F}{A_m}=f_m$ for all $m\in \bN$. Continuity is a local property, since $\smash{\restr{F}{\operatorname{Int}(A_m)}}$ is continuous then $F$ is continuous at $x$. The result follows.
\end{proof}

Finally, we can show the  $g$\--monomials admit a  $g$\--continuous extension to $\bR^+$ on both variables.

\begin{pro}\label{pro:gmondoubleext}
	Endow $\bR_+\times \bR_+$ and $\bR\times \bR$ with the product metric $d_g\times d_g$. Therefore, for all $n\in\bN$, the function
	\[
	(r,s)\in \bR\times \bR \to g_{r,n}(s)
	\]
	is continuous and admits a unique continuous extension
	\[
	(r,s)\in \bR_+\times \bR_+ \to g_{r,n}(s).
	\]
\end{pro}
\begin{proof}
	The extension must be unique since $\bR\times \bR$ is dense on $\bR_+\times \bR_+$. Fix $n\in\bN$ and consider some interval $[a,b]\ss\bR$. Let us show that $
	(r,s)\in [a,b]^2 \to g_{r,n}(s)
	$
	is uniformly continuous on $[a,b]^2$ endowed with the product metric $d_g\times d_g$. First, denote as $C = \max\{|g(a)|,|g(b)|\}$. Therefore, by Lemma~\ref{lemrgh} and \cite[Corollary 3.12]{Cora2023},
	\[
	|h_{s,m}(x)|=|g_{x,m}(s)|\leq m!|g(s)-g(x)|^m\leq m!(2C)^m
	\]
	for all $m\in\bN$ and $x,s\in[a,b]$. We can also bound $(h_{s,m})'_g$ as follows. Take $s\in[a,b]$ and $m\in\bN$. By Lemma~\ref{lem:hder},
	\begin{align*}
		|(h_{s,m})'_g(x)| & \leq \sum_{k=0}^{m-1}\frac{m!}{k!}|h_{s,k}(x)||\Delta g^{m-1-k}(x)|\leq \sum_{k=0}^{m-1}\frac{m!}{k!}k!(2C)^k\norm{\Delta g}_\infty^{m-1-k}\\
		& = m! \sum_{k=0}^{m-1}(2C)^k\norm{\Delta g}_\infty^{m-1-k}=:B
	\end{align*}
	for  $g$\--a.a. $x\in[a,b]$. Now, for $(r,s),(x,y)\in [a,b]^2$,
	\begin{align*}
		|g_{r,n}(s)-g_{x,n}(y)| & \leq |g_{r,n}(s)-g_{r,n}(y)| + |g_{r,n}(y)-g_{x,n}(y)|\\
		& = |g_{r,n}(s)-g_{r,n}(y)| + |h_{y,n}(r)-h_{y,n}(x)|\\
		& \leq n\left|\int_{y}^s|g_{r,n-1}(t)|\operatorname{d}\mu_g(t)\right| + \left|\int_{x}^r|(h_{y,n})'_g(t)|\operatorname{d}\mu_g(t)\right|\\
		& \leq n|g(s)-g(y)|(n-1)!(2C)^{n-1} + |g(r)-g(x)|B\\
		&\leq A\max\{|g(s)-g(y)|,|g(r)-g(x)|\}
	\end{align*}
	where $A$ denotes $2\max\{n!(2C)^{n-1},B\}$. Thus, $(r,s)\in [a,b]^2 \to g_{r,n}(s)$ is Lipschitz, and, therefore, uniformly continuous. By Theorem~\ref{extensionpseudometrix}, $(r,s)\in [a,b]^2 \to g_{r,n}(s)$ admits a unique uniformly continuous extension to $\smash{\overline{[a,b]}^2}$ where $\smash{\overline{[a,b]}}$ denotes the closure of $[a,b]$ in $\bR_+$ with respect to the  $g$\--topology. Even more, as we showed in Section~\ref{sec:gdomain}, $[a,b]_+\ss\smash{\overline{[a,b]}}$.

	Consider the intervals $I_m = [-m,m]$ for $m\in\bN$. Thus, $I_m\ss I_{m+1}$ and $\bR = \bigcup_{m\in\bN} I_m$. Also, by Remark~\ref{rem:recordatorioopens}, $\bR_+ = \bigcup_{m\in\bN} (-m,m)^g_+\ss\bigcup_{m\in\bN} \overline{I_m}$. Let
	\[
	F_m: (x,y)\in (I_m)^2\ss \bR^2 \to F_m(x,y) = g_{x,n}(y)
	\]
	and consider its uniformly continuous extension to $\overline{I_m}^2$
	\[
	\til F_m: (x,y)\in \overline{I_m}^2\ss\bR_+^2 \to F_m(x,y) = g_{x,n}(y).
	\]
	Since $\restr{F_m}{(I_k)^2}=F_k$ for all $k\leq m$, $\restr{\til F_m}{\overline{I_k}^2}=\til F_k$ for all $k\leq m$. Also, $\bigcup_{m\in\bN} ((-m,m)^g_+)^2 = \bR_+^2$. Applying Theorem~\ref{th:extensiononopensetsmaps} with $A_m = \smash{\overline{I_m}^2}$ we get the result.
\end{proof}

The extended  $g$\--monomials also satisfy Proposition~\ref{center}.

\begin{lem}\label{lem:centerformulaonRmas}
For all $r,s\in \mathbb R_+$ and $n\in\mathbb N$,
\[
g_{r,n}(x)=\sum_{k=0}^n {n\choose k} g_{r,k}(s) g_{s,n-k}(x)
\]
for all $x\in \mathbb R_+$.
\end{lem}
\begin{proof}
	The function
	\[
	(r,s,x)\in \bR_+^3 \to g_{r,n}(x)-\sum_{k=0}^n {n\choose k} g_{r,k}(s) g_{s,n-k}(x)
	\]
	is continuous on $\bR_+^3$ with the product topology by Proposition~\ref{pro:gmondoubleext}. Since it is zero restricted to $\bR^3$ by Proposition~\ref{center} while $\bR^3$ is a dense subset of $\bR_+^3$ we infer it is zero everywhere.
\end{proof}
\begin{rem}
Let $t\in D_g$. By Lemma~\ref{lem:centerformulaonRmas}, for all $x\in\mathbb R$,
\[
g_{t^+,n}(x)=\sum_{k=0}^n {n\choose k} g_{t^+,k}(t) g_{t,n-k}(x).
\]
Hence, the functions
 $
x\in\mathbb R
\to g_{t^+,k}(x)$
and their linear combinations are $g$\--polynomials. Also, Lemma~\ref{lem:centerformulaonRmas} shows that $\operatorname{P}_n$ is generated by $\{g_{t^+,k}\}_{k=0}^n$ for $n\in\mathbb N$. This implies that $\{g_{x_0,k}\}_{k=0}^n$ is a set of generating $g$\--monomials for all $x_0\in\mathbb R_+$ and the operators $\operatorname{L}_{x_0}$ and $\operatorname{M}_{x_0\to x_1}$ admit natural generalizations to $x_0,x_1\in\mathbb R_+$ while Lemma~\ref{lem:commutetativeoperators} continues to hold.
Notice that for all $n\in\mathbb N$, $
g_{t,n}(t^+)=g_{t,n}(t)+n\Delta g(t)g_{t,n-1}(t)$.
Thus,
\begin{equation}\label{eq:fttmas}
g_{t,n}(t^+)=\begin{dcases}
 1, & n=0,\\
 \Delta g(t), & n=1,\\
 0, & n\geq 2.
\end{dcases}
\end{equation}
Therefore, for all $x\in\mathbb R$,
\[
g_{t,n}(x)=\sum_{k=0}^n {n\choose k} g_{t,k}(t^+) g_{t^+,n-k}(x)= g_{t^+,n}(x)+n\Delta g(t)g_{t^+,n-1}(x)
\]
which gives a recursive formula for $g_{t^+,n}(x)$. Indeed, for $n\geq 1$,
\begin{equation}\label{eq:gtmasenx}
g_{t^+,n}(x) = g_{t,n}(x)-n\Delta g(t)g_{t^+,n-1}(x)
\end{equation}
for all $x\in\mathbb R$. Hence, $
g_{t^+,0}(x) = 1$ and $g_{t^+,1}(x) = g(x)-g(t^+)$, for every $ x\in\mathbb R$.
\end{rem}
\begin{lem}\label{dg+}
 For all $n\in\mathbb N$, $n\geq 1$, $
 (g_{t^+,n})'_g(x) = n g_{t^+,n-1}(x)$
 for $g$\--a.a. $x\in\mathbb R$.
\end{lem}
\begin{proof}
 We proceed by induction. The case $n=1$ is shown above. Let $n\in\mathbb N$, $n\geq 2$ and assume the induction hypothesis. Now, for $g$\--a.a. $x\in\mathbb R$,
 \begin{align*}
 (g_{t^+,n})'_g(x) & = (g_{t,n}-n\Delta g(t)g_{t^+,n-1})'_g(x) \\
 & = ng_{t,n-1}(x)-n(n-1)\Delta g(t)g_{t^+,n-2}(x) \\
 & = n[g_{t,n-1}(x)-(n-1)\Delta g(t)g_{t^+,n-2}(x)] = n g_{t^+,n-1}(x).\tag*{\qedhere}
 \end{align*}
\end{proof}

Lemma~\ref{dg+} shows that $\operatorname{D}$ provides differentiation over the coordinates to the $g$\--monomials centered at right hand side limits as well. For $x,y\in\mathbb R$, $x<y$,
\[
g_{t^+,n}(x)-g_{t^+,n}(y) = n\int_{[x,y)} g_{t^+,n-1}(s)\operatorname{d}\mu_g(s).
\]
Therefore, for all $n\in\mathbb N$, $\lambda\in \mathbb F^{n+1}$, and $x,y\in\mathbb R$ with $x<y$,
\[
\operatorname{L}_{x_0}(\lambda)(x)-\operatorname{L}_{x_0}(\lambda)(y) = \int_{[x,y)} \operatorname{L}_{x_0}\operatorname{D}(\lambda)(s)\operatorname{d}\mu_g(s)
\]
for all $x_0\in\mathbb R_+$.

Using equations~\eqref{eq:gtmasenx} and~\eqref{eq:fttmas} we get,
\[
g_{t^+,n}(t^+)=\begin{dcases}
 1, & n=0,\\
 0, & n\geq 1.
\end{dcases}
\]
We will need the following result for the next section.
\begin{pro}\label{pro:gmonforderparaRmas}
 Let $x_0,y\in\mathbb R_+$ such that $x_0\leq _g y$ and
$
 \#g([x_0,y]^g_+)\leq m\in\mathbb N$.
 Then, $g_{x_0,k}(y)=0$ for $k\geq m$.
\end{pro}
\begin{proof}
 We proceed by induction on $m$. If $m=1$, then $g(x_0)=g(y)$. Thus, $g_{x_0,k}(y)=0$ for $k\geq 1$. Suppose $m>1$ and assume the induction hypothesis. By Proposition~\ref{intervalsabmasandmeasure},
 \[
 g_{x_0,k}(y) - g_{x_0,k}(x_0) = \int_{[x_0,y)^g}k g_{x_0,k-1}(s)\operatorname{d}\mu_g(s).
 \]
 Let $k\geq m$. Then, $g_{x_0,k}(x_0)=0$. Take $s\in [x_0,y)^g$. Since $g(s)<g(y)$ and $[x_0,s]^g_+\subset [x_0,y]^g_+$, $\#g([x_0,s]^g_+)\leq m-1$. By the induction hypothesis, $g_{x_0,k-1}(s)=0$. Therefore, $g_{x_0,k-1}$ vanishes on $[x_0,y)^g$ and $g_{x_0,k}(y)=0$.
\end{proof}

\section{Hermite interpolation}\label{sec:hermite}
In this section, we address the Hermite* interpolation problem. A generalization of the classical Hermite interpolation problem. We will show that Hermite* admits a unique solution given some necessary and sufficient condition. The standard formulation of the Hermite interpolation problem is as follows.

\begin{dfn}[Hermite interpolation problem]\label{dfnhip} Consider $k+1 \in \mathbb{N}$ points $x_0,\dots,x_k\in\mathbb R$ with associated \emph{multiplicities} $m_i \in \mathbb{N}$ for $i \in \{0,\dots,k\}$ and
 $t_{ij} \in \mathbb{F}$. We call the system of equations
	\begin{equation}
		\label{hermite}
		p^{(j)}_g(x_i) = t_{i,j}, \quad j \in \{0,\dots,m_i\}, \quad i \in \{0,\dots,k\},
	\end{equation}
	the \emph{Hermite interpolation problem}, for which we try to find a \emph{polynomial solution} $p\in \operatorname{P}_n(\bF)$ where $n = k+\sum_{j=0}^k m_j$.
\end{dfn}

\begin{rem} The notion of multiplicity used in Definition~\ref{dfnhip} differs from the usual one, which counts the number of ``roots with the same value''. In fact, the latter is exactly one greater than the former in the classical case.
	\end{rem}
To properly formulate the Hermite interpolation problem, higher--order Stieltjes derivatives must be well defined. However, as this condition fails for certain derivators, we reformulate problem~\eqref{hermite} within a more general framework. Furthermore, interpolation nodes are allowed on $\mathbb{R}_+$.

\begin{dfn}[Hermite* interpolation problem] Consider $k+1\in\mathbb N$ points $x_0,\dots,x_k\in\mathbb R_+$ with associated multiplicities $m_i\in\mathbb N$, $i\in\{0,\dots,k\}$. Fix $n = k+\sum_{j=0}^k m_j$, $y_0\in\mathbb R_+$ and $t_{r,j} \in \mathbb{F}$. We call the system of equations
	\begin{equation}
		\label{hermite*}
		\operatorname{L}_{y_0}\operatorname{D}^{j}(\lambda)(x_r) = t_{r,j},
		\quad j \in \{0,\dots,m_r\},\quad r \in \{0,\dots,k\},
	\end{equation}
	the \emph{Hermite* interpolation problem}, for which we try to find a \emph{solution} $\lambda\in \mathbb{F}^{n+1}$.
\end{dfn}

Whenever higher\--order Stieltjes derivatives are well\--defined problem~\eqref{hermite*} is equivalent to problem~\eqref{hermite}. In contrast, problem~\eqref{hermite*} can always be considered.

Hermite* interpolation can give rise to nonstandard situations. Observe that if the Hermite interpolation problem~\eqref{hermite} has a unique solution, then $\operatorname{dim}\operatorname{P}_n = n+1$. However, by Lemma~\ref{lem:independencegmnomialsX}, for this to occur, there must exist at least $n+1$ $g$\--distinct points in $\mathbb{R}$. This is not the case for Hermite*, as we will show in Example~\ref{examplegconstanthermite}.

This phenomenon is related to the well--definedness of higher-order Stieltjes derivatives and to whether the formula $
	(g_{x_0,j})'_g = j g_{x_0,j-1}$
holds or not. As functions, the above identity holds everywhere except on a set of $g$\--measure zero. In fact, it holds at every point where the Stieltjes derivative is properly defined. Consider the following example.

\begin{exa}
 \label{examplegconstanthermite}
Take a constant derivator $g$. Hence, $g_{x_0,1}(x) = 0$ for all $x_0, x \in \mathbb{R}$. Take $x_0\in\mathbb R$ with multiplicity $m_0\in\mathbb N$, $m_0>0$. Consider $t_i\in \mathbb F$ for $i\in\{0,\dots,m_0\}$. Hence, the Hermite* problem~\eqref{hermite*} is given by
\begin{equation}\label{eqL}
\operatorname{L}_{x_0}\operatorname{D}^{i}(\lambda)(x_0) = \sum_{j=i}^{m_0} \lambda_j \frac{j!}{(j-i)!} g_{x_0,j-i}(x_0) = t_i,\quad i\in\{0,\dots,m_0\}.
\end{equation}
Since $g_{x_0,j-i}(x) = 0$ for $i\neq j$, $x\in\mathbb R$ and $g_{x_0,0}(x_0)=1$, the above problem admits unique solution with $
\lambda_j = t_j/j!$
for $j\in\{0,\dots,m_0\}$. Note how since $g$ is constant the Stieltjes derivative is not defined at any point. We cannot consider a problem of the type of~\eqref{hermite}. Also, $\mu_g(\mathbb{R}) = 0$ and the formula 	$(g_{x_0,j})'_g = j g_{x_0,j-1}$ holds nowhere. Besides, $\operatorname{dim} \operatorname{P} = 1$.
\end{exa}

Example~\ref{examplegconstanthermite} shows that the formulation of Hermite* is more general than Hermite and allows solutions even when we have more equations that dimensions on $\operatorname{P}_n$. We shall focus on solving Hermite*. Let us introduce some notation.

\begin{dfn}\label{def:hermiteoperator}
Let $y_0\in\mathbb R_+$, $x\in\mathbb R_+$ and $n,m\in \mathbb N$ with $n\ge m$. Taking into account formula~\eqref{eqL}, consider, for each $i \in \{0,\dots,m\}$, the linear functionals defined by
\[
\begin{tikzcd}[row sep=0ex]
	\mathbb{F}^{n+1} \arrow[r,"\operatorname{D}^{i}"]
	& \mathbb{F}^{n+1} \arrow[r,"\operatorname{L}_{y_0}"]
	& \mathcal{P}_n \arrow[r,"\operatorname{ev}_{x}"]
	& \mathbb{F}\\
	\lambda \arrow[rrr,mapsto] &&&
	\operatorname{L}_{y_0}\operatorname{D}^{i}(\lambda)(x).
\end{tikzcd}
\]
Collecting these functionals for $i=0,\dots,m$ yields a linear mapping
\[
\operatorname{H}_{y_0}(x,m,n) : \mathbb{F}^{n+1} \to \mathbb{F}^{m+1},
\]
whose associated $(m+1)\times(n+1)$ matrix is given by
\begin{equation}
	\label{formulaHxm}
	[\operatorname{H}_{y_0}(x,m,n)]_{i,j}
	= \frac{j!}{(j-i)!}\, g_{y_0,j-i}(x),
\end{equation}
for $i \in \{0,\dots,m\}$ and $j \in \{0,\dots,n\}$, with the convention that $g_{y_0,j-i}=0$ whenever $j<i$. That is,
\[
[\operatorname{H}_{y_0}(x,m,n)] =
\begin{pmatrix}
	1 & g_{y_0,1}(x) & g_{y_0,2}(x) & g_{y_0,3}(x) & \cdots & g_{y_0,n}(x) \\
	0 & 1 & 2!\,g_{y_0,1}(x) & \frac{3!}{2!} g_{y_0,2}(x) & \cdots & \frac{n!}{(n-1)!} g_{y_0,n-1}(x) \\
	0 & 0 & 2! & 3! g_{y_0,1}(x) & \cdots & \frac{n!}{(n-2)!} g_{y_0,n-2}(x) \\
	\vdots & \vdots & \vdots & \vdots & \ddots & \vdots \\
	0 & 0 & 0 & 0 & \cdots & \frac{n!}{(n-m)!} g_{y_0,n-m}(x)
\end{pmatrix}.
\]

Now let $x_0,\dots,x_k\in\mathbb R_+$ be $k+1$ points with associated multiplicities
$m_i\in\mathbb N$, $i\in\{0,\dots,k\}$.
Fix $y_0\in\mathbb R_+$ and let $
n = k + \sum_{i=0}^{k} m_i$, $
X = (x_0,x_1,\dots,x_k)\in\mathbb R_+^{k+1}$, and $
M = (m_0,m_1,\dots,m_k)\in\mathbb N^{k+1}$.
Define $\operatorname{H}_{y_0}(X,M):\bF^{n+1}\to \bF^{n+1}$ as the linear endomorphism given by the $(n+1)\times(n+1)$ matrix
\[
[\operatorname{H}_{y_0}(X,M)] =
\begin{pmatrix}
	[\operatorname{H}_{y_0}(x_0,m_0,n)] \\
	[\operatorname{H}_{y_0}(x_1,m_1,n)] \\
	\vdots \\
	[\operatorname{H}_{y_0}(x_k,m_k,n)]
\end{pmatrix}.
\]

Therefore, $\operatorname{H}_{y_0}(X,M)$ is the linear endomorphism associated to the Hermite* interpolation problem. That is, take $t_{r,j}$ like in \eqref{hermite*}. Define
\[
T=(t_{0,0},\dots,t_{0,m_0},t_{1,0},\dots,t_{1,m_1},\dots,t_{k,0},\dots,t_{k,m_k})\in \bF^{n+1}.
\]
Hence, $\lambda\in \bF^{n+1}$ solves \eqref{hermite*} if and only if $\operatorname{H}_{y_0}(X,M)(\lambda)=T$.
\end{dfn}
The definition of the matrix \eqref{formulaHxm} is used to give the standard formulation of Hermite interpolation, see for example \cite[Definition 2.1]{LiLin2026ConfluentVandermonde}.
Clearly, the Hermite* interpolation problem always admits a unique solution if and only if for any $T\in\bF^{n+1}$ there exists a unique $\lambda\in \bF^{n+1}$ such that $\operatorname{H}_{y_0}(X,M)(\lambda)=T$. Therefore, we have the following characterization.
\begin{lem}
 \label{lemHermite*}
 Problem~\ref{hermite*} admits a unique solution if and only if the mapping $\operatorname{H}_{y_0}(X,M)$ is bijective, in other words, if and only if $|[\operatorname{H}_{y_0}(X,M)]|\neq 0.$
\end{lem}

Seemingly, the Hermite* interpolation problem depends on a prescribed $y_0\in \bR_+$. However, the original Hermite interpolation problem does not. Let us show this dependence is just superficial using the commutative property of the operators $\operatorname{M}_{y_0'\to y_0}$ and $\operatorname{D}$.
The following result, analogous to Lemma~\ref{reduction}, holds with an identical proof.

\begin{lem}\label{lem:hermitechangeofcenter}
 For any $y_0,y_0'\in\mathbb R_+$,
 $\operatorname{H}_{y_0}(X,M) = \operatorname{H}_{y_0'}(X,M)\operatorname{M}_{y_0'\to y_0} $. Also,
 \[|[\operatorname{H}_{y_0}(X,M)]| = |[\operatorname{H}_{y_0'}(X,M)]|.\]
\end{lem}
\begin{proof}
Given $\operatorname{L}_{y_0}\operatorname{M}_{y_0'\to y_0} = \operatorname{L}_{y_0'}$ and Lemma~\ref{lem:commutetativeoperators}, we have $
\operatorname{L}_{y_0'}\operatorname{D}^j = \operatorname{L}_{y_0}\operatorname{M}_{y_0'\to y_0}\operatorname{D}^j
= \operatorname{L}_{y_0}\operatorname{D}^j \operatorname{M}_{y_0'\to y_0}$.
Hence, $\operatorname{ev}_{x_r}\operatorname{L}_{y_0'}\operatorname{D}^j = \operatorname{ev}_{x_r}\operatorname{L}_{y_0}\operatorname{D}^j\operatorname{M}_{y_0'\to y_0}$ for all $r\in\{0,\dots,k\}$ and $j\in\{0,\dots,m_r\}$. Thus,
\[
\operatorname{H}_{y_0}(X,M) = \operatorname{H}_{y_0'}(X,M)\operatorname{M}_{y_0'\to y_0}.\]
Since $|[\operatorname{M}_{y_0'\to y_0}]|=1$, the result follows.
\end{proof}

Lemma~\ref{lem:hermitechangeofcenter} shows that the ranges of the operators $\operatorname{H}_{y_0}(X,M)$ are the same for all $y_0\in\bR_+$. Also, some $\lambda\in \bF^{n+1}$ solves \eqref{hermite*} for $y_0$ if and only if $\operatorname{M}_{y_0'\to y_0}(\lambda)$ solves \eqref{hermite*} for $y_0'$.
We now establish a necessary condition for the Hermite* interpolation problem to admit a unique solution. We will show with Theorem~\ref{th:hermite*} that this condition is also sufficient.

\begin{pro}
 \label{existenciapuntosher}
 Let $k+1\in\mathbb N$ points $x_0,\dots,x_k\in\mathbb R_+$ ordered from left to right with associated multiplicities $m_i\in\mathbb N$, $i\in\{0,\dots,k\}$. Then, for~\eqref{hermite*} to have a unique solution, it must hold that
$
 \#g([x_{i},x_{i+1}]^g_+)\geq m_i+2
$
 for $i\in\{0,\dots,k-1\}$. In other words, there must exist $g$\--distinct points
 \[
 x_0<_g y_1^0 <_g \dots <_g y_{m_0}^0 <_g x_1 <_g y_1^1 <_g \dots <_g y_{m_1}^1 <_g x_2 <_g \dots <_g x_k.
 \]
\end{pro}
\begin{proof}
 Denote $n=k+\sum_{i=0}^k m_i$. Assume $\#g([x_{i},x_{i+1}]^g_+) < m_i+2$ for some $i\in\{0,\dots,k-1\}$. We will show that $|[\operatorname{H}_{x_i}(X,M)]|=0$. Note that $[\operatorname{H}_{x_i}(x_i,m_i,n)]$ is a concatenation of the $(m_i+1)\times (m_i+1)$ diagonal matrix
 \[
 \begin{pmatrix}
 0! & 0 & 0 &\dots & 0\\
 0 & 1! & 0 &\dots & 0\\
 0 & 0 & 2! &\dots & 0\\
 \vdots & \vdots & \vdots &\ddots & \vdots\\
 0 & 0 & 0 &\dots & m_i!
 \end{pmatrix}
 \]
 with the zero $(m_i+1)\times(n-m_i)$ matrix. By Proposition~\ref{pro:gmonforderparaRmas}, $g_{x_i,j}(x_{i+1})=0$ for all $j\geq m_i+1$. Hence, the first row of $[\operatorname{H}_{x_i}(x_{i+1},m_{i+1},n)]$ looks like
 \[
 \begin{array}{cc}
 \underbrace{1 \quad g_{x_i,1}(x_{i+1}) \quad \cdots \quad g_{x_i,m_i}(x_{i+1})}_{m_i+1}
 & \underbrace{0 \quad \cdots \quad 0}_{n-m_i}
 \end{array}
 \]
 which is clearly a linear combination of the rows of $[\operatorname{H}_{x_i}(x_i,m_i,n)]$. Therefore, the determinant $|[\operatorname{H}_{x_i}(X,M)]|$ is zero and the result follows.
\end{proof}

We will prove a series of results that will eventually lead us to the Hermite* interpolation problem.

\begin{dfn}\label{def:zerosmultiples}
 Let $\lambda\in\mathbb F^{n+1}$, $x_0\in\mathbb R_+$ and $t\in\mathbb R_+$. We say $t$ is a zero (or root of type I) of $\operatorname{L}_{x_0}(\lambda)\in\operatorname{P}_n$ of multiplicity $m\in\mathbb N$ if $
 \operatorname{L}_{x_0}\operatorname{D}^{k}(\lambda)(t) = 0,\quad k\in\{0,\dots,m\}.$
\end{dfn}

Let $x_1\in\mathbb R_+$, so $\operatorname{L}_{x_0}(\lambda)= \operatorname{L}_{x_1}\operatorname{M}_{x_0\to x_1}(\lambda)$.
By Lemma~\ref{lem:commutetativeoperators},
\[
\operatorname{L}_{x_1}\operatorname{D}^k(\operatorname{M}_{x_0\to x_1}(\lambda))(t) = \operatorname{L}_{x_1}\operatorname{M}_{x_0\to x_1}\operatorname{D}^k(\lambda)(t) = \operatorname{L}_{x_0}\operatorname{D}^k(\lambda)(t).
\]
Hence, $t$ is a zero of multiplicity $m$ for $\lambda$ and $x_0$ if and only if it is a zero of multiplicity $m$ for $\operatorname{M}_{x_0\to x_1}(\lambda)$ and $x_1$.

Note that Definition~\ref{def:zerosmultiples} depends on $\lambda$. Indeed, take $g$ as in Example~\ref{examplegconstanthermite} and any $x_0\in\mathbb R_+$. Then, $\operatorname{L}_{x_0}(0,0) = \operatorname{L}_{x_0}(0,1)$ which equals the zero function. For $\operatorname{L}_{x_0}(0,0)$, every $t\in\mathbb R_+$ is a zero of arbitrary multiplicity. However, for $\operatorname{L}_{x_0}(0,1)$, every $t\in\mathbb R_+$ is a zero of multiplicity just $0$.

\begin{pro}\label{pro:setCvanishder}
	Let $x_0\in\mathbb R_+$ and $\lambda\in\bF^{n+1}$. Suppose $t\in\mathbb R_+$ is a zero of $\operatorname{L}_{x_0}(\lambda)$ of multiplicity $m\in\mathbb N$. Then, $\operatorname{L}_{x_0}(\lambda)$ vanishes on
	\begin{equation}
		\label{eq:prosetvanishder}
	C:=	\{z\in\bR_+ \,:\, g(t)\leq g(z) \text{ and }\#g([t,z]^g_+)\leq m+1\}.
	\end{equation} 
\end{pro}
\begin{proof}
	We proceed by induction on $m$. If $m=0$, $C$ contains only one $g$\--distinct point, namely, $t$. But $\operatorname{L}_{x_0}(\lambda)$ vanishes on $t$ so the case $m=0$ holds. Let $m\in\mathbb N$, with $m>0$ and $z\in C$, by Proposition~\ref{intervalsabmasandmeasure},
	\[
	\operatorname{L}_{x_0}(\lambda)(z)- \operatorname{L}_{x_0}(\lambda)(t) = \int_{[t,z)^g}\operatorname{L}_{x_0}\operatorname{D}(\lambda)(s)\operatorname{d}\mu_g(s).
	\]
 However, $[t,z)^g\subset\{s\in\bR\,:\,g(t)\leq g(s) \text{ and } \#g([t,s]^g_+)\leq m\}$. Note that $t$ is a zero of multiplicity $m-1$ for $x_0$ and $\operatorname{D}(\lambda)$. By the induction hypothesis, $\operatorname{L}_{x_0}\operatorname{D}(\lambda)$ vanishes on $[t,z)^g$. Therefore, $\operatorname{L}_{x_0}(\lambda)$ vanishes on $C$ and the claim follows.
\end{proof}

\begin{rem}
The result above gives us another route of proof for Proposition~\ref{existenciapuntosher}. We prove Proposition~\ref{existenciapuntosher} in the proof of Theorem~\ref{th:chp} using Proposition~\ref{pro:setCvanishder}.
\end{rem}
\begin{lem}
 \label{lem:ptoshermitelema}
 Let $x_0\in\mathbb R_+$ and $\lambda\in\mathbb R^{n+1}$. Take $t_1,t_2\in \bR_+$ such that $t_1<_g t_2$. Suppose that $t_1$ is a zero of $\operatorname{L}_{x_0}(\lambda)$ of multiplicity $m\in\mathbb N$ or a root of type II (if this is the case define $m=0$), $t_2$ is a root (of type I or II) of $\operatorname{L}_{x_0}(\lambda)$ and $\#g([t_1,t_2]^g_+)\geq m+2$.
 Then, there exists $t_3\in[t_1,t_2)^g_+$ a root of $\operatorname{L}_{x_0}\operatorname{D}(\lambda)$ such that $\#g([t_1,t_3]^g_+)\geq m+1$.
\end{lem}
\begin{proof}
 If $m=0$, taking $t_3$ given by Proposition~\ref{rolle2} proves the result. Assume from now on that $m\geq 1$.
 Consider the set $C=\{z\in[t_1,t_2)^g\,:\, \#g([t_1,z]^g_+)\leq m\}$. $C$ is a possibly empty interval and thus,  $g$\--measurable. Note that $t_1$ is a zero of $\operatorname{L}_{x_0}\operatorname{D}(\lambda)$ of multiplicity $m-1$. Hence, by Proposition~\ref{pro:setCvanishder}, $\operatorname{L}_{x_0}\operatorname{D}(\lambda)$ vanishes on $C$.
We claim that \[[t_1,t_2)^g\bs C = \{z\in[t_1,t_2)^g\,:\, \#g([t_1,z]^g_+)\geq m+1\}\] is a nonempty interval of positive $g$\--measure. It is clearly an interval; if we show it has positive  $g$\--measure the claim holds. Since $\#g([t_1,t_2]^g_+)\geq m+2$, there exists $s\in[t_1,t_2)^g_+$ such that $\#g([t_1,s]^g_+)\geq m+1$. By Proposition~\ref{intervalsabmas}, $\mu_g([s,t_2)^g)=g(t_2)-g(s)>0$, but $[s,t_2)^g\subset [t_1,t_2)^g\bs C$, which proves the claim.

 We now consider two cases:\\
 \noindent 1.\emph{ $t_2$ is a root of type I:} In this case
 \[
 \operatorname{L}_{x_0}(\lambda)(t_2)-\operatorname{L}_{x_0}(\lambda)(t_1) = \int_{[t_1,t_2)^g}\operatorname{L}_{x_0}\operatorname{D}(\lambda)(s)\operatorname{d}\mu_g(s).
 \]
 Therefore,
 \[
 0 = \int_{[t_1,t_2)^g\setminus C}\operatorname{L}_{x_0}\operatorname{D}(\lambda)(s)\operatorname{d}\mu_g(s).
 \]
 By arguments similar to those used in Proposition~\ref{rolle}, $\operatorname{L}_{x_0}\operatorname{D}(\lambda)$ has a root $t_3\in[t_1,t_2)^g_+$ such that $\#g([t_1,t_3]^g_+)\geq m+1$.\\
 \noindent 2. \emph{$t_2$ is a root of type II:} In this case $\operatorname{L}_{x_0}\operatorname{D}(\lambda)(t_2)\neq 0$. Without loss of generality, assume $\operatorname{L}_{x_0}\operatorname{D}(\lambda)(t_2)> 0$, thus, $\operatorname{L}_{x_0}(\lambda)(t_2)<0$ and $\operatorname{L}_{x_0}(\lambda)(t_2^+)>0$. Also,
 \[
 0>\operatorname{L}_{x_0}(\lambda)(t_2)-\operatorname{L}_{x_0}(\lambda)(t_1)=\int_{[t_1,t_2)^g\setminus C}\operatorname{L}_{x_0}\operatorname{D}(\lambda)(s)\operatorname{d}\mu_g(s).
 \]
 Therefore, there must exist some $s\in [t_1,t_2)^g\setminus C$ such that $\operatorname{L}_{x_0}\operatorname{D}(\lambda)(s)<0$. By Corollary~\ref{bolzanomas}, $\operatorname{L}_{x_0}\operatorname{D}(\lambda)$ has a root $t_3\in[s,t_2)^g_+$. Hence, $\#g([t_1,t_3]^g_+)\geq m+1$.
\end{proof}

We now present a schematic interpretation of an algorithm based on the recursive application of Lemma~\ref{lem:ptoshermitelema}, aimed at progressively reducing the degree of the polynomial associated with the operator $\operatorname{L}_{x_0}(\lambda)$. For the sake of clarity, we illustrate this procedure with an example.

Assume that $\operatorname{L}_{x_0}(\lambda)$ has a zero at $t_1$ and a root at $t_2$.
We represent this situation as two points on a line (Figure~\ref{fig:a}).
Assume that $\# g([t_1,t_2]^g_+) \geq 7$.
Then, there exist at least five $g$\--distinct points between $t_1$ and $t_2$,
which we indicate by the symbol $+$ (Figure~\ref{fig:b}).
%

Suppose that $t_1$ is a zero of multiplicity five of
$\operatorname{L}_{y_0}(\lambda)$.
We represent this multiplicity by stacking points vertically,
where the $j$\--th row corresponds to
$\operatorname{L}_{y_0}\operatorname{D}^{\,j-1}$ (Figure~\ref{fig:c}).
%
%
Applying Lemma~\ref{lem:ptoshermitelema}, we deduce that
$\operatorname{L}_{y_0}\operatorname{D}(\lambda)$ has a zero
$t_3$ such that $\# g([t_1,t_3]^g_+) \geq 6$.
We represent this by placing the symbol $|$ in the second row,
five positions to the right of $t_1$, leaving four symbols $+$ in the middle
reflecting the existence of four $g$\--distinct points between $t_1$ and $t_3$ (Figure~\ref{fig:d}).
%
%
We may apply Lemma~\ref{lem:ptoshermitelema} again to
$\operatorname{L}_{y_0}\operatorname{D}(\lambda)$
with the points $t_1$ and $t_3$. This yields Figure~\ref{fig:e}.
%
%
Proceeding recursively, we obtain (Figure~\ref{fig:f}).
%

Suppose now that instead of only two points $t_1$ and $t_2$ we have a sequence of $g$\--distinct points that are either roots of type I with their respective multiplicity or roots of type II. Thus, we may have a situation like the one in Figure~\ref{fig:g}.
Apply Lemma~\ref{lem:ptoshermitelema} recursively to every possible pair of concatenated roots, hence, we get Figure~\ref{fig:h}.

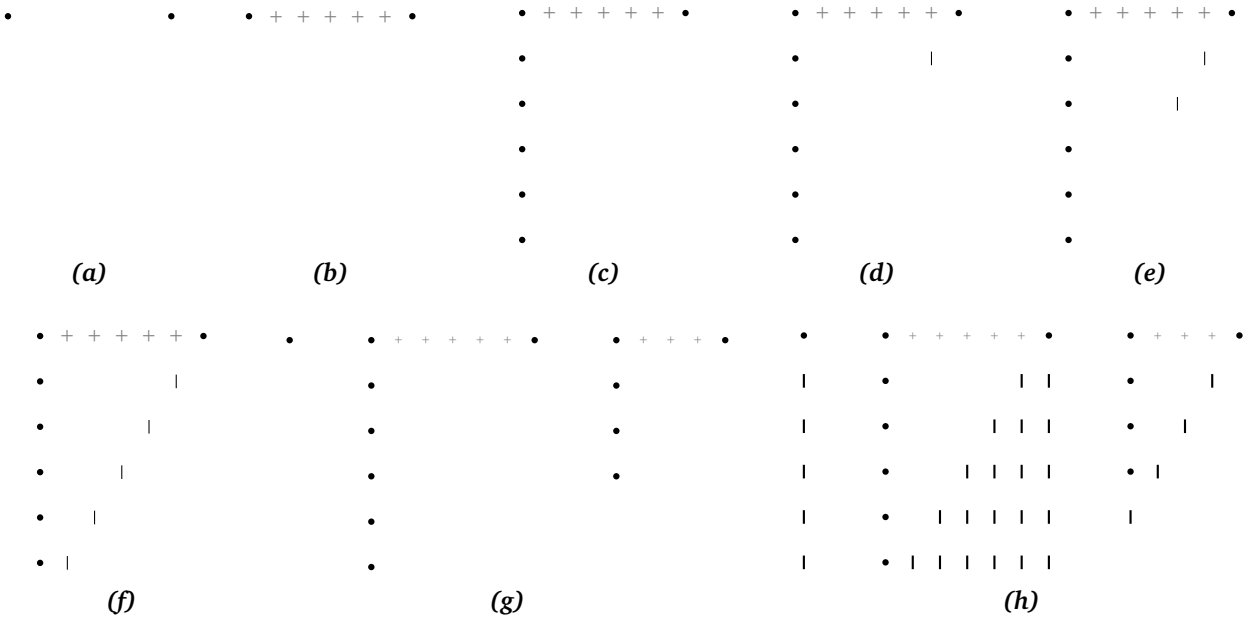
\begin{figure}[h!]
	\centering

	\begin{subfigure}[t]{0.15\textwidth}
		\centering
		\begin{tikzpicture}[scale=0.6, transform shape]
			\useasboundingbox (0,-5) rectangle (3.6, 0);
			\fill (0,0) circle (2pt);
			\fill (3.6,0) circle (2pt);
		\end{tikzpicture}
		\caption{}
		\label{fig:a}
	\end{subfigure}
	\hfill
	\begin{subfigure}[t]{0.20\textwidth}
		\centering
		\begin{tikzpicture}[scale=0.6, transform shape]
			\useasboundingbox (0,-5) rectangle (3.6, 0);
			\fill (0,0) circle (2pt);
			\fill (3.6,0) circle (2pt);
			\node[text opacity=0.5] at (0.6,0) {$+$};
			\node[text opacity=0.5] at (1.2,0) {$+$};
			\node[text opacity=0.5] at (1.8,0) {$+$};
			\node[text opacity=0.5] at (2.4,0) {$+$};
			\node[text opacity=0.5] at (3,0) {$+$};
		\end{tikzpicture}
		\caption{}
		\label{fig:b}
	\end{subfigure}
	\hfill
	\begin{subfigure}[t]{0.20\textwidth}
		\centering
		\begin{tikzpicture}[scale=0.6, transform shape]
			\node[text opacity=0.5] at (0.6,0) {$+$};
			\node[text opacity=0.5] at (1.2,0) {$+$};
			\node[text opacity=0.5] at (1.8,0) {$+$};
			\node[text opacity=0.5] at (2.4,0) {$+$};
			\node[text opacity=0.5] at (3,0) {$+$};
			\fill (0,0) circle (2pt);
			\fill (3.6,0) circle (2pt);
			\fill (0,-1) circle (2pt);
			\fill (0,-2) circle (2pt);
			\fill (0,-3) circle (2pt);
			\fill (0,-4) circle (2pt);
			\fill (0,-5) circle (2pt);
		\end{tikzpicture}
		\caption{}
		\label{fig:c}
	\end{subfigure}
	\hfill
	\begin{subfigure}[t]{0.20\textwidth}
		\centering
		\begin{tikzpicture}[scale=0.6, transform shape]
			\node[text opacity=0.5] at (0.6,0) {$+$};
			\node[text opacity=0.5] at (1.2,0) {$+$};
			\node[text opacity=0.5] at (1.8,0) {$+$};
			\node[text opacity=0.5] at (2.4,0) {$+$};
			\node[text opacity=0.5] at (3,0) {$+$};
			\fill (0,0) circle (2pt);
			\fill (3.6,0) circle (2pt);
			\fill (0,-1) circle (2pt);
			\draw (3,-0.85) -- (3,-1.15);
			\fill (0,-2) circle (2pt);
			\fill (0,-3) circle (2pt);
			\fill (0,-4) circle (2pt);
			\fill (0,-5) circle (2pt);
		\end{tikzpicture}
		\caption{}
		\label{fig:d}
	\end{subfigure}
	\hfill
	\begin{subfigure}[t]{0.20\textwidth}
		\centering
		\begin{tikzpicture}[scale=0.6, transform shape]
			\node[text opacity=0.5] at (0.6,0) {$+$};
			\node[text opacity=0.5] at (1.2,0) {$+$};
			\node[text opacity=0.5] at (1.8,0) {$+$};
			\node[text opacity=0.5] at (2.4,0) {$+$};
			\node[text opacity=0.5] at (3,0) {$+$};
			\fill (0,0) circle (2pt);
			\fill (3.6,0) circle (2pt);
			\fill (0,-1) circle (2pt);
			\draw (3,-0.85) -- (3,-1.15);
			\fill (0,-2) circle (2pt);
			\draw (2.4,-1.85) -- (2.4,-2.15);
			\fill (0,-3) circle (2pt);
			\fill (0,-4) circle (2pt);
			\fill (0,-5) circle (2pt);
		\end{tikzpicture}
		\caption{}
		\label{fig:e}
	\end{subfigure}

	\vspace{1em}

	\begin{subfigure}[t]{0.20\textwidth}
		\centering
		\begin{tikzpicture}[scale=0.6, transform shape]
			\node[text opacity=0.5] at (0.6,0) {$+$};
			\node[text opacity=0.5] at (1.2,0) {$+$};
			\node[text opacity=0.5] at (1.8,0) {$+$};
			\node[text opacity=0.5] at (2.4,0) {$+$};
			\node[text opacity=0.5] at (3,0) {$+$};
			\fill (0,0) circle (2pt);
			\fill (3.6,0) circle (2pt);
			\fill (0,-1) circle (2pt);
			\draw (3,-0.85) -- (3,-1.15);
			\fill (0,-2) circle (2pt);
			\draw (2.4,-1.85) -- (2.4,-2.15);
			\fill (0,-3) circle (2pt);
			\draw (1.8,-2.85) -- (1.8,-3.15);
			\fill (0,-4) circle (2pt);
			\draw (1.2,-3.85) -- (1.2,-4.15);
			\fill (0,-5) circle (2pt);
			\draw (0.6,-4.85) -- (0.6,-5.15);
		\end{tikzpicture}
		\caption{}
		\label{fig:f}
	\end{subfigure}
	\hfill
	\begin{subfigure}[t]{0.35\textwidth}
		\centering
		\begin{tikzpicture}[scale=0.6, transform shape]
			\node[text opacity=0.5, scale=0.6] at (2.40,0) {$+$};
			\node[text opacity=0.5, scale=0.6] at (3.00,0) {$+$};
			\node[text opacity=0.5, scale=0.6] at (3.60,0) {$+$};
			\node[text opacity=0.5, scale=0.6] at (4.20,0) {$+$};
			\node[text opacity=0.5, scale=0.6] at (4.80,0) {$+$};
			\node[text opacity=0.5, scale=0.6] at (7.80,0) {$+$};
			\node[text opacity=0.5, scale=0.6] at (8.40,0) {$+$};
			\node[text opacity=0.5, scale=0.6] at (9.00,0) {$+$};
			\fill (0.00,0) circle (2pt);
			\fill (1.80,0) circle (2pt);
			\fill (5.40,0) circle (2pt);
			\fill (7.20,0) circle (2pt);
			\fill (9.60,0) circle (2pt);
			\fill (1.80,-1) circle (2pt);
			\fill (7.20,-1) circle (2pt);
			\fill (1.80,-2) circle (2pt);
			\fill (7.20,-2) circle (2pt);
			\fill (1.80,-3) circle (2pt);
			\fill (7.20,-3) circle (2pt);
			\fill (1.80,-4) circle (2pt);
			\fill (1.80,-5) circle (2pt);
		\end{tikzpicture}
		\caption{}
		\label{fig:g}
	\end{subfigure}
	\hfill
	\begin{subfigure}[t]{0.40\textwidth}
		\centering
		\begin{tikzpicture}[scale=0.6, transform shape]
			\node[text opacity=0.5, scale=0.6] at (2.40,0) {$+$};
			\node[text opacity=0.5, scale=0.6] at (3.00,0) {$+$};
			\node[text opacity=0.5, scale=0.6] at (3.60,0) {$+$};
			\node[text opacity=0.5, scale=0.6] at (4.20,0) {$+$};
			\node[text opacity=0.5, scale=0.6] at (4.80,0) {$+$};
			\node[text opacity=0.5, scale=0.6] at (7.80,0) {$+$};
			\node[text opacity=0.5, scale=0.6] at (8.40,0) {$+$};
			\node[text opacity=0.5, scale=0.6] at (9.00,0) {$+$};
			\fill (0.00,0) circle (2pt);
			\fill (1.80,0) circle (2pt);
			\fill (5.40,0) circle (2pt);
			\fill (7.20,0) circle (2pt);
			\fill (9.60,0) circle (2pt);
			\draw[line width=0.8pt] (0.00,-0.85) -- (0.00,-1.15);
			\fill (1.80,-1) circle (2pt);
			\draw[line width=0.8pt] (4.80,-0.85) -- (4.80,-1.15);
			\draw[line width=0.8pt] (5.40,-0.85) -- (5.40,-1.15);
			\fill (7.20,-1) circle (2pt);
			\draw[line width=0.8pt] (9.00,-0.85) -- (9.00,-1.15);
			\draw[line width=0.8pt] (0.00,-1.85) -- (0.00,-2.15);
			\fill (1.80,-2) circle (2pt);
			\draw[line width=0.8pt] (4.20,-1.85) -- (4.20,-2.15);
			\draw[line width=0.8pt] (4.80,-1.85) -- (4.80,-2.15);
			\draw[line width=0.8pt] (5.40,-1.85) -- (5.40,-2.15);
			\fill (7.20,-2) circle (2pt);
			\draw[line width=0.8pt] (8.40,-1.85) -- (8.40,-2.15);
			\draw[line width=0.8pt] (0.00,-2.85) -- (0.00,-3.15);
			\fill (1.80,-3) circle (2pt);
			\draw[line width=0.8pt] (3.60,-2.85) -- (3.60,-3.15);
			\draw[line width=0.8pt] (4.20,-2.85) -- (4.20,-3.15);
			\draw[line width=0.8pt] (4.80,-2.85) -- (4.80,-3.15);
			\draw[line width=0.8pt] (5.40,-2.85) -- (5.40,-3.15);
			\fill (7.20,-3) circle (2pt);
			\draw[line width=0.8pt] (7.80,-2.85) -- (7.80,-3.15);
			\draw[line width=0.8pt] (0.00,-3.85) -- (0.00,-4.15);
			\fill (1.80,-4) circle (2pt);
			\draw[line width=0.8pt] (3.00,-3.85) -- (3.00,-4.15);
			\draw[line width=0.8pt] (3.60,-3.85) -- (3.60,-4.15);
			\draw[line width=0.8pt] (4.20,-3.85) -- (4.20,-4.15);
			\draw[line width=0.8pt] (4.80,-3.85) -- (4.80,-4.15);
			\draw[line width=0.8pt] (5.40,-3.85) -- (5.40,-4.15);
			\draw[line width=0.8pt] (7.20,-3.85) -- (7.20,-4.15);
			\draw[line width=0.8pt] (0.00,-4.85) -- (0.00,-5.15);
			\fill (1.80,-5) circle (2pt);
			\draw[line width=0.8pt] (2.40,-4.85) -- (2.40,-5.15);
			\draw[line width=0.8pt] (3.00,-4.85) -- (3.00,-5.15);
			\draw[line width=0.8pt] (3.60,-4.85) -- (3.60,-5.15);
			\draw[line width=0.8pt] (4.20,-4.85) -- (4.20,-5.15);
			\draw[line width=0.8pt] (4.80,-4.85) -- (4.80,-5.15);
			\draw[line width=0.8pt] (5.40,-4.85) -- (5.40,-5.15);
		\end{tikzpicture}
		\caption{}
		\label{fig:h}
	\end{subfigure}

	\caption{Schematic representations of the algorithm using Lemma~\ref{lem:ptoshermitelema}.}
	\label{fig:main}
\end{figure}

The following result counts the number of $g$\--distinct roots the highest derivative has.

\begin{thm}\label{th:hermiteroots}
 Take $y_0\in\mathbb R_+$ and $n\in\mathbb N$. Let $\lambda\in\mathbb R^{n+1}$. Consider a sequence of $k+1$ $g$\--distinct points $(x_j)_{j=0}^k\subset \mathbb R_+$ ordered from left to right. Assume the $x_i$ are either roots of type I with multiplicity $m_i\in\mathbb N$ or roots of type II (if this is the case define $m_i=0$) of $\operatorname{L}_{y_0}(\lambda)$ and that
 $
 \#g([x_i,x_{i+1}]^g_+)\geq m_i+2$, $ i\in\{0,\dots,k-1\}$.
 Let $m_l: = \max_{i=0,\dots,k} m_i$. Then, the $g$\--polynomial $\operatorname{L}_{y_0}\operatorname{D}^{m_l}(\lambda)$ has $\sum_{i=0}^{k}(m_i+1)-m_l$ $g$\--distinct roots on $[x_0,x_k]^g_+$.
\end{thm}
\begin{proof}
 We proceed by induction on $m_l$. If $m_l=0$, $(x_j)_{j=0}^k$ is already a sequence of $k+1$ $g$\--distinct roots of $\operatorname{L}_{y_0}(\lambda)$. Let $m_l\in\mathbb N$, with $m_l>0$ and assume the induction hypothesis. We are going to define a sequence of $g$\--distinct points in $[x_0,x_k)^g_+$ in the following way. Let $i\in\{0,\dots,k-1\}$ and consider the points $x_i$ and $x_{i+1}$.

 \noindent $\bullet$ \emph{Case 1. $m_i=0$.} Define $s^i_0\in [x_i,x_{i+1})^g_+$ as the root of $\operatorname{L}_{y_0}\operatorname{D}(\lambda)$ provided by Lemma~\ref{lem:ptoshermitelema} and set $m^i_0 = 0$ and $S_i = \{s^i_0\}$.

 \noindent $\bullet$ \emph{Case 2. $m_i>0$.} Define $s^i_1\in [x_i,x_{i+1})^g_+$ as the root of $\operatorname{L}_{y_0}\operatorname{D}(\lambda)$ given by Lemma~\ref{lem:ptoshermitelema} and set $s^i_0=x_i$, $m^i_0=m_i-1$ and $m^i_1=0$. In this case, $
 \#g([x_i,s^i_1]^g_+) \geq m_i+1$,
 which implies $s^i_0<_g s^i_1$ and $\#g([s^i_0,s^i_1]^g) \geq m^i_0+2$. Notice $s^i_0$ is a zero of multiplicity $m^i_0$ of $\operatorname{L}_{y_0}\operatorname{D}(\lambda)$.
 Set $S_i = \{s^i_0,s^i_1\}$.





 We have $S_i\subset[x_i,x_{i+1})^g_+$, for $i\in\{0,\dots,k-1\}$. Finally, if $m_k>0$ define $s^{k}_0=x_k$ with $m^k_0=m_k-1$ and set $S_{k}=\{s^{k}_0\}$ if not, just set $S_{k}= \emptyset$.

 Define,
 \[ S = \bigcup_{i=0}^k S_i\subset [x_0,x_k]^g_+\quad \text{ and }\quad
 k' = \begin{dcases}
 -1+ \sum_{\substack{i=0\\m_i=0}}^{k-1}1 + \sum_{\substack{i=0\\m_i\neq 0}}^{k-1}2, & m_k = 0,\\
 \sum_{\substack{i=0\\m_i=0}}^{k-1}1 + \sum_{\substack{i=0\\m_i\neq 0}}^{k-1}2, & m_k \neq 0.
 \end{dcases}
 \]
 $S$ is a set of $k'+1$ $g$\--distinct points. Let $(t_i)_{i=0}^{k'}$ be the sequence of the points in $S$ ordered from left to right. Also, if $t_i=s^j_{\varepsilon}$ with $j\in\{0,\dots,k\}$ and $\varepsilon\in\{0,1\}$, define $m'_i = m^j_\varepsilon$.

 We have that $(t_i)_{i=0}^{k'}$ is a sequence of $k'+1$ $g$\--distinct points with multiplicities $m'_i\in\mathbb N$ that satisfies the hypotheses of the theorem for $\operatorname{L}_{y_0}\operatorname{D}(\lambda)$. Apply the  induction hypothesis to get that $\operatorname{L}_{y_0}\operatorname{D}^{m_l-1}(\operatorname{D}\lambda)$ has $1-m_l+\sum_{i=0}^{k'}(m'_i+1)$ $g$\--distinct roots. Now, if $m_k=0$,
 \begin{align*}
 &\sum_{i=0}^{k'}(m'_i+1)-m_l+1 = k'+2-m_l+\sum_{i=0}^{k'}m'_i = k'+2-m_l+\sum_{\substack{i=0\\m_i\neq 0}}^{k-1}(m_i-1)\\
 =&1-m_l +\sum_{\substack{i=0\\m_i=0}}^{k-1}1 + \sum_{\substack{i=0\\m_i\neq 0}}^{k-1}2+\sum_{\substack{i=0\\m_i\neq 0}}^{k-1}(m_i-1)=1-m_l+\sum_{i=0}^{k-1}(m_i+1) = \sum_{i=0}^{k}(m_i+1)-m_l.
 \end{align*}
If $m_k\neq 0$,
\[
 \sum_{i=0}^{k'}(m'_i+1)-m_l+1 = 2-m_l +\sum_{\substack{i=0\\m_i=0}}^{k-1}1 + \sum_{\substack{i=0\\m_i\neq 0}}^{k-1}2+\sum_{\substack{i=0\\m_i\neq 0}}^{k}(m_i-1)=\sum_{i=0}^{k}(m_i+1)-m_l,\]
 concluding the result.
\end{proof}
\begin{rem}
In Theorem~\ref{th:hermiteroots} the hypotheses guarantee that there exist at least
$
1 + \sum_{i=0}^{k-1}(m_i+1)
$
$g$\--distinct points in $[x_0,x_k]^g_+$, which is a greater quantity than $\sum_{i=0}^{k}(m_i+1) - m_l$.
\end{rem}

We are in a position to finally prove that the Hermite* interpolation problem admits a solution.

\begin{thm}[Solvability of the Hermite* interpolation problem]\label{th:hermite*}
Let $k\in\mathbb N$ and consider a sequence of points $(x_i)_{i=0}^k \in\mathbb R_+$ ordered from left to right with associated multiplicities $m_i\in\mathbb N$, $i\in\{0,\dots,k\}$. Fix $n = k+\sum_{j=0}^k m_j$, $X=(x_0,\dots,x_k)$ and $M=(m_0,\dots,m_k)$. Then the following is equivalent:
\vspace{-2ex}\begin{enumerate}[noitemsep]
	\item The Hermite* interpolation problem~\eqref{hermite*} has a unique solution.
	\item $\#g([x_i,x_{i+1}]^g_+)\geq m_i+2$ for $ i\in\{0,\dots,k-1\}$.
\end{enumerate}
\end{thm}
\begin{proof}
I implies II by Proposition~\ref{existenciapuntosher}. Let us show that II implies I. Take some $y_0\in\bR_+$. We finish if we show that $\operatorname{H}_{y_0}(X,M)$ is injective. Let $\lambda\in\mathbb C^{n+1}$ such that $\lambda\in\operatorname{ker}\operatorname{H}_{y_0}(X,M)$, that is,
$\operatorname{L}_{y_0}\operatorname{D}^{j}(\lambda)(x_r) = 0$, $j \in \{0,\dots,m_r\}$, $r \in \{0,\dots,k\}$.
By separating $\lambda$ into its real and imaginary parts we get
$
\operatorname{L}_{y_0}\operatorname{D}^{j}(\operatorname{Re}\lambda)(x_r) =\operatorname{L}_{y_0}\operatorname{D}^{j}(\operatorname{Im}\lambda)(x_r)=0$ for $
 j \in \{0,\dots,m_r\}$ and $r \in \{0,\dots,k\}.
$
Hence, we can assume without loss of generality that $\lambda\in\mathbb R^{n+1}$. Apply Theorem~\ref{th:hermiteroots} to $\operatorname{L}_{y_0}(\lambda)$. Thus, $\operatorname{L}_{y_0}\operatorname{D}^{m_l}(\lambda)$ has $\sum_{i=0}^{k}(m_i+1)-m_l$ $g$\--distinct roots where $m_l = \max_{i=0}^k m_i$.

 However, $\operatorname{L}_{y_0}\operatorname{D}^{m_l}(\lambda)\in\operatorname{P}_{n-m_l}$ and
$
n-m_l = \sum_{j=0}^{k}(m_j+1)-(m_l+1)$.
Therefore, $\operatorname{L}_{y_0}\operatorname{D}^{m_l}(\lambda)$ has $n-m_l+1$ $g$\--distinct roots and, by Theorem~\ref{rootsofpol}, $\operatorname{D}^{m_l}(\lambda)=0$.
Equivalently, $\lambda_j=0$ for $m_l\leq j\leq n$. By centering the polynomials at $x_l$, $
\operatorname{L}_{x_l}\operatorname{D}^{j}\operatorname{M}_{y_0\to x_l}(\lambda)(x_l) = 0$ for $j \in \{0,\dots,m_l\}$.
Denote as $\beta = \operatorname{M}_{y_0\to x_l}(\lambda) \in\mathbb R^{n+1}$. By the upper triangularity of ${M}_{x_l\to y_0}(\lambda)$, $\beta_j = 0$ for $m_l\leq j\leq n$. Since $\operatorname{L}_{x_l}\operatorname{D}^{j}(\beta)(x_l) = \beta_jj!$,
we obtain that $\beta_j=0$ for $0\leq j\leq m_l$. Thus, $\beta=0$ which implies $\lambda=0$, so $\operatorname{H}_{y_0}(X,M)$ is injective.
\end{proof}

We now present an example that demonstrates how to solve the Hermite interpolation problem in a particular case.

\begin{exa}
	Let $x_0,x_1\in \bR$ such that $x_0<x_1$. We consider the following Hermite interpolation problem
	\begin{equation}
		\label{eq:exahermite}
	\begin{array}{l}
	p^{(j)}_g(x_0) = t_{j}, \quad j \in \{0,1,2\},\\
	p^{(j)}_g(x_1) = c_{j}, \quad j \in \{0,1\},
	\end{array}
	\end{equation}
	for some $t_j,c_j\in\bF$. In this case, $m_0=2$, $m_1=1$ and $n=4$. This issue arises in the differentiable extension of functions. Indeed, assume we have some function $f\colon (-\infty,x_0]\cup [x_1,+\infty)\to \bF$ that is $g$\--smooth and we wish to extend the domain of $f$ to the whole real line $\bR$ with some degree of smoothness. 
	In this case, take
	\[
	t_{j} = f_g^{(j)}(x_0), \; j\in\{0,1,2\} \text{ and }c_{j} = f_g^{(j)}(x_1), \; j\in\{0,1\}.
	\]
	Hence, for a solution $p\in\operatorname{P}_4$ of \eqref{eq:exahermite}, the function
	\[
	\tilde f\colon x\in\bR\to \begin{dcases}
		f(x), & x\in (-\infty,x_0]\cup [x_1,+\infty),\\
		p(x), & x\in (x_0,x_1),
	\end{dcases}
	\] is  $g$\--smooth on $\bR\setminus\{x_0,x_1\}$ and, on $x_0$ and $x_1$, $\tilde f$ is two and one time  $g$\--differentiable respectively.
	Problem~\eqref{eq:exahermite} has an associated Hermite* interpolation problem. That is, fixing $y_0\in \bR_+$, we try to find $\lambda\in\bF^{5}$ such that
	\begin{equation}
		\label{eq:exahermite*}
	\begin{array}{l}
	\operatorname{L}_{y_0}\operatorname{D}^j(\lambda)(x_0) = t_{j}, \quad j \in \{0,1,2\},\\
	\operatorname{L}_{y_0}\operatorname{D}^j(\lambda)(x_1) = c_{j}, \quad j \in \{0,1\}.
	\end{array}
	\end{equation}
	Define $X=(x_0,x_1)$ and $M=(m_0,m_1)$. Hence, we can write~\eqref{eq:exahermite*} as a linear equation using the operator $\operatorname{H}_{y_0}(X,M)$. Thus, $\lambda\in\bF^{5}$ solves~\eqref{eq:exahermite*} if and only if
	\[
	\operatorname{H}_{y_0}(X,M)(\lambda)=\begin{pmatrix}
		t_0\\
		t_1\\
		t_2\\
		c_0\\
		c_1
	\end{pmatrix}.
	\]
	Writing the above as a matrix formulation we get
	\[
	\begin{pmatrix}
	1 & g_{y_0,1}(x_0) & g_{y_0,2}(x_0) & g_{y_0,3}(x_0) & g_{y_0,4}(x_0) \\
	0 & 1 & 2g_{y_0,1}(x_0) & 3 g_{y_0,2}(x_0) & 4g_{y_0,3}(x_0) \\
	0 & 0 & 2 & 6 g_{y_0,1}(x_0) & 12 g_{y_0,2}(x_0) \\
	1 & g_{y_0,1}(x_1) & g_{y_0,2}(x_1) & g_{y_0,3}(x_1) & g_{y_0,4}(x_1)\\
	0 & 1 & 2g_{y_0,1}(x_1) & 3 g_{y_0,2}(x_1) & 4g_{y_0,3}(x_1)
	\end{pmatrix}\begin{pmatrix}
		\lambda_0\\
		\lambda_1\\
		\lambda_2\\
		\lambda_3\\
		\lambda_4
	\end{pmatrix}=\begin{pmatrix}
		t_0\\
		t_1\\
		t_2\\
		c_0\\
		c_1
	\end{pmatrix}.
	\]
	If we take $y_0=x_0$ the above gives
	\[
	\begin{pmatrix}
	1 & 0 & 0 & 0 & 0 \\
	0 & 1 & 0 & 0 & 0 \\
	0 & 0 & 2 & 0 & 0 \\
	1 & g_{x_0,1}(x_1) & g_{x_0,2}(x_1) & g_{x_0,3}(x_1) & g_{x_0,4}(x_1)\\
	0 & 1 & 2g_{x_0,1}(x_1) & 3 g_{x_0,2}(x_1) & 4g_{x_0,3}(x_1)
	\end{pmatrix}\begin{pmatrix}
		\lambda_0\\
		\lambda_1\\
		\lambda_2\\
		\lambda_3\\
		\lambda_4
	\end{pmatrix}=\begin{pmatrix}
		t_0\\
		t_1\\
		t_2\\
		c_0\\
		c_1
	\end{pmatrix},
	\]
	which has unique solution if and only if
	\[
	\begin{vmatrix}
	g_{x_0,3}(x_1) & g_{x_0,4}(x_1)\\
	3g_{x_0,2}(x_1) & 4g_{x_0,3}(x_1)
	\end{vmatrix} = 4g_{x_0,3}(x_1)g_{x_0,3}(x_1)-3g_{x_0,2}(x_1)g_{x_0,4}(x_1)\neq 0.
	\]
	Theorem~\ref{th:hermite*} guarantees the above determinant is nonzero if and only if $\#g([x_0,x_1]^g_+)\geq 4$. If this is the case, then we can just compute $\lambda$ inverting the matrix $[\operatorname{H}_{y_0}(X,M)]$ numerically and solving for $\l$:
	\[
	\lambda= [\operatorname{H}_{y_0}(X,M)]^{-1}\begin{pmatrix}
		t_0\\
		t_1\\
		t_2\\
		c_0\\
		c_1
	\end{pmatrix}.
	\]
	Therefore, by taking $p=\operatorname{L}_{y_0}(\lambda)$ we have a potential solution for~\eqref{eq:exahermite}. In fact, $p$ solves~\eqref{eq:exahermite} if and only if $g$ is a derivator for which the derivatives $p_g^{(j)}(x_0), \; j\in\{0,1,2\}$, $p_g^{(j)}(x_1), \; j\in\{0,1\}$ are well defined.
	Let us consider some specific derivators.

	\noindent\-- \emph{Case 1.} Let $g$ be the derivator
	\begin{equation}
		\label{eq:exadefgcase1}
	g(x) = \begin{dcases}
	-1, & x\leq -\frac{1}{2},\\
	-\frac{1}{2}, & x\in\left(-\frac{1}{2},-\frac{1}{4}\right],\\
	0, & x\in\left(-\frac{1}{4},0\right],\\
	\frac{1}{2}, & x>0,\\
	\end{dcases}
	\end{equation}
	\begin{figure}[h!]
	\centering
	\includegraphics[scale=0.25]{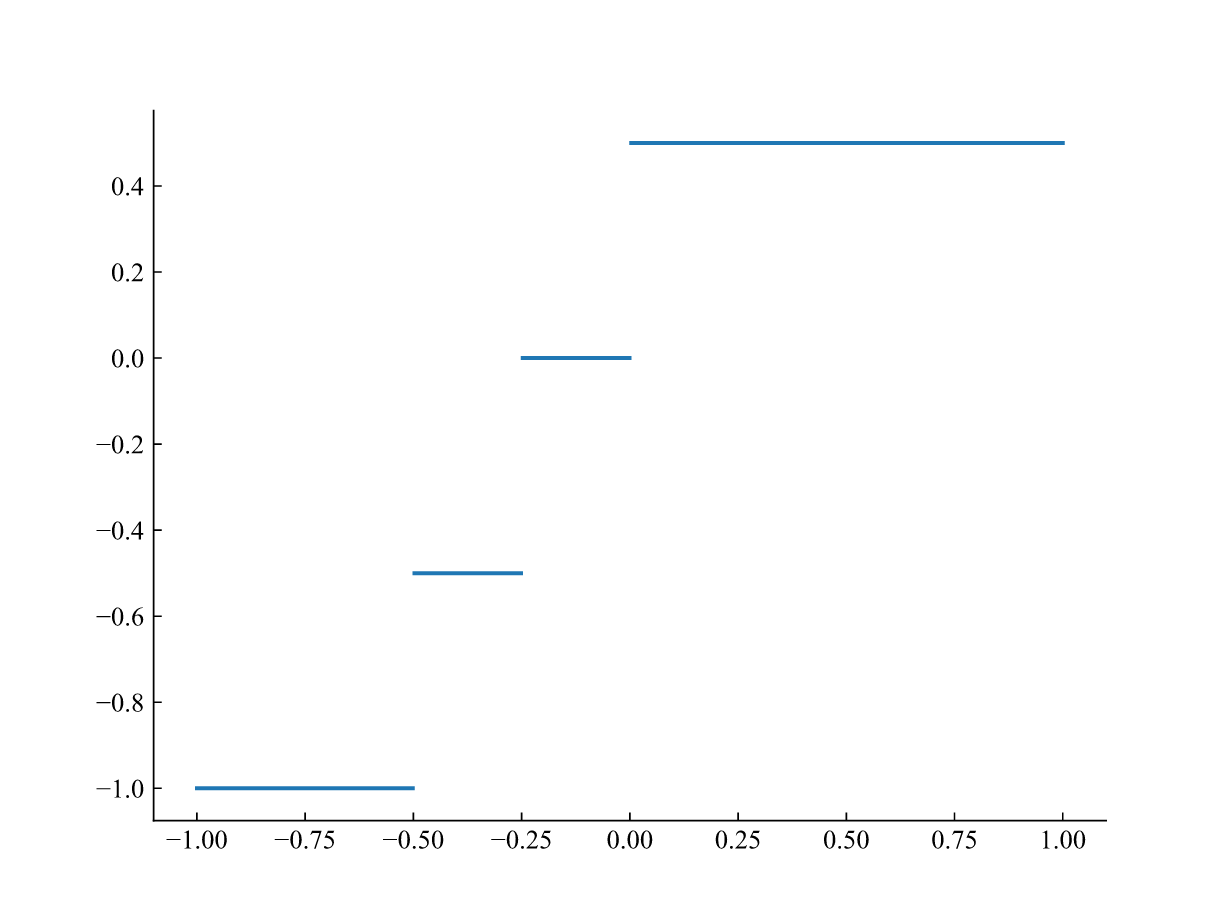}
	\caption{Graph of $g$ in  \eqref{eq:exadefgcase1}.}\label{fig:gcase1}
	\end{figure}
see Figure~\ref{fig:gcase1}. Let also $x_0 = -\tfrac{3}{4}$ and $x_1 = \tfrac{1}{2}$. Since $\# g([x_0,x_1]^g_+) = 4$, the Hermite* interpolation problem admits a unique solution. Observe that $\# g(\mathbb{R}) = 4$, and thus $\dim \operatorname{P}_4 = 4$, whereas problem~\eqref{eq:exahermite*} consists of five equations; nevertheless, the matrix $[\operatorname{H}_{y_0}(X,M)]$ is invertible.

Furthermore, derivatives of higher order are not well defined at $x_1$, so the Hermite interpolation problem~\eqref{eq:exahermite} is not applicable in this setting. In contrast, the Hermite* problem still admits a unique solution, as in Example~\ref{examplegconstanthermite}. See Figure~\ref{fig:case1interpol} for a numerically computed solution.
\begin{figure}[h!]
	\centering
	\includegraphics[scale=0.25]{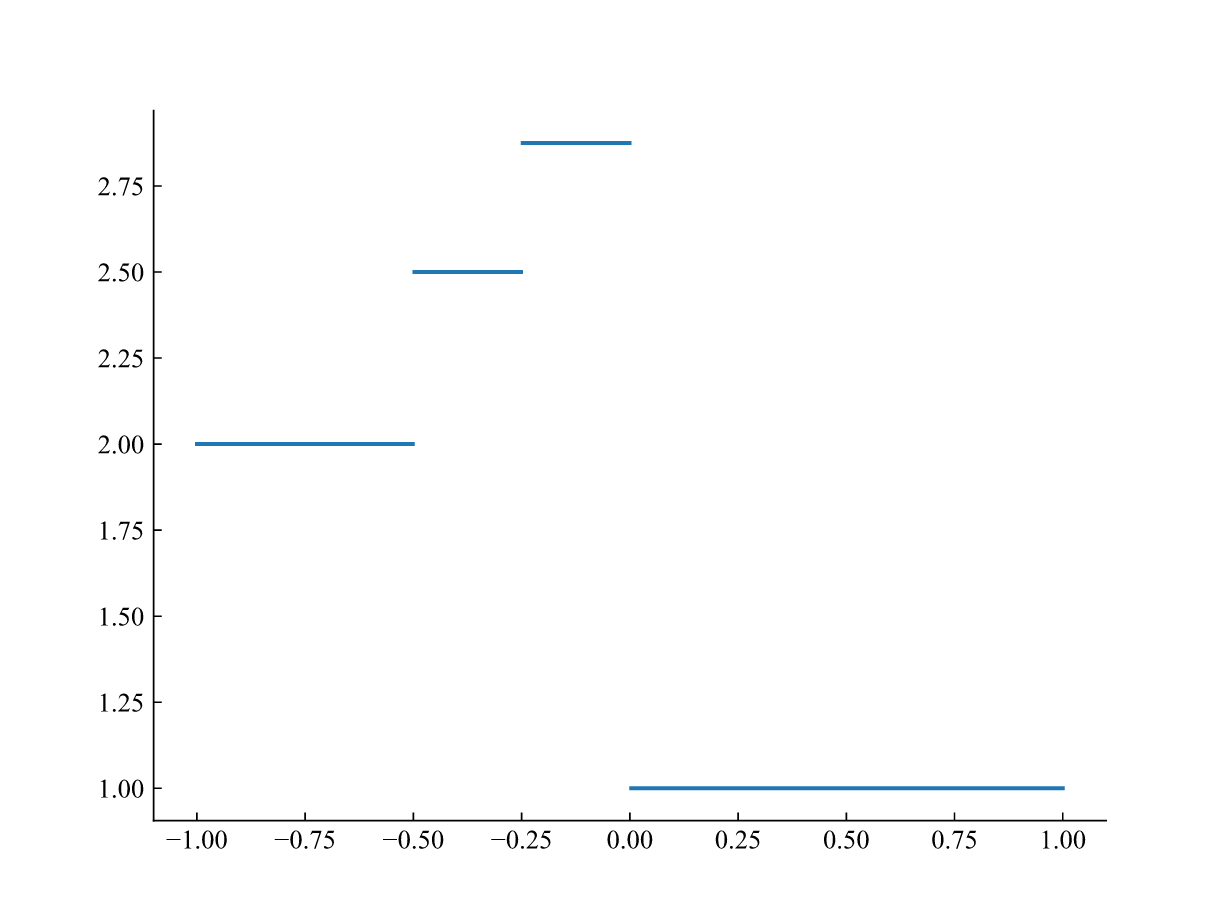}
	\caption{Graph of $\operatorname{L}_{y_0}(\lambda)$ where $\lambda$ solves~\eqref{eq:exahermite*} for $t_0=2$, $t_1=1$, $t_2=-1/2$, $c_0=1$, $c_1=1$ and $g$ as in  \eqref{eq:exadefgcase1}}.\label{fig:case1interpol}
	\end{figure}

\noindent\-- \emph{Case 2.} Let $g:[-1,1]\to\bR$ be the derivator
	\begin{equation}
		\label{eq:exadefgcase2}
	g(x) = \begin{dcases}
	\operatorname{sin}(x), & x\leq 0,\\
	\operatorname{sin}(x)+1, & x>0,
	\end{dcases}
	\end{equation}
	\begin{figure}[h!]
	\centering
	\includegraphics[scale=0.25]{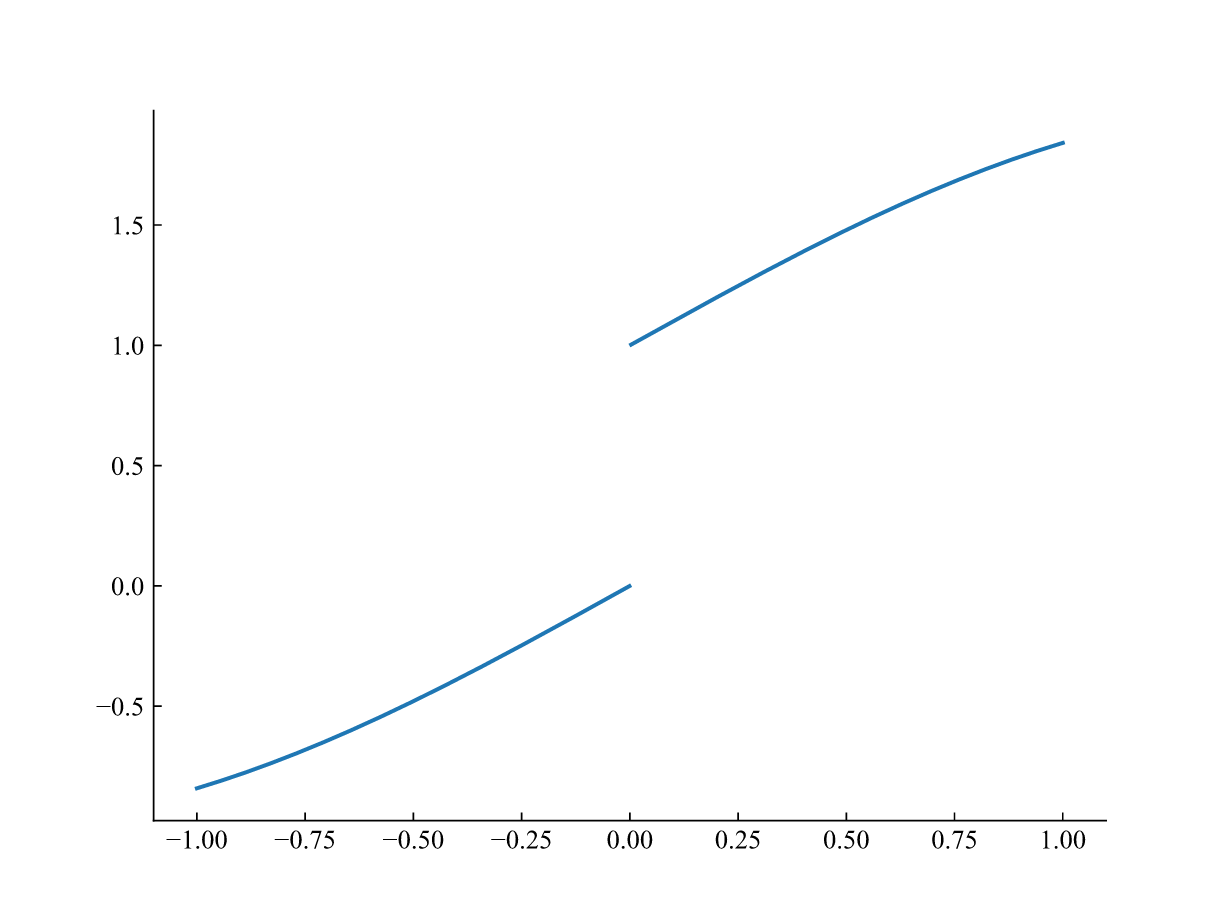}
	\caption{Graph of $g$ in  \eqref{eq:exadefgcase2}.}\label{fig:gcase2}
	\end{figure}
	see Figure~\ref{fig:gcase2}. Take $x_0 = -1/2$ and $x_1 = 1/2$. In this case, $\#g([x_0,x_1]^g_+) = \infty$, so the Hermite* interpolation problem admits a unique solution. Also, higher--order Stieltjes derivatives are well defined so problem~\eqref{eq:exahermite} and~\eqref{eq:exahermite*} are equivalent. In this case, $\operatorname{dim}\operatorname{P}_4=5$. Check Figure~\ref{fig:case2interpol} for a numerical solution.
	\begin{figure}[h!]
	\centering
	\includegraphics[scale=0.25]{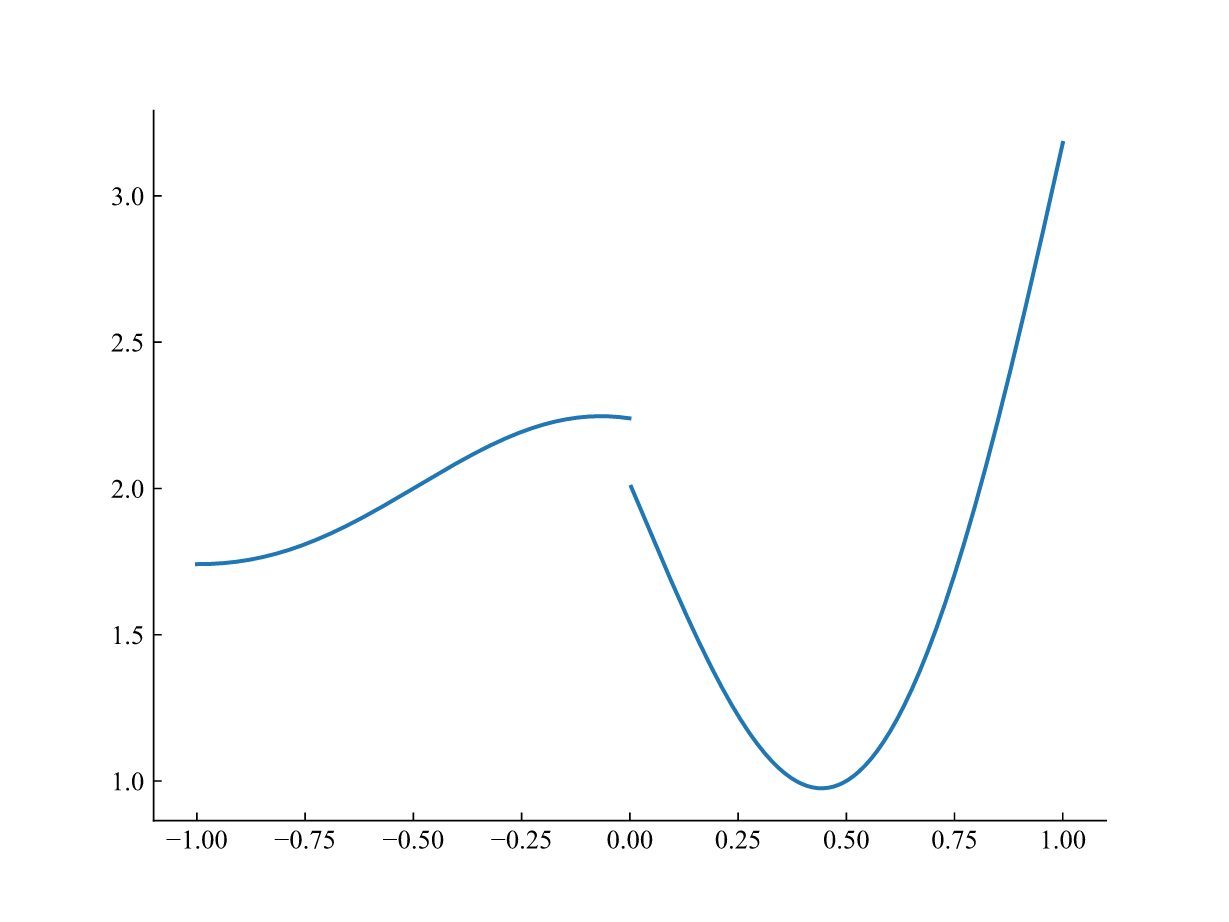}
	\caption{Graph of $\operatorname{L}_{y_0}(\lambda)$ where $\lambda$ solves~\eqref{eq:exahermite*} for $t_0=2$, $t_1=1$, $t_2=-1/2$, $c_0=1$, $c_1=1$ and $g$ as in  \eqref{eq:exadefgcase2}}.\label{fig:case2interpol}
	\end{figure}
\end{exa}
\subsection{Convex Hermite* interpolation}
We can generalize the Hermite interpolation problem like we did on Section~\ref{subsec:cli} with the Lagrange interpolation problem. However, in this case, we do not allow roots of type II to have multiplicity greater than zero.

\begin{dfn}[Convex Hermite* interpolation problem] Consider a valid sequence for interpolation $x_0,\dots,x_{k}$ in $\mathbb R_{++}$ with associated multiplicities $m_i\in\mathbb N$, $i\in\{0,\dots,k\}$ where $m_i=0$ if $x_i\in \mathbb R_{++}\setminus \mathbb R_+$. Let $n: = k+\sum_{j=0}^k m_j$, $y_0\in\mathbb R_+$ and $t_{i,j} \in \mathbb{F}$. We call the system of equations
	\begin{equation}\label{eq:ch*}
		F(x_i)(\operatorname{L}_{y_0}\operatorname{D}^{j}(\lambda)) = t_{i,j},\quad i\in\{0,\dots,k\}, \quad j\in\{0,\dots,m_i\},
	\end{equation}
	the \emph{convex Hermite* interpolation problem}, for which we try to find a \emph{solution} $\lambda\in \mathbb{F}^{n+1}$.
\end{dfn}


The Convex Hermite* interpolation problem is a generalization of the Hermite* interpolation problem~\eqref{hermite*}. The linear forms
\[
\begin{tikzcd}[row sep=0ex]
 \mathbb F^{n+1} \arrow[r,"\operatorname{D}^{j}"] & \mathbb F^{n+1} \arrow[r,"\operatorname{L}_{y_0}"]& \operatorname{P}_n \arrow[r,"F(x_i)"] & \mathbb F\\
 \lambda\arrow[rrr,mapsto] &&& F(x_i)(\operatorname{L}_{y_0}\operatorname{D}^{j}(\lambda))
\end{tikzcd}
\]
with $i\in\{0,\dots,k\}$ and $j\in\{0,\dots,m_i\}$ define the linear endomorphism $\operatorname{H}^*_{y_0}(X,M)\colon\mathbb F^{n+1}\to \mathbb F^{n+1}$ called the \emph{convex Hermite* mapping}, where $X=(x_0,\dots,x_k)$ and $M=(m_0,\dots,m_k)$, analogous to Definition~\ref{def:hermiteoperator}.

\begin{lem}
 For any $y_0,y_0'\in\mathbb R_+$,
 $\operatorname{H}^*_{y_0}(X,M) = \operatorname{H}^*_{y_0'}(X,M)\operatorname{M}_{y_0'\to y_0} $. Also,
 \[|[\operatorname{H}^*_{y_0}(X,M)]| = |[\operatorname{H}^*_{y_0'}(X,M)]|.\]
\end{lem}
\begin{proof}
	The proof is the same as Lemma~\ref{lem:hermitechangeofcenter}.
\end{proof}

We can show that the Convex Hermite* interpolation problem is solvable under sufficient and necessary conditions analogous to the Hermite* interpolation problem.

\begin{thm}[Solvability of the Convex Hermite* interpolation problem]\label{th:chp}
 Let $k\in\bN$. Consider a valid sequence for interpolation $x_0,\dots,x_{k}$ in $\mathbb R_{++}$, ordered from left to right, with associated multiplicities $m_i\in\mathbb N$, $i\in\{0,\dots,k\}$ where $m_i=0$ if $x_i\in \mathbb R_{++}\setminus \mathbb R_+$. Fix $n = k+\sum_{j=0}^k m_j$. Then, the following is equivalent:
 \vspace{-2.5ex}\begin{enumerate}[noitemsep]
	\item The Convex Hermite* interpolation problem~\eqref{eq:ch*} admits a unique solution.
	\item For $i\in\{0,\dots,k-1\}$ such that $x_i\in\mathbb R_+$, $\#g([x_i,x_{i+1}]^g_+)\geq m_i+2$ if $x_{i+1}\in\mathbb R_+$ and $\#g([x_i,t]^g_+)\geq m_i+1$ if $x_{i+1}\in\mathbb R_{++}\setminus\mathbb R_{+}$,
 where $x_{i+1} = (\alpha,t)\in(0,1)\times D_g$.
 \end{enumerate}
\end{thm}
\begin{proof}
 Denote $X=\{x_i:i\in\{0,\dots,k\}\}$. By Remark~\ref{rem:ipstar}, we can assume, w.l.o.g., that, if $\#(g^*(X)\cap[g(t),g(t^+)])=2$, then $g^*(X)\cap[g(t),g(t^+)] = \{g(t),g(t^+)\}$. Take $y_0\in\bR_+$. We show I implies II. Let $i\in\{0,\dots,k-1\}$ such that $x_i\in\mathbb R_+$. We consider two cases, although both are proven analogously.

 \noindent\-- \emph{Case 1}. $x_{i+1}\in \bR_+$. Assume $\#g([x_i,x_{i+1}]^g_+)\leq m_i+1$. Since I holds, there must exist $\lambda\in \bF^{n+1}$ such that
 \[
 \operatorname{L}_{y_0}\operatorname{D}^j(\lambda)(x_i) = 0,\text{ for }j\in\{0,\dots,m_i\} \text{ and }\operatorname{L}_{y_0}(\lambda)(x_{i+1})=1.
 \]
 However, the existence of this $\lambda$ contradicts Proposition~\ref{pro:setCvanishder}. Indeed, since $x_i$ is a zero of multiplicity $m_i$ for $\lambda$ and $y_0$, Proposition~\ref{pro:setCvanishder} ensures $\operatorname{L}_{y_0}(\lambda)(x_{i+1})=0$.

 \noindent\-- \emph{Case 2}. $x_{i+1}\in \bR_{++}\setminus \bR_{+}$. Assume $\#g([x_i,t]^g_+)\leq m_i$ where $x_{i+1} = (\alpha,t)\in(0,1)\times D_g$. Note this implies that $\#g([x_i,t^+]^g_+)\leq m_i+1$. Since I holds, there exists $\lambda\in \bF^{n+1}$ such that
\[
 \operatorname{L}_{y_0}\operatorname{D}^j(\lambda)(x_i) = 0,\text{ for }j\in\{0,\dots,m_i\} \text{ and }F(x_{i+1})(\operatorname{L}_{y_0}(\lambda))=1.
 \]
Analogously to the previous case, Proposition~\ref{pro:setCvanishder} shows that $\operatorname{L}_{y_0}(\lambda)(t)=\operatorname{L}_{y_0}(\lambda)(t^+)=0$. But that implies $F(x_{i+1})(\operatorname{L}_{y_0})=0$ which is a contradiction.

We prove II implies I. As in the proof of Theorem~\ref{th:hermite*}, it is enough to show that if $\lambda\in\operatorname{ker}\operatorname{H}^*_{y_0}(X,M)$ with $\lambda\in \bR^{n+1}$, then $\lambda=0$. We are going to define a sequence of roots of $\operatorname{L}_{y_0}(\lambda)$, $(s_i)_{i=0}^k\ss \bR_+$, with associated multiplicities $m_i'\in\bN$ that will satisfy the hypotheses of Theorem~\ref{th:hermiteroots}.

Let us construct this sequence. Take $i\in\{0,\dots,k\}$. If $x_i\in\mathbb R_+$ define $s_i=x_i$ and $m_i'=m_i$. Suppose $x_i\in\mathbb R_{++}\setminus\bR_{+}$ so $x_i = (\alpha_i,t_i)\in(0,1)\times D_g$. Since $
F(x_i)(\operatorname{L}_{y_0}(\lambda))=0$,
$\operatorname{L}_{y_0}(\lambda)$ either has a root of type II at $t$ or a root of type I at $t$ and a root of type I at $t^+$. Suppose that $i\geq 1$, $m_{i-1}>0$ (note this implies $x_{i-1}\in\bR_+$) and $\#g([x_{i-1},t]^g_+) = m_{i-1}+1$. By Proposition~\ref{pro:setCvanishder}, $\operatorname{L}_{y_0}(\lambda)(t)=0$. This, combined with $F(x_i)(\operatorname{L}_{y_0}(\lambda))=0$, gives $\operatorname{L}_{y_0}(\lambda)(t^+)=0$ with $\#g([x_{i-1},t^+]^g_+) = m_{i-1}+2$. Thus, if $i\geq 1$ with $m_{i-1}>0$ and $\#g([x_{i-1},t]^g_+) = m_{i-1}+1$ set $s_i=t^+$ and $m_i'=0$. If not, set $m_i'=0$ and $s_i=t$, since always $\operatorname{L}_{y_0}(\lambda)$ has a root of type II at $t$ or a root of type I at $t$.

Proceeding like this we always assure that (if $i\geq 1$ and $x_{i-1}\in\mathbb R_+$), $\# g([x_{i-1},s_i]^g_+)\geq m_{i-1}+2$ and, as a consequence, $\# g([s_{i-1},s_i]^g_+)\geq m_{i-1}'+2$.

Hence, $(s_i)_{i=0}^k\subset \mathbb R_{+}$ is a sequence of roots of $\operatorname{L}_{y_0}(\lambda)$ that satisfy the hypotheses of Theorem~\ref{th:hermiteroots}. Therefore, $\operatorname{L}_{y_0}\operatorname{D}^{m_l}(\lambda)$ has $\sum_{i=0}^{k}(m_i+1)-m_l$ $g$\--distinct roots. The rest of the proof follows the proof of Theorem~\ref{th:hermite*}.
\end{proof}

\bibliographystyle{spmpsciper}
\bibliography{rootinter}

\end{document}